\documentclass[a4paper, 12pt]{amsart}
\usepackage[utf8]{inputenc}

\usepackage{url}
\usepackage[margin=1in]{geometry}  
\usepackage{epsfig}
\usepackage{color}
\usepackage{graphicx}              
\usepackage{amsmath}
\usepackage{amscd}               
\usepackage{amsfonts}
\usepackage{amssymb}             
\usepackage{amsthm}                
\usepackage{epsfig}
\usepackage{mathrsfs}
\usepackage{hyperref}
\usepackage{enumerate}
\usepackage{epstopdf}

\numberwithin{equation}{section}

\newtheorem{theorem}{Theorem}

\newtheorem*{theorem*}{Theorem}
\newtheorem{lemma}[theorem]{Lemma}
\newtheorem*{lemma*}{Lemma}
\newtheorem{proposition}[theorem]{Proposition}
\newtheorem*{proposition*}{Proposition}
\theoremstyle{definition}

\theoremstyle{remark}

\newtheorem*{remark}{Remark}
\newtheorem*{rem}{Remark}

\newcommand{\RE}{\operatorname{Re}}
\newcommand{\IM}{\operatorname{Im}}
\newcommand{\meas}{\operatorname{meas}}

\newcommand{\GCD}{\operatorname{GCD}}
\newcommand{\good}{\operatorname{good}}
\newcommand{\bad}{\operatorname{bad}}

\newcommand{\sm}{\operatorname{sm}}

\newcommand{\mix}{\operatorname{mix}}
\newcommand{\rel}{\operatorname{rel}}
\newcommand{\TV}{\operatorname{TV}}
\newcommand{\gap}{\operatorname{gap}}
\newcommand{\SL}{\operatorname{SL}}
\newcommand{\vol}{\operatorname{vol}}

\newcommand{\GL}{\operatorname{GL}}
\newcommand{\spec}{\operatorname{spec}}

\newcommand{\sav}{\operatorname{sav}}

\newcommand{\old}{\operatorname{old}}
\newcommand{\ini}{\operatorname{init}}
\newcommand{\fin}{\operatorname{fin}}
\newcommand{\diam}{\operatorname{diam}}
\newcommand{\geom}{\operatorname{geom}}

\newcommand{\rad}{\operatorname{rad}}
\newcommand{\new}{\operatorname{new}}

\newcommand{\E}{\mathbf{E}}     
\newcommand{\Prob}{\mathbf{P}} 
\newcommand{\one}{\mathbf{1}}
\newcommand{\Var}{\mathbf{Var}}

\newcommand{\bR}{\mathbb{R}}

\newcommand{\bU}{\mathbb{U}}

\newcommand{\bF}{\mathbb{F}}
\newcommand{\bS}{\mathbb{S}}
\newcommand{\zed}{\mathbb{Z}}

\newcommand{\sA}{{\mathscr{A}}}
\newcommand{\sB}{{\mathscr{B}}}
\newcommand{\sC}{{\mathscr{C}}}
\newcommand{\sD}{{\mathscr{D}}}

\newcommand{\sF}{{\mathscr{F}}}

\newcommand{\sI}{{\mathscr{I}}}
\newcommand{\sJ}{{\mathscr{J}}}
\newcommand{\sK}{{\mathscr{K}}}
\newcommand{\sL}{{\mathscr{L}}}
\newcommand{\sM}{{\mathscr{M}}}
\newcommand{\sN}{{\mathscr{N}}}

\newcommand{\sP}{{\mathscr{P}}}

\newcommand{\sS}{{\mathscr{S}}}
\newcommand{\sT}{{\mathscr{T}}}

\newcommand{\fX}{\mathfrak{X}}
\newcommand{\fS}{\mathfrak{S}}

\newcommand{\vo}{{\underline{\theta}}}

\newcommand{\ualpha}{\underline{\alpha}}
\newcommand{\ua}{\underline{a}}
\newcommand{\uA}{\underline{A}}

\newcommand{\un}{\underline{n}}

\newcommand{\uv}{{\underline{v}}}
\newcommand{\uw}{\underline{w}}
\newcommand{\ux}{\underline{x}}
\newcommand{\uy}{\underline{y}}

\title[Cycle walks]{Mixing and cut-off in cycle walks}
\author{Robert Hough}
\address[Robert Hough]{School of Mathematics, Institute of Advanced Study, 1
Einstein Drive, Princeton, NJ, 08540, USA}
\email{hough@math.ias.edu}

\thanks{This material is based upon work supported by the National
Science Foundation under agreements No. DMS-1128155 and DMS-1712682.  Any opinions, findings and
conclusions or recommendations expressed in this material are those of the
author and do not necessarily reflect the views of the National Science
Foundation.}
\thanks{The author thanks Persi Diaconis, Yuval Peres, David Wilson,
Lionel Levine, Elon Lindenstrauss and Daniel Jerison for useful conversations.}

\begin{document}

\begin{abstract}
Given a sequence $(\fX_i, \sK_i)_{i=1}^\infty$ of Markov chains, the cut-off
phenomenon describes a period of transition to stationarity which is
asymptotically lower order than the mixing time. We study mixing times and the
cut-off
phenomenon in the total variation metric in the case of random walk on the
groups $\zed/p\zed$, $p$ prime, with driving measure uniform on a symmetric
generating set $A \subset \zed/p\zed$.

\end{abstract}

\maketitle

\section{Introduction}
The mixing analysis of random walk on a finite abelian group is a classical
problem of probability
theory, with widespread applications; the
Ehrnfest urn and sandpile models of statistical mechanics are motivating
examples \cite{DGM90, JLP15, S15}.    Among the early results in this area is
a theorem of
Greenhalgh \cite{G89}, which shows that for generating set of
size $k$ contained in $\zed/n\zed$, the mixing time of the corresponding
random walk satisfies $t^{\mix} \gg_k n^{\frac{2}{k-1}}$.  A
set of size $k$ with mixing time bounded by $\ll_k n^{\frac{2}{k-1}}\log n$ is
also exhibited. Dou, Hildebrand and Wilson \cite{DH96}, \cite{H94}, \cite{W97}
consider the mixing of measures driven by typical generating sets on cyclic and
more general groups. Among the results of \cite{H94} is that typical generating
sets of size $k = (\log n)^a$, $a > 1$ produce a random walk satisfying the
cut-off phenomenon. We confine our attention to cyclic groups and symmetric
generating sets
which are smaller than logarithmic size in the order of the group, and prove
a number of refined results on the mixing behavior. Our results are in a
similar spirit to those of Diaconis and Saloff-Coste \cite{DS94} proven in the
more general context of random walk on groups of polynomial growth, but in
narrowing our focus we emphasize strong uniformity in the number of generators
of the random walk. Note that in the context of random walk on nilpotent groups, the mixing of the walk projected to the abelianization often controls the mixing in the group as a whole, see \cite{GT12}, \cite{DH15}.

To briefly summarize the results,
  Theorem
\ref{spectral_mixing_time_theorem} gives spectral upper and lower bounds for the
mixing time in a sharper form than previous results which have appeared in the
literature.
A natural
conjecture regarding random walk on a connected graph is that the total
variation mixing time is bounded by the maximum degree times the diameter
squared.
A highlight of our work is
Theorem \ref{mixing_time_theorem}, which  verifies the conjecture for the
mixing time
of random walk on the Cayley graph of $\zed/p\zed$
with a small symmetric generating set. Theorem \ref{bounded_gen_set_theorem}
gives a lower bound for the period of transition to uniformity relative to the
mixing time -- a lower bound on the cut-off window. Theorem \ref{random_theorem}
determines the generic and worst case mixing behavior for a sequence of typical
symmetric random walks. We conclude by analyzing the mixing time of a
walk which may be considered an approximate embedding of the hypercube
$(\zed/2\zed)^d$
into the cycle, demonstrating a cut-off phenomenon.

\subsection{Precise statement of results}
Let $\sP$ be the set of primes. Given $p \in \sP$ let $A \subset
\zed/p\zed$ be symmetric ($x \in A$ if and only if $-x \in A$), lazy ($0
\in A$) and generating $(|A| > 1)$.
Write $\sA(p)$ be the collection of symmetric, lazy, generating subsets of
$\zed/p\zed$, and for $k \in
\zed_{>0}$ write $\sA(p,k) \subset \sA(p)$ be those sets of size
$2k+1$. Given $A \in \sA(p)$
let $\mu_{A}$ denote the uniform measure on $A$, \[\mu_{A} =
\frac{1}{|A|} \sum_{x \in A} \delta_x.\] The distribution at step $n \geq 1$ of
random walk driven by $\mu_{A}$  is given by the convolution power
\[
\mu_{A}^{*1} = \mu_A, \qquad \mu_A^{*n} = \mu_A^{*(n-1)} * \mu_A, \; n > 1.
\]
As $n \to  \infty$, $\mu_A^{*n}$ converges to the uniform measure
$\bU_{\zed/p\zed}$ on $\zed/p\zed$ and we consider asymptotic behavior of this
convergence
for large $p$.  In particular,  the behavior of these walks
as $k = k(p)$ varies as a function of $p$, and as $A$ varies in the set
$\sA(p,k)$ is studied.

Given measure space $(\fX, \sB)$, a norm $\|\cdot\|$ on
the space $\sM(\fX)$ of probability measures on $\fX$,  a Markov chain
$P^n(\cdot)$ with stationary measure $\nu \in \sM(\fX)$, and $0 < \epsilon < 1$,
define the $\epsilon$-mixing
time
\[
 t^{\mix}(\epsilon)= \inf\left\{n: \sup_{\mu \in \sM(\fX)} \| P^n(\mu) - \nu\|
\leq
\epsilon\right\}
\]
and the standard mixing time $t^{\mix} = t^{\mix}\left(\frac{1}{e}\right)$.
In the cases considered $\fX$ is a (finite, compact, locally compact)
abelian group, and, due to the symmetry of the walk, it is sufficient to
take for $\mu$ the point mass at 0.
Of primary interest is the total variation norm, which for
$\mu, \nu \in \sM(\fX)$ is given by
\[
 \|\mu - \nu\|_{\TV(\fX)} = \sup_{S \in \sB} |\mu(S) - \nu(S)| .
\]
The mixing time with respect to this norm is indicated
$t^{\mix}_1$. Two further important parameters in
considering reversible Markov chains are the spectral gap of the transition
kernel
\[
 \gap = 1 - \sup\left\{|\lambda|: \lambda \in \spec(P) \setminus \{\pm 1\}
\right\}
\]
and the relaxation time
\[
t^{\rel} = \frac{1}{-\log (1-\gap)} \approx \frac{1}{\gap}.
\]
 In
stating our results we let $\tau_0$ 
denote the ratio $\frac{t^{\mix}_1}{t^{\rel}}$ of the
one dimensional Gaussian diffusion
\begin{equation}\label{def_1_d_diffusion}
\theta(x,t) =
\sum_{j \in \zed} \exp(-2\pi^2 tj^2)e^{2\pi i jx}
\end{equation}
on $(\bR/\zed, dx)$; $2\pi^2 t =  \tau_0$ solves the equation
\[
 \int_0^1 |\theta(x,t)-1|dt = \frac{2}{e}
\]
and has numerical value\footnote{We use parentheses to indicate the last significant digit of numerical constants.}
\begin{equation}\label{def_tau_0}
 \tau_0 = 0.56161265(1).
\end{equation}

In the context
of
random walk on $\zed/p\zed$ with small symmetric generating sets, the
relaxation and total variation mixing times are related as follows.\footnote{
We write $A(x)
\lesssim_x B(x)$ meaning that there is a non-increasing function $f :\bR^+
\to
\bR^+$ with $\lim_{x \to \infty}f(x) = 1$ such that $A(x) \leq f(x)B(x)$, thus
indicating the parameter which must grow for the asymptotic to hold.}
\begin{theorem}\label{spectral_mixing_time_theorem}
 Let $p$ be prime, let $1 \leq k \leq \frac{\log p}{\log \log p}$ and let $A
\in \sA(p,k)$.  Denote $t^{\rel}, t^{\mix}_1$ the relaxation time
and total variation mixing time of $\mu_A$ on $\zed/p\zed$.   We
have
\[
\frac{\tau_0 e }{4\pi } p^{\frac{2}{k}}\lesssim_k \tau_0
t^{\rel} \lesssim_p
t^{\mix}_1 \lesssim_k 0.163 k t^{\rel}.
\]
Also, uniformly in $k$,
\[
 \frac{2k+1}{16 \pi \Gamma\left(\frac{k}{2}+1 \right)^{\frac{2}{k}}}p^{\frac{2}{k}} \lesssim_p 
t^{\rel}.
\]

\end{theorem}
\begin{rem}
 The relationship $\frac{t_1^{\mix}}{t^{\rel}} \gtrsim \tau_0$ exhibits Gaussian diffusion on $\bR/\zed$ as asymptotically extremal for the ratio between the mixing and relaxation times.
\end{rem}

\begin{rem}
 The lower bound gives an explicit dependence on $k$ in Greenhalgh's theorem.
An upper bound of this type may be extracted from \cite{DS94}, Theorem 1.2, but
the $k$ dependence there is, in worst case, exponential.
\end{rem}

Theorem \ref{spectral_mixing_time_theorem} relates the mixing time to spectral
data, but in some cases it is more desirable to understand the mixing time
geometrically.  Given symmetric generating set $A \subset \zed/p\zed$ denote
$\sC(A, p)$ the Cayley graph with vertices $V=\zed/p\zed$ and edge set $E =
\{(n_1, n_2) \in (\zed/p\zed)^2: n_1 - n_2 \in A \}$.  Write $\diam(\sC(A,p))$
for the graph-theoretic diameter of $\sC(A,p)$.  Since $\zed/p\zed$ is abelian
there is a more geometric notion of diameter
\[
\diam_{\geom}(\sC(A,p)) = \max_{x \in \zed/p\zed} \min \left(\|\un\|_2:
\un \in \zed^k, \;\exists
\ua \in A^k, \; \un \cdot \ua \equiv x \bmod p \right).
\]
One has (the second inequality is given in Lemma \ref{diameter_lemma})\footnote{The notation $A \gg B$ means $B = O(A)$.}
\[
\diam(\sC(A,p)) \geq \diam_{\geom}(\sC(A,p)) \gg \sqrt{\frac{t^{\rel}}{k}}.
\] 
Random walk driven by
$\mu_A$ on $\zed/p\zed$ may be interpreted as  random walk on
$\sC(A, p)$ in which at each step the walker chooses a uniform edge leaving its
current position.

\begin{theorem}\label{mixing_time_theorem}
 Let $p$ be an odd prime and let $A \in \sA(p)$ with $|A| = 2k+1$, $1 \leq k
\leq \frac{\log p}{\log\log p}.$ The mixing time $t^{\mix}_1$ of random walk
driven by $\mu_A$ satisfies, as $p \to \infty$,
\[
 t^{\mix}_1 \ll k \cdot \diam_{\geom}(\sC(A,p))^2.
\]
\end{theorem}

\begin{rem}
In the context of random walk on a cycle, Theorem \ref{mixing_time_theorem}
refines in two ways the much more general Theorem 1.2 of \cite{DS94}, which
applies in the context of groups of moderate growth. The dependence on the
number of generators $k$ there is, in worst case, exponential.  Also, we
replace the diameter
there with the smaller geometric diameter here. See also \cite{V98}.
\end{rem}

Given a sequence of triples $(\fX_i, P_i, \nu_i)_{i=1}^\infty$ where $\fX_i$ is
a measure space and $P_i$ is a Markov kernel on $\fX_i$ which has
$\nu_i \in \sM(\fX_i)$ as its stationary distribution, the sequence
exhibits the cut-off phenomenon in total variation if for all $0 < \epsilon <
\frac{1}{2}$,
\[
 \lim_{i \to \infty}
\frac{t^{\mix}_{1,i}(\epsilon)}{t^{\mix}_{1,i}(1-\epsilon)} = 1.
\]
The cut-off phenomenon is frequently observed in natural families of Markov
chains including the hypercube walk of \cite{DGM90} and
riffle
shuffling viewed as a random walk on the symmetric group \cite{AD86}.
Especially in total variation, the cut-off phenomenon is still imperfectly
understood, so that there is significant interest in deciding its occurrence
in specific examples, see for instance \cite{DS06}, \cite{DLP10},
\cite{DDMP13}, \cite{BC14}, \cite{LP15}.

One necessary condition for cut-off in total variation to occur is
\[
 \lim_{i \to \infty} \frac{t^{\mix}_{1,i}}{t^{\rel}_i} = \infty,
\]
see Chapter 18.3 of \cite{LPW09}.
In particular, by Theorem \ref{spectral_mixing_time_theorem} any sequence of
walks generated by $\{A_p \bmod p \subset \zed/p\zed\}_{p \in \sP}$ for which
$|A_p|$ remains bounded does not have cut-off, a result first obtained in
\cite{DS94}.
We give a different proof of this result found independently by the
author, which gives further information on the period of transition to
uniformity.

\begin{theorem}\label{bounded_gen_set_theorem}
Let $p \geq 3$ be prime, let $1 \leq
k
\leq \frac{\log p}{\log \log p}$ and let $A \in \sA(p,k)$.
For any $0 < \epsilon < \frac{1}{e}$ the
total variation mixing times of $\mu_A$ on $\zed/p\zed$ satisfy
\[
 t^{\mix}_1(\epsilon) - t^{\mix}_1(1-\epsilon) \gg_\epsilon
\frac{t^{\mix}_1}{k}.
\]

\end{theorem}

In contrast to Theorem \ref{bounded_gen_set_theorem}, our next theorem shows
that the generic behavior when $|A_p|$ grows slowly is for
there to be a sharp transition to uniformity with infrequent exceptions.

\begin{theorem}\label{random_theorem}
 Let $k:\sP \to \zed_{>0}$ tend to $\infty$ with $p$ in such a way that
$k(p)
\leq \frac{\log p}{\log\log p}$. Let sets $\{A_p \bmod p\}_{p \in
\sP}$
be
chosen independently with $A_p$ chosen uniformly from $\sA(p,k(p))$.
The following hold with
probability 1.
\begin{enumerate}
 \item Let $\rho:\sP \to \bR^+$ satisfy $\sum_p \frac{1}{\rho(p)^{k}} =
\infty$.  There is an infinite subsequence $\sP_0 \subset \sP$ such that for
$p$ increasing through $\sP_0$,
\[
 t^{\rel}(p) \gtrsim \frac{e}{\pi}\rho(p)^2 p^{\frac{2}{k(p)}}
\qquad \text{and} \qquad t^{\mix}_1(p) \sim \tau_0t^{\rel}(p).
\]
 In particular, the cut-off phenomenon does not occur for
$(\zed/p\zed, \mu_{A_p},
\bU_{\zed/p\zed})_{p \in \sP}$.
\item Let $\rho:\sP \to \bR^+$ satisfy $\sum_p \frac{1}{\rho(p)^{k}} < \infty$.
Then
 \[
  t^{\mix}_1(p) \lesssim \frac{\tau_0 e}{\pi}\rho(p)^2
p^{\frac{2}{k(p)}}.
 \]

 \item For any sequence $\{\epsilon(p)\}_{p \in \sP} \subset \bR_{>0}$
satisfying $\epsilon(p)\sqrt{k(p)} \to \infty$ there is a density 1 subset
$\sP_0 \subset
\sP$ such that in the family $(\zed/p\zed, \mu_{A_p}, \bU_{\zed/p\zed})_{p \in
\sP_0}$ we have
 \[
  t^{\mix}_1(p) \sim \frac{k(p)}{2\pi e} p^{\frac{2}{k(p)}},
 \]
 and as $p$ increases through $\sP_0$
 \[
  \lim \left\| \mu_{A_p}^{(1-\epsilon)t^{\mix}_1(p)} -
\bU_{\zed/p\zed}\right\|_{\TV(\zed/p\zed)} = 1, \qquad
  \lim \left\| \mu_{A_p}^{(1+\epsilon)t^{\mix}_1(p)} -
\bU_{\zed/p\zed}\right\|_{\TV(\zed/p\zed)} = 0.
 \]
 In particular, the cut-off phenomenon occurs.
\end{enumerate}
\end{theorem}
\begin{remark}
 Since $\sum_p \frac{1}{p} = \infty$, items (1) and (3) of Theorem \ref{random_theorem} demonstrate that almost surely among a sequence of walks, infinitely often there are slowly mixing walks which are slower than the typical behavior by a factor of  $\gg \frac{p^{\frac{2}{k(p)}}}{k(p)}$.
\end{remark}

\begin{remark}
  (3) of Theorem \ref{random_theorem} gives a cut-off
sequence with, for $0 < \epsilon < \frac{1}{2}$,
period of transition between $t_1^{\mix}(1-\epsilon)$ and
$t_1^{\mix}(\epsilon)$ of length
$O_\epsilon\left(\frac{t^{\mix}_1}{\sqrt{k}}\right)$.
While this is
 longer than the lower bound $\frac{t^{\mix}_1}{k}$ given in
Theorem \ref{bounded_gen_set_theorem}, it is much shorter than the true
transition period for many known examples giving cut-off.  For instance, the
transition period of random walk on the hypercube is faster than the mixing
time by a factor which is logarithmic in the number of generators.
\end{remark}

Our proofs of Theorems \ref{spectral_mixing_time_theorem}--\ref{random_theorem}
approximate the distribution of random walk on the cycle $\zed/p\zed$ with that
of a Gaussian diffusion on $\bR^k/\Lambda$ where $\Lambda$ is a co-volume $p$
lattice.
In making the transition between these models we use the following quantitative
normal approximation lemma for which we don't know an easy reference in the
literature.  A proof is included in Appendix \ref{local_limit_appendix}.
\begin{lemma}\label{normal_approximation_lemma} Let $n, k(n) \geq 1$ with $k^2 =
o\left(n\right)$ for
large $n$.
Let $\nu_k$ be the measure on $\bR^k$ which is uniform on $\{0, \pm e_i,\; 1
\leq i \leq k\}$, where $e_i$ denotes the $i$th standard basis vector.
For $\sigma > 0$ set \[\eta_k\left( \sigma, \ux\right) = \left(\frac{1}{2\pi
\sigma^2}\right)^{\frac{k}{2}}\exp\left(-\frac{\|\ux\|_2^2}{2\sigma^2} \right)\]
the standard Gaussian density.
As $n \to \infty$ we have
 \[
\left\| \nu_k^{*n} \ast \one_{\left[-\frac{1}{2}, \frac{1}{2}\right)^k} -
\eta_k\left(\sqrt{\frac{2n}{2k+1}}, \cdot\right)\right\|_{\TV(\bR^k)} = o(1).
 \]
\end{lemma}

After transition to the diffusion model, the measure on lattices induced from
the random choice in Theorem
\ref{random_theorem} is close to the uniform measure on the (rescaled)
$p$-Hecke points, which are the index $p$ lattices of $\zed^k$.  It is known
that, after rescaling to volume 1, as $p \to \infty$ these lattices are
equidistributed with respect to the induced Haar measure in the space
$\SL_k(\zed)\backslash \SL_k(\bR)$ of all volume 1 lattices in
$\bR^k$. Statistics
regarding correlations of vectors in a random lattice are well-known, see for
instance \cite{S11} for a modern treatment.  Although we estimate somewhat
different quantities, the results considered there may be useful in
understanding our argument.

We conclude by giving an example of random walk on the cycle which
has cut-off.  This may be considered an approximate embedding of the classical
hypercube walk into the cycle.
\begin{theorem}\label{power_2_theorem}
 For $p \in \sP$ let $\ell_2(p) = \lceil \log_2 p \rceil$ (logarithm base 2)
and let the
power-of-2 set be $A_{2,p} = \{0, \pm 1, \pm 2, ..., \pm 2^{\ell_2(p)-1}\}
\subset
\zed/p\zed$.
Set
\[
 c_0 = \sum_{j=1}^\infty \left(1 - \cos \frac{2\pi}{2^j}\right) =
3.394649802(1).
\]
The power-of-2 walk $(\zed/p\zed, \mu_{A_{2,p}}, \bU_{\zed/p\zed})_{p \in
\sP}$ has
cut-off in total variation at mixing time
\[
 t_1^{\mix}(p) \sim \frac{\ell_2(p)\log \ell_2(p)}{2c_0}.
\]
\end{theorem}

\subsection{Discussion of method}
Our arguments view random walk on the cycle $\zed/p\zed$ with symmetric generating set $A$, $|A| = 2k+1$ as random walk on an index $p$ quotient of $\zed^k$, in which a standard basis vector is assigned to each non-zero symmetric pair $\{x,-x\}$ of generators. The index $p$ lattice is the set $\Lambda =\{\un \in \zed^k: \sum_x n_x x \equiv 0 \bmod p\}$.  In the case of Theorem \ref{power_2_theorem}, the corresponding lattice is approximately cubic, and the argument is a perturbation of the Fourier analytic analysis of the hypercube walk in \cite{DS87}.  In particular, the mixing time and cut-off are the same in total variation and in $L^2$.

For $k \leq \frac{\log p}{\log \log p}$, a random index $p$ lattice gives a mixing time in total variation which is less than the $L^2$ mixing time by a constant, and thus the $L^2$ methods of proving cut-off are not immediately suitable.  Thus in our first four Theorems the arguments are made initially in time domain by first applying Lemma \ref{normal_approximation_lemma} to replace the discrete random walk with a diffusion on $\bR^k/\Lambda$.  This initial step is the reason for the restriction on the size of $k$ since the corresponding approximation fails for $k > (1+\epsilon) \frac{\log p}{\log \log p}$.  For larger $k$ there is a standard method of correcting the approximation using the saddle point method, but we have not made an attempt to do so.  

After having made the Gaussian approximation, Theorem \ref{spectral_mixing_time_theorem} combines standard spectral estimates with bounds for the shortest vector in a lattice (the lower bound) and for sphere packing (the upper bound).  Theorem \ref{mixing_time_theorem} goes through in time domain, using convexity.  Theorem \ref{bounded_gen_set_theorem} goes through in time domain, and uses an estimate for the derivative of the density in time. 

Parts (1) and (2) of Theorem 4 study rare events in which the random lattice is essentially one dimensional due to the presense of many short vectors.  We study these cases in frequency space. The dual lattice of an index $p$ lattice of $\zed^k$ is $\Lambda^\vee = \zed^k + \ell$ where \[\ell = \ell_v= \{a v: 0 < a < p\},\qquad v\in \frac{1}{p}\zed^k \setminus \zed^k\] is a line. We are able to show that with high probability the large Fourier coefficients arise from  frequencies which are small multiples of a single vector.  The analysis restricts attention to primitive vectors, and their multiples by \emph{Farey fractions modulo $p$}, which are residues $bq^{-1}\bmod p$ in which $b$ and $q$ are bounded. 

Part (3) of Theorem 4 is proven in time domain again.  After removing a small $L^1$ error, the modified density may be estimated using a variance bound.  In particular, our argument requires  averages concerning pairs of short vectors in a random lattice which are discrete analogues of the averages performed by Siegel and Rogers  \cite{S45}, \cite{R55} regarding the
distribution of vectors in a random lattice.

\subsection{Possible extensions}
From the point of view of mixing of Markov chains, an attractive open problem  is to decide the Peres conjecture \[\text{cut-off} \; \Leftrightarrow \;t^{\mix}/t^{\rel} \to \infty\] for random walk on a cycle. 

Abelian groups are prevalent in arithmetic, and there would be interest in extending the results to random walks on more general abelian groups. The class group of an imaginary quadratic field grows like the discriminant to the power $\frac{1}{2} + o(1)$, so a reweighting of Theorem \ref{random_theorem} with roughly $d$ groups of order $d$ would be of interest.  The techniques presented should translate without any great difficulty to studying random walk on cycles of composite order.  The general  case has not been considered, but see \cite{W97} for a study of random random walk on the hypercube. 

To model abelian sandpiles, asymmetric generating sets should be considered.
\section*{Notation and conventions}
Given groups $G, H$, $H< G$ indicates that $H$ is a subgroup of $G$ and $[G:H]$
denotes the index. $\fS_k$ is the symmetric group on $k$ letters and we write
$ (\zed/2\zed)^k\rtimes\fS_k  = O_k(\zed)$ for the $k
\times k$
 orthogonal group over $\zed$.  For ring $R = \zed, \zed/p\zed$,
$\GL_n(R)$ and $\SL_n(R)$ are the usual linear groups with entries in $R$.
We denote $e(x)= e^{2\pi i x}$ the
standard additive character on $\bR/\zed$.

Given measure space $(\fX, \sB)$, $\sM(\fX)$ indicates the Borel probability
measures on $\fX$.   When $\fX$ is a finite set, $\bU_{\fX}$ denotes the uniform
probability measure on $\fX$ and when $\fX$ is a compact abelian group,
$\bU_{\fX}$ denotes the probability Haar measure.  In either case expectation
and variance with respect to $\bU_{\fX}$ are indicated $\E_\fX$ and
$\Var_{\fX}$. $\| \cdot \|_{\TV(\fX)}$ indicates the total variation norm on
$\sM(\fX)$.

Unless otherwise stated, $\| \cdot \|$ indicates the $\ell^2$-norm on $\bR^k$,
$k
\geq 1$, $\| \cdot \|_p$ denotes the $\ell^p$ norm, $p \geq 1$, and $\| \cdot
\|_{(\bR/\zed)^k}$ denotes the $\ell^2$ distance to the nearest integer lattice
point.  $\bS^{k-1}$ is the unit sphere in $\bR^k$, $\bS^{k-1} = \{\ux \in
\bR^k: \|\ux\|_2 = 1\}$.  Given $\ux \in \bR^k$, $R \in \bR_{>0}$, and $p \geq
1$, $B_p(\ux, R)$
denotes the
$\ell^p$ ball \[B_p(\ux, R) = \left\{\uy \in \bR^k: \|\uy - \ux\|_p \leq
R\right\},\] the ambient dimension being clear from the context. If $p$ is not
stated $\ell^2$ is assumed.  Given further
parameter $0 < \tau < 1$, $S(\ux, R, \tau)$ indicates the spherical shell
\[S(\ux, R, \tau) = \left\{\uy \in \bR^k: \|\ux - \uy\|_2 \in [(1-\tau) R, (1 +
\tau) R]\right\}.\] For $k \geq 1$,
\[
 R_k = \left(\frac{\Gamma\left(\frac{k}{2}+1 \right)}{\pi^{\frac{k}{2}}}
\right)^{\frac{1}{k}}= \left(1 + \frac{\log
(k+1)}{2k} + O\left(\frac{1}{k} \right) \right)\sqrt{\frac{k}{2\pi e}}
\]
is the radius of an $\ell^2$ ball of unit volume in $\bR^k$. One may check that $R_k > \sqrt{\frac{k}{2\pi e}}$ for all $k \geq 1$.

For $k \geq 1$, given $\ux \in \bR^k$ and $\sigma \in \bR_{>0}$,
$\eta_k(\sigma, \ux)$ denotes the density at  $\ux$ of a symmetric centered
Gaussian distribution scaled by $\sigma$,
\[
 \eta_k(\sigma, \ux) = \left(\frac{1}{2\pi
\sigma^2}\right)^{\frac{k}{2}}\exp\left(-\frac{\|\ux\|_2^2}{2 \sigma^2}\right).
\]

By default, quantities considered depend upon a large prime parameter $p$
varying
over a set of primes $\sP_0$.    We use the Vinogradov notation $A \ll B$ with
the same meaning as $A(p) = O(B(p))$. $A\asymp B$ means $A\ll B$ and $B\ll A$.
For positive parameters $A, B$, $A \sim
B$
means $\lim_{p \to \infty} \frac{A(p)}{B(p)} = 1$ and $A \lesssim B$, resp. $A
\gtrsim B$ means $\limsup \frac{A(p)}{B(p)} \leq 1$, resp. $\liminf
\frac{A(p)}{B(p)} \geq 1$. We also use the non-standard notation already
introduced in the introduction $A \lesssim_x B$, with the meaning that there is
a non-increasing function $f : \bR^+ \to \bR^+$ with $\lim_{x \to \infty}f(x)=1$
such that $A(x) \leq f(x) B(x)$.
\section{Background}\label{background_section}
This section collects together several statements regarding classical
probability
theory and lattice theory on $\bR^k$, $k \geq 1$.

\subsection{Classical probability}
See \cite{D88} for background regarding random walk on a group and \cite{LPW09} for a thorough treatment of Markov chains.  We have provided proofs of the statements which we use for the reader's convenience.

We have already introduced the total variation distance between two probability
measures $\mu, \nu$ on a measure space $(\fX, \sB)$, by
\[
 \left\|\mu - \nu\right\|_{\TV(\fX)} = \sup_{S \in
\sB}\left|\mu(S)-\nu(S)\right|.
\]
In the case when $\mu$ has a density with respect to $\nu$, equivalent
characterizations are
\[
\left\|\mu - \nu\right\|_{\TV(\fX)} =
\frac{1}{2}\int_{\fX}\left|\frac{d\mu}{d\nu}-1 \right| d\nu =
\int_{\fX}\left(\frac{d\mu}{d\nu}-1 \right)\one\left(\frac{d\mu}{d\nu}>1
\right)d\nu.
\]

When $\mu$ is the distribution of a Markov chain with stationary measure $\nu$
define the $L^2(d\nu)$ distance to stationarity by
\[
 \|\mu- \nu\|_{L^2(d\nu)} = \frac{1}{2}\left(\int \left(\frac{d\mu}{d\nu}
-1\right)^2
d\nu
\right)^{\frac{1}{2}}
\]
with the convention that the norm is infinite if $\frac{d\mu}{d\nu}$ is not in
$L^2(d\nu)$. The factor of $\frac{1}{2}$ is for consistency with the
interpretation of total variation distance as half the $L^1(d\nu)$ norm.
For $\epsilon > 0$  denote $t^{\mix}_2(\epsilon)$ the $\epsilon$-mixing time
of the $L^2(d\nu)$ norm.

\begin{lemma}\label{t_mix_lower}
 Convolution with a probability measure is a contraction in the total variation
norm. Also, given symmetric probability measure $\mu$ on
finite or compact
abelian group $G$, for any $0 < \epsilon < 1$ the total variation mixing time of
random walk driven by
$\mu$ satisfies
$t^{\rel} \log \frac{1}{2\epsilon} \leq t^{\mix}_1(\epsilon) \leq
t^{\mix}_2(\epsilon)$ and
$\frac{2\pi^2}{27}\epsilon^3 t^{\rel} \lesssim t^{\mix}_1(1-\epsilon)$ as
$\epsilon \downarrow 0$.

\end{lemma}
\begin{proof}
 The contraction property follows from the triangle inequality.

To prove $t^{\mix}_1(\epsilon) \leq t^{\mix}_2(\epsilon)$, use the $L^1$
characterization of the total
variation metric and Cauchy-Schwarz
\[
\|\mu^{*n} - \bU_G\|_{\TV(G)} = \frac{1}{2}\int_{G}\left|\frac{d
\mu^{*n}}{d\bU_G} - 1\right| d\bU_G \leq  \|\mu^{*n}-
\bU_G\|_{L^2(d \bU_G)}.
\]

 To prove the lower bounds regarding $t^{\rel}$,
observe that the eigenvalues of the
transition kernel for the
random walk are given by
\[
\spec(\mu) = \left\{\E_{\mu}[\chi]: \chi \in \widehat{G}\right\},
\]
where $\widehat{G}$ denotes the set of characters of $G$.
Let $\chi_1$ generate the spectral gap.
Since $\|\chi_1\|_\infty \leq 1$, we have, for any $n \geq 1$,
\[
 \|\mu^{*n} - \bU_G\|_{\TV(G)} \geq
\frac{1}{2}\left|\E_{\mu^{*n}}[\chi_1]\right| =
\frac{1}{2}\left|\left(\E_{\mu}[\chi_1]
\right)^n\right|,
\]
so that the first mixing time bound follows by taking logarithms.

To obtain the
bound for $t^{\mix}_1(1-\epsilon)$, let $\epsilon_0 > \epsilon_1$ be small
parameters, satisfying, for some $A, B > 0$,  $\epsilon_0 = A
\epsilon^2$, $\epsilon_1 = B\epsilon^{3}$. Let $n$
be maximal such that $\E_{\mu^{*n}}[\chi_1] \geq 1-\epsilon_1$.
Set $S = \{g \in G: \RE(\chi_1(g))\geq 1-\epsilon_0 \}$ and $\alpha =
\mu^{*n}(S)$.  Bounding $\RE\left(\chi|_S\right) \leq 1$ and
$\RE\left(\chi|_{S^c}\right) \leq 1-\epsilon_0$,
\[
 (1-\epsilon_1) \leq \E_{\mu^{*n}}[\chi_1] \leq \alpha +
(1-\epsilon_0)(1-\alpha)
\]
whence $\alpha \geq 1 - \frac{\epsilon_1}{\epsilon_0}$. According to uniform
measure, $\RE(\chi)$ has the same distribution as $\cos(2\pi x)$ on $(\bR/\zed,
dx)$, so that
\[
 \bU_G(S) = \frac{\cos^{-1}(1-\epsilon_0)}{\pi} =
\frac{\sqrt{2\epsilon_0}}{\pi}(1 + O(\epsilon_0)).
\]
It follows that
\begin{align*}
\left\|\mu^{*n}-\bU_G\right\|_{\TV(G)} &\geq \mu^{*n}(S) - \bU_G(S) \\&\geq 1 -
\frac{\epsilon_1}{\epsilon_0} - \frac{\cos^{-1}(1-\epsilon_0)}{\pi} = 1 -
\left(\frac{B}{A}+ \frac{\sqrt{2 A}}{\pi} + O(\epsilon^2) \right)\epsilon.
\end{align*}
Imposing the constraint  $\mu^{*n}(S) - \bU_G(S) \geq 1 -\epsilon$ gives
$t^{\mix}_1(1-\epsilon) \geq n+1$.
As $\epsilon \downarrow 0$, one obtains the constraint $\left(\frac{B}{A}+
\frac{\sqrt{2 A}}{\pi} + O(\epsilon^2) \right) < 1$, which gives the asymptotic
claimed with $A \sim \frac{2\pi^2}{9}$, $B \sim \frac{2\pi^2}{27}$.

\end{proof}

Define the standard symmetric centered normal distribution on $\bR^k$ scaled
by $\sigma \in \bR_{>0}$ to be
\[\eta_k(\sigma, \ux) =
\frac{1}{(2\pi \sigma^2)^{\frac{k}{2}}}
\exp\left(\frac{-\|\ux\|_2^2}{2\sigma^2}\right).\]  For $t \in \bR_{>0}$,
$\eta(\sqrt{t}\sigma, \ux)$ is its $t$-fold
convolution.
We use several results regarding concentration of the Gaussian measure.
\begin{lemma}\label{gaussian_length_concentration}
 Let $k \geq 1$ and $\sigma > 0$.  There are positive constants $C,
\{C_p\}_{2\leq
p < \infty}$ such
that,
for any $t>C$,
 \[
\int_{\ux \in \bR^k} \eta_k(\sigma, \ux)\one\left( \left| \|\ux\|_2 -
\sigma\sqrt{k}\right| > \sigma t \right)d\ux \leq \exp\left(-\frac{(t-C)^2}{2}
\right),
 \]
 and, for all $t > 0$, for all $2 \leq p < \infty$,
 \begin{align*}
  \int_{\ux \in \bR^k} \eta_k(\sigma, \ux)\one\left( \|\ux\|_p > C_p
\sigma k^{\frac{1}{p}} + t\sigma\right) d\ux &\leq  \exp\left(-
\frac{t^2}{2}\right).
 \end{align*}

\end{lemma}
\begin{proof}
 All quantities scale with $\sigma$ so we may assume $\sigma = 1$.  Let
$\gamma_k$ denote the measure on $\bR^k$ with density $\gamma_k(\ux) =
\frac{1}{(2\pi)^{\frac{k}{2}}}\exp\left(-\frac{\|\ux\|_2^2}{2}\right).$  Let
$M_p$, $2 \leq p < \infty$ denote the median with respect to $\gamma_k$ of
$\|\cdot\|_p$, that is, $\gamma_k\left(\ux: \|\ux\|_p \leq M_p\right)
=\frac{1}{2}$. Since $\|\cdot\|_p$ is 1-Lipschitz on $(\bR^k
,\|\cdot\|_2)$ for $p \geq 2$, Talagrand's
inequality (\cite{LT11}, p.21) gives, for any $t > 0$,
\[
\gamma_k\left(\ux: |\|\ux\|_{p} - M_p| > t \right) \leq
\exp\left(-\frac{t^2}{2}\right).
\] The first statement follows,
since the mean, root mean square, and median of $\|\cdot\|_2$ differ by
constants, as is evident from the concentration around the median. The
second statement follows since $M_p \ll
k^{\frac{1}{p}}$.
\end{proof}

\subsection{Lattices}
Siegel's \emph{Lectures on the Geometry of Numbers} \cite{S13} are a recommended reference.

A lattice $\Lambda < \bR^k$ is a discrete finite co-volume subgroup of $\bR^k$.
Write
 \[
  \vol(\Lambda) = \int_{\bR^k/\Lambda} d \ux
 \]
for its co-volume.
Fixing the usual inner product $\langle \cdot, \cdot \rangle$ on $\bR^k$,
the dual lattice of lattice $\Lambda$ is
\[
 \Lambda^{\vee} = \left\{\lambda' \in \bR^k: \forall\, \lambda \in \Lambda,
\langle \lambda', \lambda \rangle \in \zed \right\}.
\]
This satisfies $\vol(\Lambda)\cdot \vol(\Lambda^\vee) = 1$. For instance, the dual lattice to $\Lambda = 2\zed$ is $\frac{1}{2} \zed$. More generally, if $\Lambda = Q \zed^k$ for some $Q \in \GL_k(\bR)$, then $\Lambda^{\vee} = (Q^{-1})^t \zed^k$.  We reserve
$\lambda^*$
for the shortest non-zero vector of $\Lambda^{\vee}$.

Given lattice $\Lambda < \bR^k$,  its norm-minimal fundamental domain (Voronoi cell) is
\[
 \sF(\Lambda) = \{ \ux \in \bR^k: \forall \;\lambda \in \Lambda \setminus \{0\},
\|\ux\| < \|\ux - \lambda\|\}.
\]
One may choose a set $\sF^0(\Lambda)$
\[
 \sF(\Lambda) \subset \sF^0(\Lambda) \subset \overline{\sF(\Lambda)}
\]
such that every $\ux \in \bR^k/\Lambda$ has a unique representative in
$\sF^0(\Lambda)$.

Minkowski's geometry of numbers gives an upper bound for the shortest non-zero
vector in a lattice.

\begin{theorem}[Minkowski's Theorem] Let $\Lambda \subset \bR^k$ be a
  lattice and let $C$ be a convex symmetric body, i.e. $\ux \in C
  \Leftrightarrow - \ux \in C$.  If \[\mathrm{vol}(C) > 2^k
  \mathrm{vol}(\Lambda)\] then $C$ contains a non-zero vector in
  $\Lambda$.  In particular
  \[
   \min_{\lambda \in \Lambda \setminus \{0\}} \|\lambda\|_2 \leq
\frac{2}{\sqrt{\pi}} \left(\Gamma\left(\frac{k}{2}+1 \right) \vol(\Lambda)
\right)^{\frac{1}{k}} \sim \sqrt{\frac{2k}{\pi e}} \vol(\Lambda)^{\frac{1}{k}}.
  \]
  with the asymptotic holding as $k \to \infty$.
\end{theorem}

For lattice $\Lambda$, the diameter of the norm-minimal fundamental domain and
the shortest non-zero vector in the dual lattice are related as follows.
\begin{lemma}\label{diameter_lemma}
 Let $\Lambda$ be a lattice with norm-minimal fundamental domain $\sF$ and dual
lattice $\Lambda^{\vee}$. Let $\lambda^*$ be the shortest non-zero vector in
$\Lambda^{\vee}$.  We have
\[
 \|\lambda^*\|_2 \cdot \diam(\sF) \geq 1.
\]
\end{lemma}

\begin{proof}
 Let $\uv = \frac{\lambda^*}{\|\lambda^*\|_2}$ and choose $\ux$ the point on
the boundary of $\sF$ on the ray determined by $\uv$.  Write $\ux = x_0
\uv$.  Since $\ux \in \partial(\sF)$ we may find $\uy \in
\Lambda\setminus\{0\}$
with $\left|\langle \ux, \uy\rangle\right| = \frac{1}{2}\|\uy\|_2^2.$  Set $\uy
= y_0 \uv +  \uv'$ where $\langle \uv, \uv'\rangle = 0$.  In particular,
$y_0 \neq 0$ so $\left|\langle \uy, \lambda^*\rangle\right| = \|\lambda^*\|_2
|y_0| \geq 1$.  Since $|x_0 y_0| \geq \frac{1}{2} y_0^2$ it follows that
$\|\ux\|_2\cdot\|\lambda^*\|_2 \geq \frac{1}{2}.$ The diameter is at
least as large as $2 \|\ux\|_2$.
\end{proof}

Given $\ux \in \bR^k$ and $R>0$, let $B_2(\ux, R)$ denote the ball
\[
B_2(\ux, R) = \{\uy \in \bR^k: \|\ux - \uy\|_2 \leq R\}.
\]
The following is an easy estimate for the number of lattice points
contained in a ball.
\begin{lemma}\label{point_count_lemma}
 Let $k \in \zed_{>0}$, let $\ux \in \bR^k$ and let $R > k^{\frac{3}{2}}$. Then
\[
\left| \zed^k \cap B_2(\ux, R)\right| = \left(1 +
O\left(\frac{k^{\frac{3}{2}}}{R}\right) \right) \vol(B_2(\ux, R)).
\]

\end{lemma}
\begin{proof}
 Let $\mu_{\ux,
R} = \sum_{\un \in \zed^k \cap B_2(\ux, R)}\delta_{\un}$.
Since the hypercube $\left[-\frac{1}{2},
\frac{1}{2}\right)^k$  has diameter $\sqrt{k}$,
\[
\one_{B_2(\ux, R-\sqrt{k})} \leq \mu_{\ux, R} \ast \one_{\left[-\frac{1}{2},
\frac{1}{2}\right)^k} \leq \one_{B_2(\ux,
R+\sqrt{k})}
\]
and thus
\begin{align*}
\left| \zed^k \cap B_2(\ux, R)\right| &= \int_{\bR^k}\mu_{\ux, R}\ast
\one_{\left[-\frac{1}{2},
\frac{1}{2}\right)^k}
\\&= \left(1 + O\left(\frac{\sqrt{k}}{R}\right)\right)^k \vol(B_2(\ux, R))\\& =
\left(1 + O\left(\frac{k^{\frac{3}{2}}}{R}\right)\right) \vol(B_2(\ux, R)).
\end{align*}
\end{proof}

We also use the following estimate counting lattice points of a more general
lattice.
\begin{lemma}\label{lattice_point_bound_lemma}
 Let $\Lambda < \bR^k$ be a lattice with shortest non-zero vector $\lambda^*$.
For any $t \geq 1$,
\[
\log \left|\Lambda \cap B_{2}(0, t\|\lambda^*\|) \right| \lesssim_k k
\left[\frac{1 + \sin \theta}{2\sin \theta}
\log \frac{1 + \sin \theta}{2 \sin \theta} - \frac{1 - \sin \theta}{2 \sin
\theta} \log \frac{1 - \sin \theta}{2 \sin \theta} \right],
\]
where $\theta = 2
\sin^{-1}\left( \frac{1}{2t}\right)$.
\end{lemma}
\begin{proof}
This follows from \cite{KL78}, see
\cite{CZ14} for a nice exposition and related results.  We sketch the argument.

Write $B_{2,j}$ for a $\ell^2$ ball in $\bR^j$.
 By rescaling we may assume $\|\lambda^*\|_2 = 1$. View $\bR^k$ as a hyperplane
through zero in $\bR^{k+1}$, and consider the ball $\tilde{B} = B_{2, k+1}(0,
t)$ in $\bR^{k+1}$.  Project $\Lambda \cap B_{2, k}(0,t)$ orthogonally onto
$\tilde{B}$.  The points remain 1-spaced and thus satisfy an angular spacing of
at least $\theta = 2 \sin^{-1}(\frac{1}{2t}).$  Let, as in \cite{KL78}, $A(n,
\theta)$ denote the largest set $S \subset \bS^{n-1}$ which is separated by
angle $\theta$ as above.  Thus
\[
 \Lambda \cap B_{2, k}(0,t) \leq A(k+1, \theta).
\]
The claimed estimate for $A(k+1, \theta)$ is the main result of \cite{KL78}.

\end{proof}

Given a probability measure $\mu \in \sM(G)$, $G = \zed^k$ or $G = \bR^k$, and
a lattice $\Lambda < G$ the quotient measure $\mu_{\Lambda}$ is defined for $f
\in C(G/\Lambda)$ by
\[
 \langle f, \mu_\Lambda\rangle_{G/\Lambda} = \langle f, \mu\rangle_G.
\]
Quotienting commutes with convolution and contracts the total variation norm.
For lattice $\Lambda < \bR^k$, $t \in \bR_{>0}$ and $\ux \in
\bR^k$, the quotient measure of Gaussian $\eta_k\left(\sqrt{t}, \cdot\right)$
is the  theta function
\[
 \Theta(\ux, t; \Lambda) = \sum_{\lambda \in \Lambda} \eta_k\left(\sqrt{t}, \ux
+ \lambda\right).
\]
This has a representation in frequency space as
\[
\Theta(\ux, t;\Lambda) = \frac{1}{\vol(\Lambda)}\sum_{\lambda \in
\Lambda^{\vee}}\exp\left(-2\pi^2 t\|\lambda\|_2^2 \right)e(\lambda \cdot
\ux).
\]
To check the expansion, Fourier expand $\Theta$ in the orthonormal basis $\left\{\frac{e(\lambda \cdot \ux)}{\sqrt{\vol(\Lambda)}}\right\}_{\lambda \in \Lambda^{\vee}}$ for $L^2(\bR^k/\Lambda)$ (this is the usual proof of the Poisson summation formula).
In the case of a cubic lattice, where for some $\alpha \in \bR_{>0}$, $\Lambda =
\alpha \zed^k$, the theta function is particularly pleasant.
\begin{lemma}\label{theta_lemma}
 Let $k \in \zed_{>0}$, $\alpha, t \in \bR_{>0}$ and $\ux \in \bR^k$.  We
have
\[
 \Theta\left(\ux, t; \alpha \zed^k\right) = \prod_{i=1}^k \Theta\left(x_i, t;
\alpha \zed\right).
\]
The one dimensional theta function $\Theta(x, t; \alpha\zed)$ satisfies
\begin{align*}
 \Theta(x, t; \alpha\zed) &= \frac{\exp\left(- \frac{\alpha^2
\left\|\frac{x}{\alpha} \right\|_{\bR/\zed}^2}{2t} \right)}{\sqrt{2\pi t}} +
O\left(\frac{\exp\left(-\frac{\alpha^2}{8t} \right)}{\sqrt{2\pi
t}\left(1-\exp\left(-\frac{\alpha^2}{8t} \right) \right)} \right) \\
& = \frac{1}{\alpha} + O\left(\frac{\exp\left(-\frac{2\pi^2 t}{\alpha^2}
\right)}{\alpha\left(1- \exp\left(-\frac{2\pi^2 t}{\alpha^2}
\right)\right)} \right).
\end{align*}

\end{lemma}

\begin{proof}
 The factorization is immediate from the definition of $\Theta$.  The first
estimate for $\Theta(x,t; \alpha\zed)$ is the result of pulling out the largest
term and bounding the remaining terms by a geometric progression.  For the
second, apply the Poisson summation formula,
\begin{align*}
 \sum_{n \in \zed} \eta_1\left(\sqrt{t}, x+\alpha n \right) =
\frac{1}{\alpha}\sum_{n
\in \zed} \exp\left(-\frac{2\pi^2 tn^2}{\alpha^2}
\right)e\left(\frac{xn}{\alpha} \right)
\end{align*}
and bound the $n \neq 0$ terms by a geometric progression.
\end{proof}

\subsection{Identification between generating sets and
lattices}\label{gen_sets_lattices_section}
Our proofs of Theorems \ref{spectral_mixing_time_theorem}-- \ref{random_theorem}
approximate random walk on $\zed/p\zed$ with
symmetric generating set $A$, $|A| = 2k + 1$ with a Gaussian diffusion on
$\bR^{k}/\Lambda$ where $\Lambda$ is a co-volume $p$ lattice.  The reduction
is as follows.

Let $O_k(\zed) \cong (\zed/2\zed)^k \rtimes \fS_k$ be the orthogonal group over
$\zed$ consisting of signed $k
\times k$ permutation matrices, which acts naturally on $\bR^k$. Let
\begin{align*}
L &= L(p,k) = \{ \Lambda < \zed^k: [\zed^k:\Lambda] = p\}\\
\sL &= \sL(p, k) =  O_k(\zed)\backslash L(p,k)
\end{align*}
be the set of index-$p$ lattices of $\zed^k$, resp.\ those lattices up to
$O_k(\zed)$-equivalence. The action is matrix multiplication on the left applied to lattice vectors.
Define subsets
\begin{align*}
 L^0(p,k) &= \left\{ \Lambda \in L(p, k): \lambda \in \Lambda \setminus\{0\}
\; \Rightarrow \; \|\lambda\|_2^2 > 2\right\}\\
 \sL^0(p,k) &=  O_k(\zed)\backslash L^0(p,k).
 \end{align*}

 Let
 \[
A(p,k) = \left\{\ua \in (\bF_p^\times)^k: \forall 1 \leq i < j \leq k, \, a_i
\neq \pm a_j \right\}.
 \]
 $\sA(p,k)$ may be identified with $O_k(\zed)\backslash A(p,k)$ by
interpreting the
factors of $\left(\zed/2\zed\right)^k$ as flipping signs, and the factor of
$\fS_k$ as rearranging the order of the coordinates in the vector. Evidently
the action is free, so that uniform measure on $A(p,k)$ descends to uniform
measure on $\sA(p,k)$.

 $\bF_p^\times$ acts freely on $A(p,k)$ dilating all coordinates
simultaneously. $\bF_p^\times \backslash A(p,k)$ and $L^0(p,k)$ are
in bijection via the map
\[
 A(p,k) \ni \ua  \overset{\phi}{\mapsto} \Lambda(\ua) = \left\{\un \in \zed^k:
\sum_{i=1}^k n_i a_i \equiv 0 \bmod p\right\} \in L^0(p,k).
\]
The map in the reverse direction is
\[\Lambda  \overset{\phi}{\mapsto}
\ua(\Lambda) = \{1, a_2, \cdots, a_k: \forall i,\, e_1 - a_i e_i \equiv 0 \bmod
p\}.
\]
It follows that uniform measure on $A(p,k)$ pushes forward to uniform
measure on $L^0(p,k)$.  $O_k(\zed)$ acts on $L^0(p,k)$, and we obtain a
map $\bF_p^\times \backslash \sA(p,k) \overset{\overline{\phi}}{\mapsto}
\sL^0(p,k)$ which we write as $\Lambda(A)$. Note that the joint action of
$\bF_p^\times \times O_k$ on $A(p,k)$ need not be free, but this will not
concern us. We write $\bU_{L}, \bU_{L^0}$ for uniform measure on $L$ and $L^0$.

Let $\nu = \nu_{k} \in \sM(\zed^k)$ be the uniform measure on
$ S_{k} = \{0, \pm e_1, ..., \pm e_k\},$ $e_i$ the $i$th standard basis vector.
Let $A \in \sA(p,k)$.    For
any $n \geq 1$ the
law of $\mu_A^{*n}$ on $\zed/p\zed$ and $(\nu_k^{*n})_{\Lambda(A)}$ on
$\zed^k/\Lambda(\ua)$ are equal.  The above observations imply that we
may sample the laws of $\mu_A^{*n}$ with $A$ chosen according to
$\bU_{\sA(p,k)}$ by instead sampling the laws of $(\nu_k^{*n})_{\Lambda}$ with
$\Lambda$ drawn according to $\bU_{L^0(p,k)}$.

Combining this discussion with Minkowski's theorem has the following
consequence.
\begin{lemma}\label{mixing_time_lower_bound_lemma}
Let $p $ be a large prime, let $1 \leq k < \frac{2\log p}{\log \log
p}$ and let $A \in \sA(p,k)$. Let
$\Lambda <
\zed^k$ be any lattice in the class of $\Lambda(A) \in \sL$, and let
\[
 \ell(A) = \min\{\|\lambda\|_2: 0 \neq \lambda \in \Lambda^\vee\}.
\]
The  relaxation time of random walk driven by
$\mu_A$ on $\zed/p\zed$ satisfies
\[
t^{\rel} \sim \frac{2k+1}{4\pi^2 \ell(A)^2}.
\]

\end{lemma}
\begin{proof}
The characters of $\zed^k/\Lambda$ are given by the dual group,
$\Lambda^{\vee}/\zed^k$.
 Let $\lambda^* = (\lambda_1, ..., \lambda_k)$ be a vector of minimal length in
$\Lambda^{\vee}\setminus
\{0\}$.  The claim  follows on noting that the spectral gap is given by
 \[
 1- \hat{\nu}_{\Lambda}(\lambda^*) =  \frac{1}{2k+1}\sum_{j=1}^k \left(2 -
2\cos\left(2\pi \lambda_j \right) \right) = \frac{4\pi^2}{2k+1}\sum_{j=1}^k
\left(\lambda_j^2 + O(\lambda_j^4) \right).
 \]
The error is of lower order since $\|\lambda^*\|_\infty \ll
\sqrt{k}p^{-\frac{1}{k}}= o(1)$ by Minkowski's
Theorem.
\end{proof}

Lemma \ref{normal_approximation_lemma} from the introduction has the following
consequence.
\begin{lemma}\label{continuous_approx_lemma}
 Let $p \geq 3$ be a prime, let $1 \leq k \leq \frac{\log p}{\log \log
p}$ and let $A\in \sA(p,k)$ with
$\Lambda < \zed^k$ any representative of $\Lambda(A) \in \sL^0(p,k)$.
There is a function $\epsilon : \bR_{>0} \to \bR_{>0}$ with $\lim_{x \to
\infty} \epsilon(x)=0$, such that, for $n \geq 1$
\[
 \left\|\mu_{A}^{*n} - \bU_{\zed/p\zed} \right\|_{\TV(\zed/p\zed)} =
\left\|\Theta\left(\cdot, \frac{2n}{2k + 1};
\Lambda\right) - \bU_{\bR^k/\Lambda}\right\|_{\TV(\bR^k/\Lambda)} +
O\left(\epsilon\left(\frac{n}{k^2 }\right)
\right).
\]
\end{lemma}

\begin{proof}
Write $\Lambda = \Lambda(A)$ and $\one_{\left[-\frac{1}{2},
\frac{1}{2}\right)^k}$ for the
indicator function of the cube $\left[-\frac{1}{2}, \frac{1}{2}\right)^k \subset
\bR^k$.
We have
\[
 \left\|\mu_{A}^{*n} - \bU_{\zed/p\zed} \right\|_{\TV(\zed/p\zed)} =
\left\|\nu_{\Lambda}^{*n} - \bU_{\zed^k/\Lambda}\right\|_{\TV(\zed^k/\Lambda)} =
\left\|\nu_{\Lambda}^{*n}\ast \one_{\left[-\frac{1}{2}, \frac{1}{2}\right)^k} -
\bU_{\bR^k/\Lambda}\right\|_{\TV(\bR^k/\Lambda)}
\]
and
\begin{align*}
\Biggl|\left\|\nu_{\Lambda}^{*n}\ast \one_{\left[-\frac{1}{2},
\frac{1}{2}\right)^k} -
\bU_{\bR^k/\Lambda}\right\|_{\TV(\bR^k/\Lambda)}& - \left\|\Theta\left(\ux,
\frac{2n}{2k + 1};
\Lambda\right) - \bU_{\bR^k/\Lambda}\right\|_{\TV(\bR^k/\Lambda)} \Biggr| \\&
\leq
 \left\| \nu_{ \Lambda}^{*n} \ast \one_{\left[-\frac{1}{2},
\frac{1}{2}\right)^k} -
\Theta\left(\ux, \frac{2n}{2k + 1}; \Lambda\right)\right\|_{\TV(\bR^k/\Lambda)}
\\& \leq \left\| \nu^{*n} \ast \one_{\left[-\frac{1}{2}, \frac{1}{2}\right)^k} -
\eta_k\left(\sqrt{\frac{2n}{2k+1}}, \cdot\right)\right\|_{\TV(\bR^k)}
\end{align*}
by two applications of the triangle inequality.  The bound now follows from
Lemma \ref{normal_approximation_lemma}.

\end{proof}

Combining the pieces above we prove the following lemma which is the main
reduction in this section.
\begin{lemma}\label{continuous_replacement_lemma}
 Let $0 < \epsilon <1$, and let $k = k(p)$ satisfy $1 \leq k \leq
\frac{\log p}{\log\log p}$. For
any set $A \in \sA(p,k)$ with uniform measure $\mu_A$ of total variation
mixing time $t^{\mix}_1(\epsilon)$, we have, as $p \to \infty$, for all $n \geq
t_1^{\mix}(\epsilon)$
\[
\left\|\mu_{A}^{*n}(x) - \bU_{\zed/p\zed}  \right\|_{\TV(\zed/p\zed)} =
\left\|\Theta\left(\ux, \frac{2n}{2k + 1};
\Lambda(A)\right) - \bU_{\bR^k/\Lambda(A)}
\right\|_{\TV(\bR^k/\Lambda(A))} + o_\epsilon(1).
\]
\end{lemma}
\begin{proof}
 By Minkowski's geometry of numbers, the shortest non-zero vector in the dual
lattice $\Lambda(A)^\vee$ has length
\[
 \ell(A) \ll \sqrt{k}p^{\frac{-1}{k}}
\]
so that Lemmas \ref{t_mix_lower} and \ref{mixing_time_lower_bound_lemma} give
for the discrete walk $t^{\mix}_1(\epsilon) \gg t^{\rel} \gg
p^{\frac{2}{k}}$.  The claim now follows from
Lemma \ref{continuous_approx_lemma}, since $k = o\left(p^{\frac{1}{k}}\right)$.
\end{proof}

\section{Mixing time estimates}\label{mixing_section}
Let $p$ be prime, $A \in \sA(p)$ with $|A| = 2k+1$ and $1
\leq k \leq \frac{\log p}{\log\log p}$.  Let $\Lambda = \Lambda(A)$ be any
lattice associated to $A$ in $\zed^k$, as above.

\begin{proof}[Proof of Theorem \ref{spectral_mixing_time_theorem}]
Theorem \ref{spectral_mixing_time_theorem} is contained in the set of estimates
 \[
\tau_0\frac{2k+1}{16 \pi \Gamma\left(\frac{k}{2}+1 \right)^{\frac{2}{k}}}p^{\frac{2}{k}} \lesssim_p \tau_0
t^{\rel} \lesssim_p
t^{\mix}_1 \lesssim_k 0.163 k t^{\rel}.
\]
since \[\frac{2k+1}{16 \pi \Gamma\left(\frac{k}{2}+1 \right)^{\frac{2}{k}}} \to \frac{e}{4\pi}, \qquad k \to \infty.\] 

Combining Lemma \ref{mixing_time_lower_bound_lemma} and Minkowski's theorem gives
\[
 t^{\rel} \sim \frac{2k+1}{4\pi^2 \ell(A)^2} \geq \frac{2k+1}{16 \pi \Gamma\left(\frac{k}{2}+1 \right)^{\frac{2}{k}}}p^{\frac{2}{k}}. 
\]

The estimate $t^{\rel}(1-\log 2)\leq t^{\mix}_1$ is given in Lemma
\ref{t_mix_lower}. To replace $(1-\log 2)$ with the larger constant $\tau_0$, consider the theta function $\Theta\left(\ux, \frac{2t}{2k+1}; \Lambda\right)$, which has asymptotically the same relaxation time as $\mu_A$ by Lemma \ref{mixing_time_lower_bound_lemma}. Let $\lambda^*$ be a shortest non-zero vector in the dual space, and consider
\[
 \Theta_0\left(\ux, \frac{2t}{2k+1}; \Lambda\right) = \frac{1}{p} \sum_{j\in \zed} \exp\left(\frac{-4\pi^2 t}{2k+1} \|\lambda^*\|_2^2 j^2 \right)e(j \lambda^* \cdot \ux),
\]
which is found by projecting $\Theta$ in frequency space onto the line determined by $\lambda^*$. 
Equivalently, identify $\bR^{k-1}$ with $\bR^k \cap (\lambda^*)^\perp$ and let $\eta_{k-1}(T,\cdot)$ denote a Gaussian of covariance matrix $T^2 I$ on this space.  Write $\lambda \in \Lambda^\vee$ as $\lambda = \lambda_1 + \lambda_2$ where $\lambda_1$ is the projection to the span of $\lambda^*$ and $\lambda_2$ is orthogonal to $\lambda^*$.  One has, for $T>0$,
\begin{align*}
&\int_{\bR^k \cap (\lambda^*)^\perp}\eta_{k-1}(T,\uy) \Theta\left(\ux + \uy, \frac{2t}{2k+1}; \Lambda\right)d\uy \\&= \sum_{\lambda \in \Lambda^\vee}\exp\left(-\frac{4\pi^2t}{2k+1}\|\lambda\|_2^2 - 2\pi^2 T^2 \|\lambda_2\|_2^2 \right)e(\lambda_1 \cdot x)
\end{align*}
and thus
\[
 \Theta_0\left(\ux, \frac{2t}{2k+1}; \Lambda\right) = \lim_{T \to \infty}\int_{\bR^k \cap (\lambda^*)^\perp}\eta_{k-1}(T,\uy) \Theta\left(\ux + \uy, \frac{2t}{2k+1}; \Lambda\right)d\uy.
\]
The convergence is uniform in $\ux$ as the error at $T$ is dominated by the case in which $\ux$ is orthogonal to $\lambda^*$ so that all the terms are positive.    This justifies exchanging the limit and integral in the following calculation. Let $\sF$ be a fundamental domain for $\bR^k/\Lambda$.
\begin{align*}&
\left\|\Theta_0\left(\ux, \frac{2t}{2k+1}; \Lambda\right) - \bU_{\bR^k/\Lambda}\right\|_{\TV(\bR^k/\Lambda)}= \frac{1}{2}\int_{\sF} \left| \Theta_0\left(\ux, \frac{2t}{2k+1}; \Lambda\right)-\frac{1}{p}\right|d\ux\\
&=\lim_{T\to \infty}\frac{1}{2}\int_{\sF}\left|\int_{\bR^k \cap (\lambda^*)^\perp}\eta_{k-1}(T, \uy)\left(\Theta\left(\ux+\uy, \frac{2t}{2k+1};\Lambda \right) -\frac{1}{p}\right)d\uy\right|d\ux
\end{align*}
Applying the triangle inequality,
\begin{align*}
\left\|\Theta_0 - \bU_{\bR^k/\Lambda}\right\|_{\TV(\bR^k/\Lambda)}&\leq \lim_{T \to \infty}\frac{1}{2}\int_{\sF}\int_{\bR^k \cap (\lambda^*)^\perp}\eta_{k-1}(T, \uy)\left|\Theta\left(\ux+\uy, \frac{2t}{2k+1};\Lambda \right) - \frac{1}{p}\right|d\uy d\ux\\
&= \lim_{T \to \infty}\int_{\bR^k \cap (\lambda^*)^\perp}\eta_{k-1}(T, \uy) \left\|\Theta - \bU_{\bR^k/\Lambda}\right\|_{\TV(\bR^k/\Lambda)} d\uy\\& =  \left\|\Theta - \bU_{\bR^k/\Lambda}\right\|_{\TV(\bR^k/\Lambda)}.
\end{align*}

Let \[\theta(x,t) = \sum_{j \in \zed} \exp(-2\pi^2 tj^2) e(jx)\] denote the time $t$ Gaussian diffusion on $\bR/\zed$.  For $t > 0$,
\[
 \left\|\Theta_0\left(\cdot, \frac{t}{\|\lambda^*\|_2^2}; \Lambda\right) - \bU_{\bR^k/\Lambda}\right\|_{\TV(\bR^k/\Lambda)} = \|\theta(\cdot, t) - \bU_{\bR/\zed}\|_{\TV(\bR/\zed)}.
\]
Since the latter distance is monotonically decreasing and smooth, and since for $n \geq t_1^{\mix}$, \[\left\|\mu_A^{*n} - \bU_{\zed^k/\Lambda}\right\|_{\TV(\zed^k/\Lambda)} = \left\|\Theta\left(\cdot, \frac{2n}{2k+1}; \Lambda\right) - \bU_{\bR^k/\Lambda}\right\|_{\TV(\bR^k/\Lambda)} + o(1)\] by Lemma \ref{continuous_replacement_lemma}, it follows that $t_1^{\mix} \gtrsim \tau_0 t^{\rel}$.

To give the spectral upper bound for $t^{\mix}_1$, again consider instead the
distance from uniformity of $\Theta\left(\cdot, \frac{2n}{2k+1};\Lambda\right)$
on $\bR^k/\Lambda$.  
For $t > 0$,
\begin{align}\notag
 &\left\|\Theta\left(\cdot, \frac{2t}{2k+1};\Lambda\right) -
\bU_{\bR^k/\Lambda}\right\|_{\TV(\bR^k/\Lambda)}^2 \leq
\frac{1}{4}\sum_{\lambda \in \Lambda^{\vee}\setminus\{0\}}
\exp\left(-\frac{8\pi^2 t \|\lambda\|_2^2}{2k+1} \right)
\end{align}
Writing the sum as a Stieltjes integral, then integrating by parts, the right hand side becomes
\begin{align}\label{L_2_estimate}&
\frac{1}{4}\int_{s=1^-}^{\infty} \exp\left(-\frac{8\pi^2 t\|\lambda^*\|_2^2
s^2
}{2k+1} \right)d\left(\left|\Lambda^\vee \cap B_2(0,
s\|\lambda^*\|) \right| \right)\\\notag
&= \frac{4\pi^2 t \|\lambda^*\|_2^2}{2k+1}\int_{1-}^\infty s
\exp\left(-\frac{8\pi^2 t\|\lambda^*\|_2^2 s^2
}{2k+1}
\right) \left|\Lambda^\vee \cap B_2(0,
s\|\lambda^*\|) \right|ds.
\end{align}
Set $t = \tau \frac{2k+1}{4\pi^2 \|\lambda^*\|_2^2}$ so that $\tau \sim
\frac{t}{t^{\rel}}.$  Thus (\ref{L_2_estimate}) simplifies to
\begin{align*}
 (\ref{L_2_estimate}) &= \tau \int_{1^-}^{\infty}s
\exp\left(-2\tau s^2 \right) \left|\Lambda^\vee \cap B_2(0,
s\|\lambda^*\|) \right|ds\\
& \leq \tau\int_{1^-}^{\infty}s
\exp\left(-2\tau s^2 + (1 + \varepsilon(k))kF(s)\right)ds
\end{align*}
where $\varepsilon(k) \to 0$ as $k \to \infty$, and
\[
 F(s) = \left[\frac{1 + \sin \theta}{2\sin
\theta}
\log \frac{1 + \sin \theta}{2 \sin \theta} - \frac{1 - \sin \theta}{2 \sin
\theta} \log \frac{1 - \sin \theta}{2 \sin \theta}  \right], \qquad \theta(s) =
2
\sin^{-1} \left(\frac{1}{2s} \right)
\] see Lemma
\ref{lattice_point_bound_lemma}.  The maximum of $\frac{F(s)}{s^2}$ in $s \geq
1$ occurs at $s = 1.260816271(1)$ with maximum $< 0.324908241$ and
$\frac{F(s)}{s^2} \to
0$ as $s \to \infty$.  Thus, choosing $2\tau = (0.325 +
\tilde{\varepsilon}(k))k$
for an appropriate function $\tilde{\varepsilon}(k)$ tending to 0 as $k \to
\infty$ the $L^2$ distance is negligible so that $\tau t^{\rel}$ is an upper
bound for $t^{\mix}_2 \geq t^{\mix}_1$.
\end{proof}

\subsection{Geometric mixing time bound, proof of Theorem
\ref{mixing_time_theorem}}\label{geometric_mixing_section} Let $p$, $A$ and
$\Lambda$ as above, and let $\sF$ be the Voronoi cell for $\bR^k/\Lambda$.  Note that $\zed^k \cap \overline{\sF}$
contains a system of representatives for $\zed^k/\Lambda$, and that the Cayley
graph $\sC(A,p)$ is isomorphic to $\sC(\{0, \pm e_i: 1 \leq i \leq k\},
\zed^k/\Lambda)$.  Thus
\[
  \rad(\sF) := \sup\left\{\|\ux\|_2: \ux \in \sF\right\}=
\diam_{\geom}(\sC(A,p)).
\]

\begin{proof}[Proof of Theorem \ref{mixing_time_theorem}]
 Write $D = \diam_{\geom}(\sC(A,p))$ and assume, as we may, that $t > kD^2$.
In view of Lemma \ref{diameter_lemma}, which proves $D \geq \frac{1}{\ell(A)}$, we have $t \gg t^{\rel}$, and thus as in Lemma \ref{continuous_replacement_lemma} \[\left\|\mu_A^{*t} -
\bU_{\zed/p\zed}\right\|_{\TV(\zed/p\zed)} +o(1) = \left\|\Theta\left(\cdot,
\frac{2t}{2k+1};\Lambda\right) -
\bU_{\bR^k/\Lambda}\right\|_{\TV(\bR^k/\Lambda)}, \]
so we will estimate the right hand side.

Since, for any $\ux$, $t$, $\E_{\uy \in \sF}\left[\Theta\left(\ux + \uy, \frac{2t}{2k+1}; \Lambda\right)  \right] = \frac{1}{p}$, we may estimate using the triangle inequality 
\begin{align*}
 &\left\|\Theta -
\bU_{\bR^k/\Lambda} \right\|_{\TV(\bR^k/\Lambda)}
= \frac{1}{2}\int_{\ux \in \sF}\left|\Theta\left(\ux, \frac{2t}{2k+1}; \Lambda\right) -\E_{\uy \in \sF}\left[\Theta\left(\ux+\uy, \frac{2t}{2k+1}; \Lambda\right) \right]\right| d\ux
\\& \leq \frac{1}{2} \int_{\ux \in \sF} \sum_{\lambda \in \Lambda} \left|\eta_k\left(\sqrt{\frac{2t}{2k+1}}, \ux-\lambda \right) - \E_{\uy \in \sF}\left[\eta_k\left(\sqrt{\frac{2t}{2k+1}}, \ux+\uy-\lambda \right) \right] \right| d\ux.
\end{align*}
Now use the inequality $|1 -e^x| \leq e^{|x|}-1$
to obtain
\begin{align*}
&\left\|\Theta -
\bU_{\bR^k/\Lambda} \right\|_{\TV(\bR^k/\Lambda)}
\\
& \leq \frac{1}{2} \int_{\ux \in \sF}\sum_{\lambda \in \Lambda} \eta_k\left(\sqrt{\frac{2t}{2k+1}}, \ux-\lambda \right) \E_{\uy \in \sF}\left[\exp\left(\frac{2k+1}{4t}\left(\|\uy\|_2^2 + 2|\langle \ux - \lambda, \uy\rangle | \right) \right)-1 \right] d\ux.
\end{align*}
Fold together the sum over $\lambda$ and the integral over $\ux$, then integrate away all directions in $\ux$ orthogonal to $\uy$ to obtain
\begin{align*}
  &\left\|\Theta\left(\cdot,
\frac{2t}{2k+1};\Lambda\right) -
\bU_{\bR^k/\Lambda} \right\|_{\TV(\bR^k/\Lambda)}\\
& \leq \frac{1}{2} \int_{x \in \bR} \eta_1 \left(\sqrt{\frac{2t}{2k+1}}, x \right) \E_{\uy \in \sF} \left[\exp\left(\frac{2k+1}{4t}\left(\|\uy\|_2^2 + \|\uy\|_2 |x|\right)\right)  - 1\right] dx\\
& \ll D \sqrt{\frac{k}{t}}.
\end{align*}
The last estimate follows on using $\frac{1}{\sqrt{2\pi}} \int_{x \in \bR} e^{-\frac{x^2}{2} + \delta |x|} dx = 1 + O(\delta)$ as $\delta \downarrow 0$.
\end{proof}

\section{Transition window bound, proof of Theorem
\ref{bounded_gen_set_theorem}}
We prove the following somewhat more general theorem.
\begin{theorem}\label{derivative_theorem}
 Let $p$ be a large prime and let $k \leq  \frac{\log
    p}{\log \log p}$.  Let $A \subset \zed/p\zed$ be a lazy symmetric
generating set of size $|A| = 2k+1$.  For any $1>\epsilon_1> \epsilon_2>0$,
for all $n <
\exp\left(\frac{2\epsilon_2}{ k}\right)\cdot t^{\mix}_1(\epsilon_1)$ we have
\[\left\|\mu_A^{*n} - \bU_{\zed/p\zed}\right\|_{\TV(\zed/p\zed)} \geq
  \epsilon_1 - \epsilon_2+ o_{\epsilon_1,\epsilon_2}(1).\]
\end{theorem}

\begin{proof}
Let $ \Lambda< \zed^k$ be any lattice representing the class of $\Lambda(A) \in
\sL$. By Lemma \ref{continuous_replacement_lemma} we may replace
$\left\|\mu_A^{*n} - \bU_{\zed/p\zed}\right\|_{\TV(\zed/p\zed)} $ with
$\left\|\Theta\left(\ux, \frac{2n}{2k + 1};
\Lambda\right) - \bU_{\bR^k/\Lambda}
\right\|_{\TV(\bR^k/\Lambda)} $ making error $o(1)$.

Write $n = \sigma t^{\mix}_1(\epsilon_1)$.  Differentiating under the sum in the $\theta$ function, 
\begin{align}\label{time_derivative}\frac{d}{d\sigma}\Theta\left(\ux, \frac{2 \sigma
t^{\mix}_1(\epsilon_1)}{2k+1}; \Lambda \right)\Bigg|_{\sigma = \sigma'} &\geq -
\frac{k}{2\sigma'} \Theta\left(\ux, \frac{2 \sigma'
t^{\mix}_1(\epsilon_1)}{2k+1}; \Lambda \right) .
 \end{align}
Also, $\left\|\Theta\left(\ux, \frac{2 \sigma
t^{\mix}_1(\epsilon_1)}{2k+1};\Lambda \right) - \bU_{\bR^k/\Lambda}
\right\|_{\TV(\bR^k/\Lambda)}$ is a
decreasing function of $\sigma>0$. Define
\[
P(\sigma) =  \left\{\ux \in \bR^k/\Lambda:
 \Theta\left(\ux,
\frac{2 \sigma
t^{\mix}_1(\epsilon_1)}{2k+1};\Lambda \right) > \frac{1}{p}\right\}.
\]

Now for any $\sigma, \sigma_0 >0$,
\begin{align*}\left\|\Theta\left(\cdot, \frac{2 \sigma
t^{\mix}_1(\epsilon_1)}{2k+1};\Lambda \right) -
\bU_{\bR^k/\Lambda}\right\|_{\TV(\bR^k/\Lambda)} &= \int_{P(\sigma)}
\Theta\left(\ux,
\frac{2 \sigma
t^{\mix}_1(\epsilon_1)}{2k+1};\Lambda \right) - \frac{1}{p} d\ux \\&\geq
\int_{P(\sigma_0)}
\Theta\left(\ux,
\frac{2 \sigma
t^{\mix}_1(\epsilon_1)}{2k+1};\Lambda \right) - \frac{1}{p} d\ux.\end{align*}
Thus for $\sigma > \sigma_0$,
\begin{align}
\notag &\left\|\Theta\left(\cdot, \frac{2 \sigma
t^{\mix}_1(\epsilon_1)}{2k+1};\Lambda \right) -
\bU_{\bR^k/\Lambda}\right\|_{\TV(\bR^k/\Lambda)} -\left\|\Theta\left(\cdot,
\frac{2 \sigma_0
t^{\mix}_1(\epsilon_1)}{2k+1};\Lambda \right) -
\bU_{\bR^k/\Lambda}\right\|_{\TV(\bR^k/\Lambda)} \\ &\label{P_sigma} \qquad\qquad\geq
\int_{P(\sigma_0)} \Theta\left(\ux, \frac{2 \sigma
t^{\mix}_1(\epsilon_1)}{2k+1};\Lambda \right)
-
\Theta\left(\ux, \frac{2 \sigma_0 t^{\mix}_1(\epsilon_1)}{2k+1};\Lambda \right)
d\ux
\end{align}
Differentiate under the integral, then apply (\ref{time_derivative}) and finally drop the restriction to $P(\sigma_0)$ to obtain the estimate
\begin{align}
 \notag (\ref{P_sigma})  &= \int_{P(\sigma_0)} \int_{\sigma_0}^\sigma \frac{d}{ds}
\Theta\left(\ux, \frac{2 s t^{\mix}_1(\epsilon_1)}{2k+1};\Lambda
\right)\Bigg|_{s = \sigma'} d\sigma' d\ux
\\\notag &  \geq -\frac{k}{2}\int_{\sigma_0}^\sigma
\frac{1}{\sigma'} \int_{P(\sigma_0)} \Theta\left(\ux,\frac{2 \sigma'
t^{\mix}_1(\epsilon_1)}{2k+1};\Lambda
\right) d\ux  d\sigma'
\\& \label{deriv_bound}  \geq \frac{-k}{2}\log \frac{\sigma}{\sigma_0}.
\end{align}

Note that $k = o\left(t^{\mix}_1(\epsilon_1)\right)$.
Applying (\ref{deriv_bound}) with $\sigma_0 = 1 -
\frac{1}{t^{\mix}_1(\epsilon_1)}$ and
$\sigma =
1$, which corresponds to the random walk at the mixing time and the step
before, we deduce
\[
 \left\|\Theta\left(\cdot, \frac{2
t^{\mix}_1(\epsilon_1)}{2k+1};\Lambda \right) -
\bU_{\bR^k/\Lambda}\right\|_{\TV(\bR^k/\Lambda)} = \epsilon_1 +
o_{\epsilon_1}(1).
\]
Applying (\ref{deriv_bound}) again, but
now with $\sigma_0 = 1$, $\sigma = \exp(\frac{2\epsilon_2}{ k})$, we obtain
in the range $t^{\mix}_1(\epsilon_1) < n
< \exp(\frac{2\epsilon_2}{ k})\cdot t^{\mix}_1(\epsilon_1)$,
\begin{align*}\left\|\mu_A^{(n)} - \bU_{\zed/p\zed}\right\|_{\TV(\zed/p\zed)}
\geq
\epsilon_1 - \epsilon_2 +o_{\epsilon_1, \epsilon_2}(1).\end{align*}
\end{proof}
\section{Random random walk, proof of Theorem \ref{random_theorem}}
 We record several facts regarding the uniform measure $\bU_L$ on the set
$L(p,k)$ of index $p$ lattices in $\zed^k$.  
\begin{lemma}\label{dual_lattice_lemma}
 When $\Lambda$ is chosen uniformly from $L(p,k)$, the dual lattice
$\Lambda^\vee$ has the distribution of
\[
 \{0, 1, \cdots, p-1\} \frac{\uv}{p} + \zed^k
\]
where $\uv$ is a uniform random vector in $(\zed/p\zed)^k \setminus \{0\}$.

When $\Lambda$ is chosen uniformly from $L^0(p,k)$, the dual lattice
$\Lambda^{\vee}$ has the distribution of
\[
\{0, 1, \cdots, p-1\}\frac{\uv}{p} + \zed^k
\]
where $\uv$ is chosen uniformly from
\[
\sD = \{\uv \in (\zed/p\zed \setminus \{0\})^k: \forall 1 \leq i < j \leq k,
v_i \neq \pm v_j \}.
\]
\end{lemma}
\begin{proof}
 In the case of $L(p,k)$, the structure follows from $[\Lambda:\zed^k] = p$ and
$\frac{1}{p}\zed^k<\Lambda$, while the uniformity
follows from the fact that $\SL_k(\zed/p\zed)$ acts transitively on the space of dual lattices.  This holds since any non-zero vector may be completed to a basis for $(\zed/p\zed)^k$.

The further conditions imposed in the case of $L^0(p,k)$ are those necessary to
ensure that $\Lambda$ does not contain a vector $\lambda$ with $\|\lambda\|_2^2
\in \{1,2\}.$
\end{proof}

\begin{lemma}\label{expectation_lemma}
 Let $p$ be prime, let $k \geq 2$ and let $\uv \neq \uw \in \zed^k$. We have
\[
 \bU_L(\Lambda: \uv, \uw \in \Lambda) = \left\{ \begin{array}{lll}
1 && \uv, \uw \in (p\zed)^k\\
\frac{p^{k-1}-1}{p^k-1}  &&  |\zed \uv + \zed \uw \bmod p| = p\\
  \frac{p^{k-2}-1}{p^k-1}&& |\zed \uv + \zed \uw \bmod
p| = p^2
  \end{array}\right..
\]
In particular, $\bU_L(L^0(p,k)) \geq 1 - O\left(\frac{k^2}{p} \right).$
\end{lemma}
\begin{proof}
These follow immediately from the distribution of the dual group.
\end{proof}

\subsection{Summary of argument}
As the calculations in the remainder of this section are somewhat involved, we
pause to sketch the main ideas.

Theorem \ref{random_theorem} has three claims, the first two of which
consider the worst case mixing time behavior, with the third considering
typical behavior. When considering the walk as a diffusion on
$\bR^k/\Lambda$
where $\Lambda$ is a lattice, the spectrum of the transition kernel is
determined by the dual lattice $\Lambda^{\vee}$.  In general, it is
difficult to work on the spectral side due to the high concentration of
eigenvalues near the spectral gap, but in the worst case regime we are able to
show that for all behavior that persists, the dual lattice is essentially one
dimensional. When this occurs the mixing and relaxation times are
proportional and we obtain a slow transition.

In typical behavior the walk has a sharp transition to uniformity.  The analysis
in this regime consists of separate arguments estimating the
distance to uniformity at times $(1 \pm \epsilon)t^{\mix}_1$.  When considering
the walk at time $(1-\epsilon)t^{\mix}_1$ we study the diffusion
$\Theta\left(\ux, \frac{2t}{2k+1};\Lambda\right)$ on the norm-minimal
fundamental domain $\sF(\Lambda)$ for $\bR^k/\Lambda$.  For a particular lattice
$\Lambda$, $\sF(\Lambda)$ is a highly complex convex body determined by a number
of hyperplanes, but in a statistical sense, for the purpose of the lower bound,
$\sF(\Lambda)$ behaves very much like the volume $p$ ball of $\bR^k$ centered
at the origin.  A Gaussian in $\bR^k$ centered at the origin is
concentrated on a thin spherical shell (see Lemma
\ref{gaussian_length_concentration}), and the mixing time is essentially the
time needed for this spherical shell to expand to the boundary of the volume
$p$ ball.  At time $(1-\epsilon)t^{\mix}_1$ we are then
able to show that the diffusion is typically concentrated on a small measure
part of $\sF(\Lambda)$.

For the upper bound at time $(1+\epsilon)t^{\mix}_1$, we note that $p\zed^k <
\Lambda$, and we show that the distribution of values of
$\Theta\left(\ux, \frac{2t}{2k+1};\Lambda\right)$ is concentrated near 1
when $\ux$ is chosen uniformly from $\bR^k/p\zed^k$ and $\Lambda$ is chosen at
random from $L(p,k)$.  This is the most delicate part of the argument.  For
instance, it is not sufficient to consider the expectation of
$\left(\Theta\left(\ux, \frac{2t}{2k+1};\Lambda\right)-1\right)^2$ as
this gives an upper bound which is too weak, so we split $\Theta$ into an
$L^2$-concentrated piece $\Theta_M$ plus a small $L^1$ error $\Theta_E$.

\subsection{Slow mixing behavior}
We prove 
Theorem \ref{random_theorem} in two parts.  In this section we prove parts (1) and (2) which concern rare slow mixing walks.  In Section \ref{typical_section} we prove part (3) regarding the typical behavior.  The main estimate regarding slow mixing behavior is the following theorem.
\begin{theorem}\label{random_no_cut_off_theorem}
 Let $p$ be a large prime, and let $k = k(p)$ tending to $\infty$ with $p$ in
such a way that $k \leq \frac{\log
p}{\log \log p}$. For any $\delta > 0$, for all $p$ sufficiently large,
uniformly in  $\delta \frac{p^{\frac{1}{k}}}{\sqrt{k}} < \rho < \frac{(p
\log p)^{\frac{1}{k}}}{\delta}$, the following hold
\begin{enumerate}
 \item \begin{equation}
\label{small_vector_prob} \Prob_{\sA(p,k)}\left[t^{\rel} \geq  \frac{
e \rho^2 p^{\frac{2}{k}}}{\pi } \right] = \frac{\exp(o(k))}{\rho^k}.
\end{equation}
\item Let, as in Theorem \ref{random_theorem}, $\tau_0$ be the ratio
between total variation mixing time and relaxation time for Gaussian diffusion
on $\bR/\zed$.  For any $C \geq 1$, and $\frac{ \delta
p^{\frac{4}{k}}}{k}
\leq J \leq \frac{p^{\frac{4}{k}}(\log p)^{\frac{2}{k}}}{\delta}$
\begin{align}
\label{small_mixing_time_prob} \Prob_{\sA(p,k)}\left[t^{\mix}_1
\geq C(\tau_0 + \delta) t^{\rel} \text{ and } \frac{J}{2} \leq t^{\rel}\leq
J\right] &\leq
\exp\left(\frac{k}{2} \log \frac{k}{C} +
O_\delta(k)\right)\frac{p^2}{J^{k}}\\
\notag \Prob_{\sA(p,k)}\left[t^{\mix}_1
\leq (\tau_0 - \delta) t^{\rel} \text{ and }  \frac{J}{2} \leq t^{\rel}\leq
J\right]
&\leq \exp\left(\frac{k}{2} \log k +
O_\delta(k)\right)\frac{p^2}{J^{k}}.
\end{align}
\end{enumerate}
\end{theorem}

\begin{proof}[Deduction of Theorem \ref{random_theorem}, parts (1) and (2)]
Before giving the proof of the Theorem we prove an auxiliary claim.

Let $\delta>0$ be an arbitrarily small fixed quantity.  We claim that
with probability 1, only finitely many of the events
\[
B_p = \left\{ t^{\mix}_1(p) \geq \delta p^{\frac{4}{k}} \text{ and }
\left| \frac{t^{\mix}_1(p)}{\tau_0 t^{\rel}(p)} -1\right| \geq \delta
\right\}
\]
occur.  Note that by Theorem
\ref{spectral_mixing_time_theorem}, $t^{\mix}_1(p) \geq \delta p^{\frac{4}{k}}$
implies $t^{\rel}(p) \gg \delta\frac{p^{\frac{4}{k}}}{k}$.  Thus, combining
(\ref{small_vector_prob}) and (\ref{small_mixing_time_prob}),
\[
 \Prob\left(B_p\right) \leq \frac{\exp\left(k \log k +
O_{\delta}(k)\right)}{p^2} + \frac{1}{p\log p},
\]
where the first term is handled with (\ref{small_mixing_time_prob}) and covers the range $t^{\rel} \ll p^{\frac{4}{k}}(\log
p)^{\frac{2}{k}}$, the worst case occuring when $t^{\rel} \ll
\frac{p^{\frac{4}{k}}}{k}$ is minimized.
Note $k \leq \frac{\log p}{\log \log p}$ from which it follows
\[
\sum_p \Prob\left(B_p\right) \leq \sum_p \left(\frac{\exp\left(-\frac{\log
p}{\log \log p}\left(\log \log \log p + O_{\delta}(1) \right)
\right)}{p} + \frac{1}{p\log p}\right) < \infty,
\]
so that the claim holds by the Borel-Cantelli Lemma.

We now prove the Theorem.

(1) Replace $\rho(p)$ with $\rho(p):= \max\left(\rho(p),
p^{\frac{1}{k}}\right)$ without altering the divergence of $\sum_p \rho(p)^{-k}$.
Estimating with (\ref{small_vector_prob}), by Borel-Cantelli, with probability 1
there is an infinite sequence $\sP_0 \subset \sP$ such that, for $p \to \infty$
through $\sP_0$,
\[
t^{\rel}(p) \gtrsim \frac{e}{\pi}\rho(p)^2 p^{\frac{2}{k}}.
\]
The above remarks guarantee that, for this sequence,
$t^{\mix}_1(p) \sim \tau_0 t^{\rel}(p)$.

(2) Let $\delta > 0$ be fixed.  Estimate with (\ref{small_vector_prob}) to
obtain that with probability 1, for all but finitely many $p$,
\[
 t^{\rel}(p) \leq \left(1 + \frac{\delta}{2}\right)\frac{e}{\pi}\rho(p)^2
p^{\frac{2}{k}}.
\]
Since $\rho(p) \geq p^{\frac{1}{k}}$ eventually, the remarks above imply that
with probability 1
\[
 t^{\mix}_1(p) \leq (1 + \delta)\frac{e \tau_0}{\pi}\rho(p)^2
p^{\frac{2}{k}}
\] for all but finitely many $p$.

\end{proof}

In proving Theorem \ref{random_no_cut_off_theorem} we introduce two commonly used pieces of terminology from the theory of lattices.  Let $p$ be a prime and
let $k \geq 1$. Say that $\lambda \in \zed^k$ is \emph{reduced} (at $p$) if
$\lambda \in \left[ - \frac{p}{2}, \frac{p}{2}\right)^k$.  Any class $\lambda
\in (\zed/p\zed)^k$ has a unique reduced representative $r(\lambda) \in
\zed^k$.  Say that $\lambda = (\lambda_1, ..., \lambda_k) \in \zed^k$ is
\emph{primitive} if $\lambda \neq 0$ and $\GCD(\lambda_i: 1 \leq i \leq k) = 1$.

Our proof of Theorem \ref{random_no_cut_off_theorem} depends upon the
following two estimates, the first of which estimates a mean concerning pairs
of short vectors in the dual space.  
\begin{proposition}\label{pair_short_vector_prop} Let $\delta > 0$ be a fixed
constant.
 Let $p$ and $k(p)$ tend to $\infty$ in such a way that $k \leq
\frac{\log p}{\log\log p}$. Let
$ \frac{\delta p^{\frac{1}{k}}}{\sqrt{k}} \leq \rho \leq \frac{1}{\delta}(p \log
p)^{\frac{1}{k}}$.
For any $\frac{\delta}{\sqrt{k}} \leq C \leq \frac{\sqrt{k}}{\delta}$, for
any $\epsilon > 0$,
\begin{align}\label{two_short_vector_estimate}
\E_{L^0(p,k)}\left[\sum_{\substack{\lambda_1 \neq
\pm \lambda_2 \in
\Lambda^{\vee}\setminus\{0\}\\p\lambda_i \text{
primitive}}}\eta_k\left(
\frac{1}{\rho p^{\frac{1}{k}}},
\lambda_1 \right)\eta_k\left( \frac{1}{C\rho p^{\frac{1}{k}}},
\lambda_2\right)\right]
\leq p^2 + O_\epsilon\left(p^{\frac{3}{2} +
\frac{4}{k} + \epsilon}\right).&
\end{align}
\end{proposition}

\begin{rem}
 This proposition should be interpretted as expressing the approximate independence of the appearance of a pair of short primitive vectors in the dual space.
\end{rem}

\begin{proof}
It is enough to estimate with respect to
$\bU_{L(p,k)}$ since this introduces a relative error $1 + O\left(\frac{k^2}{p} \right)$, which is smaller than the error claimed. 

Let $\sS \subset  (\zed/p\zed)^k \times
\zed/p\zed$ denote the set of pairs $(\lambda, a)$ such that $\lambda \in
(\zed/p\zed)^k$, $a \in \zed/p\zed\setminus \{0, \pm 1\}$ and both reduced
vectors $r(\lambda)$ and
$r(a\lambda)$ are primitive.  Also denote for $\lambda \in (\zed/p\zed)^k$,
$\sS(\lambda)\subset \zed/p\zed$ the fiber over $\lambda$.

Lemma \ref{dual_lattice_lemma} gives
\begin{align}
& \notag\E_{L(p,k)}\left[\sum_{\substack{\lambda_1 \neq
\pm \lambda_2 \in
\Lambda^{\vee}\setminus\{0\}\\p\lambda_i  \text{
primitive}}}\eta_k\left(\frac{1}{\rho p^{\frac{1}{k}}},
\lambda_1 \right)\eta_k\left(\frac{1}{C\rho p^{\frac{1}{k}}}, \lambda_2
\right)\right]\\
& \leq  \label{first_dual}\frac{p^{2k}(p-1)}{p^k-1} \sum_{\substack{\lambda \in
\left(\zed \cap \left(-\frac{p}{2}, \frac{p}{2}\right]\right)^k\\
\text{primitive}}} \sum_{a \in \sS(\lambda)}
\Phi_1(\lambda)\Phi_C(a\lambda) + o(1),
\end{align}
where
\[
 \Phi_c(\ux) = \sum_{\un \in \zed^k} \eta_k\left(\frac{p^{1-\frac{1}{k}}}{c
\rho }, \ux+p\un\right).
\]
To briefly explain this formula, the factor of $p^{2k}$ results from scaling both the variable and the standard deviation in the Gaussians by $p$.  The condition $\lambda_1 \equiv a \lambda_2 \bmod p\zed^k$ for some $a$ follows from the characterization of $\Lambda^\vee$.
The error term $o(1)$ covers summation over pairs $\lambda_1, \lambda_2$
for which one of $\lambda_1, \lambda_2$ is not reduced but both are primitive.  The summation
in this case is bounded by, for some $c > 0$ and all $B>0$
\begin{align}\notag
 &\ll p^{-k+1}\sum_{\substack{\lambda_1, \lambda_2 \in \left(
\frac{1}{p}\zed\right)^k\\ \max(\|\lambda_1\|_\infty, \|\lambda_2\|_\infty >
\frac{1}{2})}} \eta_k\left(\frac{1}{\rho p^{\frac{1}{k}}},
\lambda_1 \right)\eta_k\left(\frac{1}{C\rho p^{\frac{1}{k}}}, \lambda_2
\right)\\&\label{tail_sum} \ll p^{O(k)}\exp\left(-cp^{\frac{2}{k}} \right) =
O_B(p^{-B}),
\end{align}
since $p^{\frac{2}{k}}$ dominates $k \log p$.

We make several modifications to the sum of (\ref{first_dual}) which make it
easier to estimate.
First we may exclude from $\sS$ any pairs $(\lambda, a)$  for which
\[
\max\left(\rho p^{\frac{1}{k}}\left\|\frac{\lambda}{p}\right\|_{(\bR/\zed)^k},
C\rho p^{\frac{1}{k}}\left\| \frac{a\lambda}{p}\right\|_{(\bR/\zed)^k}\right) \geq
Q:= \sqrt{\log p \log \log p}
\]
as these contribute, for any $B>0$, $O_B(p^{-B})$.  To obtain this, note that the cardinality of the summation set is $O(p^{k+1})$ since we have replaced summation over $\lambda_2$ with summation over $a$.  
Thus it suffices to show that for excluded pairs, $\Phi_1(\lambda)\Phi_C(a\lambda) \ll_{B} p^{-2k-2-B}$; to see this, note that $\Phi$ is controlled by the contribution of the summand nearest 0.

Let $\sS'$ be those
choices of $(\lambda, a)$
which remain.  Denote by $\sF(Q)$ the collection of Farey fractions modulo $p$
(the definition is non-standard since the numerator and denominator are bounded
by
different quantities),
\[
\sF(Q) = \left\{b q^{-1} \bmod p: \max\left(|b|, \frac{|q|}{C}\right) \leq \frac{\rho
p^{\frac{1}{k}}}{2Q}, q \neq
0\right\}.
\]
We claim that for any reduced $\lambda$, $\sS'(\lambda \bmod p) \subset
\zed/p\zed \setminus \sF(Q).$ Indeed,
suppose otherwise and let $a= bq^{-1}\in \sS'(\lambda \bmod p) \cap \sF(Q)$. Let
$\eta \equiv a
\lambda \bmod p$ with $\eta$ reduced.  Then $b \lambda \equiv q \eta \bmod p$,
but the norm condition implies that in fact $b \lambda = q\eta$, which
contradicts the primitivity.

Replace $\sS'(\lambda)$ with $\zed/p\zed \setminus \sF(Q)$ and complete the
sum over $\lambda$ to obtain
\[
(\ref{first_dual}) \leq  O_B(p^{-B}) + \frac{p^{2k}(p-1)}{p^k-1}
\sum_{\lambda \in (\zed/p\zed)^k} \sum_{a \in \sF(Q)^c}
\Phi_1(\lambda)\Phi_C(a\lambda).
\]
Applying Plancherel on $(\zed/p\zed)^k$, we obtain
\[
 (\ref{first_dual})<O_B(p^{-B})+ \frac{p^{k}(p-1)}{p^k-1} \sum_{a \in \sF(Q)^c}
\sum_{\xi \in (\zed/p\zed)^k} \hat{\Phi}_1(a \xi)
\overline{\hat{\Phi}_C(\xi)}
\]
where
\begin{align*}
 \hat{\Phi}_c(\xi) &= \sum_{\un \in \zed^k}
\eta_k\left(\frac{p^{1 -\frac{1}{k}}}{c \rho  } , \un\right)
\exp\left(\frac{2\pi i \xi \cdot \un}{p} \right)= \sum_{\un \in \zed^k}
\exp\left(-\frac{2\pi^2  p^{2-\frac{2}{k}}}{c^2
\rho ^2}\left\|\frac{\xi}{p} + \un \right\|_2^2 \right).
\end{align*}
All but one term from the sum over $\un$ is negligible, and we obtain, for any
$\epsilon > 0$,
\begin{align*}
 (\ref{first_dual})=& O_B(p^{-B})\\&+
\frac{p^{k}(p-1)}{p^k-1} \sum_{a \in \sF(Q)^c} \sum_{\xi \in (\zed/p\zed)^k}
\exp\left(-\frac{2\pi^2
p^{2-\frac{2}{k}}}{
\rho ^2}\left( \frac{1}{C^2}\left\|\frac{\xi}{p}  \right\|_{(\bR/\zed)^k}^2+
\left\|\frac{a\xi}{p} \right\|_{(\bR/\zed)^k}^2 \right)\right).
\end{align*}
Due to the decay in the exponential, we may truncate summation over $a$ and
$\xi$ to
$\left\|\frac{\xi}{p}\right\|_{(\bR/\zed)^k}, \left\|\frac{a\xi}{p}\right\|_{
(\bR/\zed)^k}\ll_\epsilon p^{-1 +
\frac{1}{k}
+\epsilon}$ with negligible error.

From $\xi= 0$ pull out a term $
 \sim p^2.$ To treat the remaining terms,
suppose $k \geq 3$, and let $\xi = q \xi^0$ for $q \in \zed_{>0}$ and $\xi_0$
primitive.  Write $a\xi \equiv \zeta \bmod p$ where
$\|\zeta\|_{\bR^k}\ll_\epsilon
p^{\frac{1}{k} +\epsilon}.$ It follows for $1 <i\leq k$, $ \xi_1^0 \zeta_i
\equiv
\xi_i^0 \zeta_1 \bmod p$, and in fact, $\xi_1^0 \zeta_i = \xi_i^0 \zeta_1$ so
$\zeta
= b \xi_0$ for some $b \in \zed$.  The sum is thus bounded by
\begin{align*}
 &\frac{p^{k}(p-1)}{p^k-1} \sum_{\substack{\xi \in \zed^k\\ 1 \leq  \|\xi\|\ll
p^{\frac{1}{k}+\epsilon}}} \sum_{\max( |b|,
\frac{|q|}{C})>\frac{\rho p^{\frac{1}{k}}}{2Q}}
\exp\left(-\frac{2\pi^2
}{
\rho ^2 p^{\frac{2}{k}}}\left( \left(\frac{q^2}{C^2} + b^2\right)\left\|\xi
\right\|_2^2 \right)\right).
\end{align*}
We may estimate this sum crudely by truncating summation over $b, q$ at $|b|,
\frac{|q|}{C} \leq \rho p^{\frac{1}{k}+\epsilon} $ with error
$O_B(p^{-B})$. The total number of such $b, q$ is $\ll k^{O(1)}\rho^2
p^{\frac{2}{k}+2\epsilon} \ll_{\epsilon'} p^{\frac{4}{k}+\epsilon'}$.  For
all such $b, q$, summation over $\xi$ is bounded by (see Lemma \ref{theta_lemma})
\[
p\sum_{0 \neq \xi \in \zed^k} \exp\left(\frac{-\pi^2 \|\xi\|_2^2}{2Q^2} \right)<p\left(\sum_{\xi \in \zed} \exp\left(\frac{-\pi^2 \xi^2}{2Q^2} \right)\right)^k
\ll p (2Q)^k \ll_\epsilon p^{\frac{3}{2} +\epsilon}.
\]

\end{proof}
Next we determine the distribution of the shortest vector in the dual lattice.
Recall that $R_k =
\left(\frac{\Gamma\left(\frac{k}{2}+1\right)}{\pi^{\frac{k}{2}}}
\right)^{\frac{1}{k}}$ is the radius of a volume 1 ball.
\begin{proposition}\label{short_vector_prop}
Let $\delta > 0$ be
a fixed constant, and let $p$, $k$ and $\rho$ be such that $k \leq \frac{\log
p}{\log \log p},$ and $\frac{\delta p^{\frac{1}{k}}}{\sqrt{k}} \leq \rho \leq
\frac{1}{\delta}(p \log p)^{\frac{1}{k}}.$  Given $\Lambda \in L^0(p,k)$ denote
$\lambda^*$ the shortest non-zero vector of the dual lattice. One has
\[
\Prob_{L^0(p,k)}\left[\|\lambda^*\|_2 \leq \frac{R_k}{\rho p^{\frac{1}{k}}}
\right] = \frac{1}{2\rho^k}\left(1 + O \left(\frac{e^{O(k)}}{\rho^k} +
\frac{k^{2}\rho}{p^{1-\frac{1}{k}}} \right)\right).
\]

\end{proposition}
\begin{proof}
By Lemma \ref{point_count_lemma}
\begin{align*}
\E_{L(p,k)}\left[\#\left\{0 \neq \lambda \in \Lambda^\vee \cap B_{2}\left(0,
\frac{R_k}{\rho p^{\frac{1}{k}}} \right) \right\} \right] &=
\frac{p-1}{p^k-1}\#\left\{0 \neq \lambda \in \zed^k \cap B_2\left(0,
\frac{R_k p^{1-\frac{1}{k}}}{\rho}\right) \right\} \\
&= \frac{1}{\rho^k}\left(1 + O\left(\frac{k \rho}{p^{1-\frac{1}{k}}}
\right)\right).
\end{align*}
By counting vectors $\lambda$ with $\lambda_1 = 0$ or $\lambda_1 = \pm
\lambda_2$ one finds
\begin{equation}\label{expected_short_vector}
 \E_{L^0(p,k)}\left[\#\left\{0 \neq \lambda \in \Lambda^\vee \cap B_{2}\left(0,
\frac{R_k}{\rho p^{\frac{1}{k}}} \right) \right\} \right] =
\frac{1}{\rho^k}\left(1 + O\left(\frac{k^{2}\rho}{p^{1-\frac{1}{k}}}
\right) \right).
\end{equation}
Let $0 < \tau < 1$ and observe that for all
$(1-\tau) \frac{\sqrt{k}}{p^{\frac{1}{k}}\rho} < \|\ux\|_2 \leq (1 + \tau)
\frac{\sqrt{k}}{p^{\frac{1}{k}}\rho}$,
\begin{equation}\label{gaussian_size}
  p \rho^k \exp\left(-\frac{k}{2}\left((1 + \tau)^2 + \log 2\pi \right)
\right)\leq \eta_k\left(\frac{1}{p^{\frac{1}{k}} \rho} , \ux\right) \leq
 p \rho^k \exp\left(-\frac{k}{2}\left((1 - \tau)^2 + \log 2\pi \right)\right).
\end{equation}
Choosing $C = 1$ in Proposition \ref{pair_short_vector_prop} and inserting
these bounds, one finds
\[
 \Prob_{L^0(p,k)}\left[\|\lambda^*\|_2 \leq \frac{R_k}{\rho p^{\frac{1}{k}}}
\right] = \frac{1}{2\rho^k}\left(1 +
O\left(\frac{e^{O(k)}}{\rho^k}+\frac{k^{2}\rho}{p^{1-\frac{1}{k}}} \right)
\right),
\]
by subtracting the contribution to (\ref{expected_short_vector}) from lattices
with pairs of primitive short vectors, and accounting for the factor of 2 from
counting $\pm \lambda^*$.

\end{proof}

\begin{proof}[Proof of Theorem \ref{random_no_cut_off_theorem}]
The estimate (\ref{small_vector_prob}) regarding the distribution of $t^{\rel}$
follows from Proposition \ref{short_vector_prop} together with $R_k \sim
\sqrt{\frac{k}{2\pi e}}$ and $t^{\rel}\sim \frac{k}{2\pi^2 \|\lambda^*\|_2^2}$
as $k \to \infty$.

For (\ref{small_mixing_time_prob}),
choose $\rho = 2^{\frac{n}{2}}$  such that $\rho^2 p^{\frac{2}{k}} \asymp J$.
Equivalently, consider $\Lambda$ for which the shortest non-zero vector
$\lambda^*$ of $\Lambda^\vee$ satisfies $\frac{\sqrt{k}}{\rho p^{\frac{1}{k}}}
\asymp \|\lambda^*\|_2$. For such $\lambda^*$,
\begin{equation}\label{gaussian_majorant}
 \eta_k\left(\frac{1}{\rho p^{\frac{1}{k}}}, \lambda^* \right) = p \rho^k
e^{O(k)}.
\end{equation}
This majorant is used in what follows.

Let
\[
\Theta_0\left(\ux, \frac{2t}{2k+1}; \Lambda\right) = \frac{1}{p}\sum_{j \in
\zed}\exp\left(\frac{-4\pi^2 t}{2k+1}\|\lambda^*\|_2^2 j^2 \right) e\left(j
\lambda^* \cdot \ux \right)
\]
denote the projection of $\Theta\left(\ux, \frac{2t}{2k+1}; \Lambda\right)$
in frequency space onto the line determined by $\lambda^*$. If $\left|\frac{t^{\mix}_1}{\tau_0 t^{\rel}} - 1\right| >
\epsilon$ then there is some $t = (1 +O(\epsilon))t_1^{\mix}$ such that
\begin{equation}\label{mixing_perturbation}
\left\|(\Theta - \Theta_0)\left(\cdot, \frac{2 t}{2k+1}; \Lambda\right)
\right\|_{L^1(\bR^k/\Lambda)} \gg_\epsilon 1.
\end{equation}
Apply
Cauchy-Schwarz to obtain
\begin{equation}\label{primitive_lower_bound}
 1 \ll_\epsilon \sum_{\lambda \in \Lambda^{\vee} \setminus \zed \cdot
\lambda^*}\exp\left(\frac{-8 \pi^2 t}{2k+1} \|\lambda\|_2^2\right)
 \ll \sum_{\substack{\lambda \in\Lambda^\vee
\setminus \{\pm \lambda^*\} \\ p\lambda \text{ primitive}}}\exp\left(\frac{-8
\pi^2 t}{2k+1} \|\lambda\|_2^2\right).
\end{equation}
The latter sum may be written as
\[
\sum_{\substack{\lambda \in\Lambda^\vee
\setminus \{\pm \lambda^*\} \\ p\lambda \text{ primitive}}}\left(\frac{8\pi
t}{2k+1} \right)^{\frac{k}{2}}\eta_k\left(\frac{1}{4\pi}\sqrt{\frac{2k+1}{t}},
\lambda \right).
\]
Since $t \gg C t^{\rel} \asymp C\frac{k}{\|\lambda^*\|_2^2} \asymp C\rho^2
p^{\frac{2}{k}}$ (take $C \asymp 1$ in the case of the second estimate of
(\ref{small_mixing_time_prob})) there is $c \asymp C$ such that
\[
 \sum_{\substack{\lambda \in\Lambda^\vee
\setminus \{\pm \lambda^*\} \\ p\lambda \text{ primitive}}}
\eta_k\left(\sqrt{\frac{k}{c}} \frac{1}{\rho p^{\frac{1}{k}}} , \lambda\right)
\gg p \rho^k \left(\frac{C}{k}\right)^{\frac{k}{2}} e^{O(k)}.
\]
Applying Proposition \ref{pair_short_vector_prop},
\[
 \E_{L^0(p,k)}\left[\sum_{\substack{\lambda_1 \neq
\pm \lambda_2 \in
\Lambda^{\vee}\setminus\{0\}\\p\lambda_i \text{
primitive}}}\eta_k\left(
\frac{1}{\rho p^{\frac{1}{k}}},
\lambda_1 \right)\eta_k\left(\sqrt{\frac{k}{c}} \frac{1}{\rho p^{\frac{1}{k}}},
\lambda_2\right)\right]
\leq p^2 + O_\epsilon\left(p^{\frac{3}{2} +
\frac{4}{k} + \epsilon}\right)
\]
and thus, by specializing to $\lambda_1 = \lambda^*$ and applying Markov's inequality,
\begin{align*}
 &\Prob_{L^0(p,k)}\left[\|\lambda^*\|_2 \asymp \frac{\sqrt{k}}{\rho
p^{\frac{1}{k}}} \text{ and } 1 \ll_\epsilon \sum_{\lambda \in \Lambda^{\vee}
\setminus \zed \cdot
\lambda^*}\exp\left(\frac{-8 \pi^2 t}{2k+1} \|\lambda\|_2^2\right)\right] \\&\ll
\rho^{-2k}\exp\left(\frac{k}{2}\log \frac{k}{C} +O(k)\right).
\end{align*}
This verifies (\ref{small_mixing_time_prob}).

\end{proof}

\subsection{Analysis of typical mixing behavior}\label{typical_section}
We turn to analysis of the mixing behavior for $A$ in the bulk of $\sA(p,k)$ proving the following theorem.
\begin{theorem}\label{random_sharp_threshold_theorem}
Let $p$ be prime, let $0 < \epsilon = \epsilon(p)<\frac{1}{2}$ and let
 $1 \leq k \leq \frac{\log p}{\log \log p}$.  Set
$\overline{t^{\mix}_1} = \frac{k}{2\pi
e}p^{\frac{2}{k}}$.  There is a function $\theta = \theta(\epsilon, k)>0$
tending
to 0 as $\epsilon^2 k \to \infty$ and a set $\sA^*(p,k) \subset \sA(p,k)$
satisfying
\[
 \left|\sA^*(p,k)\right| \geq (1-o(1))\left|\sA(p,k)\right|,
\]
such that, for all $A \in \sA^*_{p,k}$,
\begin{align*}
\forall\, n < (1-\epsilon)\overline{t^{\mix}_1}, \qquad  \left\|\mu_A^{*n} -
\frac{1}{p} \right\|_{\TV(\zed/p\zed)} &\geq 1- \theta(\epsilon, k) + o(1), \\
\forall\, n > (1+\epsilon)\overline{t^{\mix}_1}, \qquad  \left\|\mu_A^{*n} -
\frac{1}{p} \right\|_{\TV(\zed/p\zed)} &\leq \theta(\epsilon, k)+ o(1)
\end{align*}
where all quantities $o(1)$ tend to zero as $p \to \infty$ uniformly in $k$.
\end{theorem}
We can now conclude our proof of Theorem \ref{random_theorem}.
\begin{proof}[Deduction of Theorem \ref{random_theorem}, part (3)]
 For each $j = 1, 2, ...$, let $E(p, j)$
be
the event that
 \begin{align*}
\forall\, n < (1-\epsilon(p))\overline{t^{\mix}_1}, \qquad  \left\|\mu_A^{*n} -
\frac{1}{p} \right\|_{\TV(\zed/p\zed)} &> 1-2^{-j}, \\
\forall\, n > (1+\epsilon(p))\overline{t^{\mix}_1}, \qquad  \left\|\mu_A^{*n} -
\frac{1}{p} \right\|_{\TV(\zed/p\zed)} &< 2^{-j}.
\end{align*}
For a fixed $p$, the events $E(p,j)$ are nested in $j$. For each $j
\in\zed_{>0}$, let $N_j$ be minimal
such that for all $p > N_j$, $\bU_{\sA(p,k)}[E(p,j)] \geq 1 - 2^{-j}$.  This is
finite by Theorem
\ref{random_sharp_threshold_theorem}. Define
$E^*(p) = \bigcap_{j: N_j < p} E(p,j)$ and let $p
\in \sP_0$ if and only if $E^*(p)$ occurs.  Since
$\bU_{\sA(p,k)}\left[E^*(p)\right] \to 1$
as $p \to \infty$ and the events are independent, we have $\sP_0$ has density 1
with probability 1, as desired.
\end{proof}

In the remainder of
this section we shall frequently be concerned with counting
lattice points within Euclidean balls $B_2(\ux, R) \subset \bR^k$.  It is useful
to bear in mind that the radius $R_k$ of a ball of unit volume in $\bR^k$
satisfies
\[
 R_k = \left(\frac{\Gamma\left(\frac{k}{2}+1 \right)}{\pi^{\frac{k}{2}}}
\right)^{\frac{1}{k}} = \sqrt{\frac{k}{2\pi e}}\left(1 + \frac{\log k}{2k} +
O\left( \frac{1}{k}\right)\right).
\]

Let $\epsilon = \epsilon(p)$ as in the theorem and set $\delta = \frac{1}{2}(1 -
\sqrt{1-\epsilon})$.
Recall that, given lattice $\Lambda < \bR^k$, $\sF(\Lambda)$ is the
norm-minimal fundamental domain of $\Lambda$,
\[\sF(\Lambda)= \left\{\ux \in \bR^k: \forall\, \lambda \in \Lambda
\setminus\{0\}, \|\ux\| < \|\lambda - \ux\|\right\}. \]
Let $k = k(p)$ and set $t = t(p,k)=
(1-\epsilon)\overline{t^{\mix}_1} \sim (1-\epsilon)R_k^2 p^{\frac{2}{k}}$.

\begin{lemma}\label{random_lower_bound_estimate_lemma} As $k, p\to \infty$ in
such a way that $k \leq \frac{\log p}{\log \log p}$ we have
\begin{equation}\label{rand_lower_bound_estimate}
 \E_{L(p,k)} \left[\int_{\ux \in B_2\left(0, (1-\delta) R_k
p^{\frac{1}{k}}\right)
\cap \sF(\Lambda)} \Theta\left( \ux, \frac{2t}{2k+1}; \Lambda\right) d\ux
\right]= 1-o(1).
\end{equation}
\end{lemma}

\begin{proof}
 Since $\Theta\left(\ux, \frac{2t}{2k+1}; \Lambda\right) \geq
\eta_k\left(\sqrt{\frac{2t}{2k+1}}, \ux \right)$,
 \begin{align}\notag
  &\E_{L(p,k)} \left[\int_{\ux \in B_2\left(0, (1-\delta) R_k
p^{\frac{1}{k}}\right)
\cap \sF(\Lambda)} \Theta\left( \ux, \frac{2t}{2k+1}; \Lambda\right) d\ux
\right]\\& \label{rand_lower_bound_1} \geq \int_{\|\ux\| \leq
(1-\delta) R_k
p^{\frac{1}{k}}
} \eta_k\left( \sqrt{\frac{2t}{2k+1}}, \ux \right)d\ux
\\& \label{rand_lower_bound_2} - \E\left[\int_{\|\ux\| \leq
(1-\delta) R_k
p^{\frac{1}{k}}
} \eta_k\left( \sqrt{\frac{2t}{2k+1}}, \ux \right) \one\left(\exists\;
\lambda
\in \Lambda\setminus\{0\}: \|\lambda-\ux\| < \|\ux\| \right) d\ux \right].
 \end{align}
Since $\delta \sim \frac{\epsilon}{2}$ as $\epsilon \downarrow 0$,
$(\ref{rand_lower_bound_1}) = 1-o(1)$ follows from concentration of the norm of
a Gaussian vector on scale $\frac{1}{\sqrt{k}}$ times its median length, see
Lemma \ref{gaussian_length_concentration}.

We estimate
\begin{align*}
 (\ref{rand_lower_bound_2}) \leq \int_{\|\ux\| \leq
(1-\delta) R_k
p^{\frac{1}{k}}}\eta_k\left( \sqrt{\frac{2t}{2k+1}}, \ux \right)\E\left[
\sum_{\lambda
\in \Lambda\setminus\{0\}}\one\left(\|\lambda-\ux\| < \|\ux\|
\right)\right]d\ux.
\end{align*}
For $k$ sufficiently large, any $\lambda$ counted in the expectation satisfies
$\|\lambda\| < p$, and thus, by Lemma \ref{expectation_lemma},
\[
 \E\left[
\sum_{\lambda
\in \Lambda\setminus\{0\}}\one\left(\|\lambda-\ux\| < \|\ux\| \right)\right] =
\frac{p^{k-1}-1}{p^k-1} \#\{\lambda \in \zed^k: \|\lambda-\ux\| <
\|\ux\|\}.
\]
For any $\ux \in \bR^k$, any lattice point $\tilde{\ux} \in \zed^k$
which is the vertex of the unit lattice cube containing $\ux$ satisfies
$\|\tilde{\ux}\| = \left(1 +
O\left(\frac{\sqrt{k}}{\|\ux\|}\right)\right)\|\ux\|$.  Since $k^{\frac{3}{2}} =
o(R_k p^{\frac{1}{k}})$, it follows that for all $\|\ux\| \leq (1-\delta) R_k
p^{\frac{1}{k}}$ we have
\[
 \frac{p^{k-1}-1}{p^k-1} \#\{\lambda \in \zed^k: \|\lambda-\ux\| <
\|\ux\|\} \leq \left(1-\delta + o\left(\frac{1}{k}\right)\right)^k = o(1),
\]
and thus
\[
 (\ref{rand_lower_bound_2}) = o\left( \int_{\|\ux\| \leq
(1-\delta) R_k
p^{\frac{1}{k}}}\eta_k\left( \sqrt{\frac{2t}{2k+1}}, \ux \right) d\ux\right) =
o(1).
\]

\end{proof}

\begin{proof}[Proof of Theorem \ref{random_sharp_threshold_theorem}, lower bound]
For $n \geq \frac{t(p,k)}{2}$,
Lemma \ref{continuous_approx_lemma} gives
\[
 \E_{\sA(p,k)} \left[\left\|\mu_A^{*n} -
\bU_{\zed/p\zed}\right\|_{\TV(\zed/p\zed)} \right] = o(1) +
\E_{L^0(p,k)}\left[\left\|\Theta\left(\cdot, \frac{2n}{2k+1};\Lambda \right)
-\frac{1}{p}\right\|_{\TV (\bR^k/\Lambda)} \right]
\]
while, for all $n < t(p,k)$,
\begin{align*}
& (1 + o(1))\E_{L^0(p,k)}\left[ \left\|\Theta\left(\cdot,
\frac{2n}{2k+1};\Lambda \right)
-\frac{1}{p}\right\|_{\TV(\bR^k/\Lambda)}\right] \\&\geq
\E_{L(p,k)}\left[\int_{\ux \in B_2\left(0, (1-\delta) R_k p^{\frac{1}{k}}\right)
\cap \sF(\Lambda)}
\Theta\left( \ux, \frac{2t}{2k+1}; \Lambda\right) -\frac{1}{p}d\ux \right].
\end{align*}
By Lemma \ref{random_lower_bound_estimate_lemma}, the expectation of the integral against $\Theta$ is $1-o(1)$, while the expectation of the integral against $\frac{1}{p}$ is bounded by
\[
  \int_{\ux \in
B_2\left(0,(1-\delta) R_k p^{\frac{1}{k}}\right)} \frac{1}{p} d\ux =
(1-\delta)^k
= o(1).
\]
\end{proof}

\subsubsection{Proof of Theorem \ref{random_sharp_threshold_theorem}, upper
bound}
The main proposition of the upper bound is as follows.
\begin{proposition}\label{theta_dist_prop}  Let $p$ and $k(p)$ tend to $\infty$
in
such a way that $k \leq
\frac{\log p}{\log \log p}$, and let $0 < \epsilon(p)<1$ with $\epsilon(p)^2k(p)
\to \infty$.  Set
$t=t(p)  = (1+\epsilon)\frac{k}{2\pi e}
p^{\frac{2}{k}} $. For any fixed $\delta > 0$
\begin{equation}
 \bU_{L(p,k) \times (\bR/p\zed)^k} \left[(\Lambda, \ux) \in
L(p,k)\times (\bR/p\zed)^k: \left|\Theta\left( \ux, \frac{2t}{2k+1};
\Lambda\right)-\frac{1}{p}\right|<\frac{\delta}{p}\right] = (1 +
o_{\delta}(1)).
\end{equation}

\end{proposition}
\begin{proof}[Deduction of Theorem \ref{random_sharp_threshold_theorem}, upper
bound]
 For any $\Lambda \in L^0(p,k)$ we have $p\zed^k < \Lambda$, and thus
 \begin{align*}
  \left\|\Theta\left( \ux, \frac{2t}{2k+1};
\Lambda\right) -\frac{1}{p}\right\|_{\TV(\bR^k/\Lambda)} &= \int_{\bR^k/\Lambda}
\frac{1}{p}- \min\left(\Theta\left( \ux, \frac{2t}{2k+1};
\Lambda\right) ,\frac{1}{p}\right)d\ux\\
&= p^{-k+1} \int_{\ux \in (\bR/p\zed)^k}\frac{1}{p}- \min\left(\Theta\left(
\ux, \frac{2t}{2k+1};
\Lambda\right) ,\frac{1}{p}\right)d\ux
 \end{align*}
and so
\begin{align*}
 &(1 + o(1))\E_{L^0(p,k)}\left[ \left\|\Theta\left( \ux, \frac{2t}{2k+1};
\Lambda\right) -\frac{1}{p}\right\|_{\TV(\bR^k/\Lambda)} \right]
\\&=\E_{L(p,k)}\left[ \left\|\Theta\left( \ux, \frac{2t}{2k+1};
\Lambda\right) -\frac{1}{p}\right\|_{\TV(\bR^k/\Lambda)} \right] < \delta +
o(1).
\end{align*}
\end{proof}

Let $\tau = \tau(p)= \frac{\epsilon(p)}{2}$.  Given $\ux \in \bR^k$, define
spherical shell
\[
 S(\ux, R, \tau) = \left\{\uy \in \bR^k: \|\uy-\ux\| \in [(1-\tau)R, (1 +
\tau)R]\right\}.
\]
We use several times the estimate
for $\ux \in S(0, \sqrt{t}, \tau)$
\begin{align}\label{gaussian_bound}
\eta_k\left(\sqrt{\frac{2t}{2k+1}}, \ux \right) & \leq
\frac{1}{p}\exp\left(-\left(\frac{3\epsilon^2}{4} + O(\epsilon^3)
\right)\frac{k}{2} \right) = o\left(\frac{1}{p} \right).
\end{align}

The
critical part of $\Lambda$ when considering
$\Theta\left( \ux,
\frac{2t}{2k+1};
\Lambda\right)$ in $L^1$ is $
 \Lambda_c(\ux) = \Lambda \cap S(\ux, \sqrt{t}, \tau).$
We split
\begin{align*}
\Theta\left( \ux,
\frac{2t}{2k+1};
\Lambda\right) &= \Theta_M\left( \ux,
\frac{2t}{2k+1};
\Lambda\right) + \Theta_E\left( \ux,
\frac{2t}{2k+1};
\Lambda\right); \\
\Theta_M\left( \ux,\frac{2t}{2k+1};\Lambda\right) &= \sum_{\lambda \in
\Lambda_c(\ux)} \eta_k\left(\sqrt{\frac{2t}{2k+1}}, \lambda - \ux \right).
\end{align*}

\begin{lemma}\label{mean_lemma}
 For all $\ux \in (\bR/p\zed)^k$,
 \[
  \E_{L(p,k)}\left[\Theta_M\left(\ux, \frac{2t}{2k+1};\Lambda \right) \right]=
\frac{1}{p}(1 + o(1)).
 \]

\end{lemma}
\begin{proof}
If $p$ is sufficiently large then there is at most one point of $p\zed^k$
contained in $\Lambda_c(\ux)$, and so (\ref{gaussian_bound}) gives
\begin{align*}
&\E_{L(p,k)}\left[\Theta_M\left(\ux, \frac{2t}{2k+1};\Lambda \right) \right]\\&=
o\left(\frac{1}{p} \right) + \frac{p^{k-1}-1}{p^k-1} \sum_{\lambda \in \zed^k}
\eta_k\left(\sqrt{\frac{2t}{2k+1}}, \lambda - \ux \right) \one\left(\lambda \in
S\left(\ux, \sqrt{t}, \tau\right) \right)
\end{align*}
Let $\uv \in \bR^k$ be a unit vector, and let $D_{\uv}$ denote the directional
derivative in the $\ux$ variable in direction $\uv$.  For any $\lambda \in
S\left(\ux, \sqrt{t}, 2\tau\right)$ we have
\[
 \left|D_\uv\left(\log \eta_k\left(\sqrt{\frac{2t}{2k+1}}, \lambda - \ux \right)
\right)\right| \ll \frac{k}{\sqrt{t}} \ll \frac{\sqrt{k}}{p^{\frac{1}{k}}}.
\]
In particular, for any $\uy \in \left[-\frac{1}{2}, \frac{1}{2}\right)^k$,
since $\|\uy\|_2 \leq \frac{\sqrt{k}}{2}$, we have
\[
 \eta_k\left(\sqrt{\frac{2t}{2k+1}}, \lambda - \ux \right) = (1 +
o(1))\eta_k\left(\sqrt{\frac{2t}{2k+1}}, \lambda - \ux -\uy\right).
\]
Thus the sum may be approximated with an integral, and the result
follows.
\end{proof}

\begin{lemma}
 We have the following estimates.
 \begin{align*}
  \E_{L(p,k) \times (\bR/p\zed)^k}\left[ \Theta_M\left(
\ux,\frac{2t}{2k+1};\Lambda\right)\right] &= \frac{1}{p}\left(1 +
o(1)\right)\\
  \E_{L(p,k) \times (\bR/p\zed)^k} \left[\Theta_E\left(
\ux,\frac{2t}{2k+1};\Lambda\right) \right] &= o\left(\frac{1}{p}\right)
\end{align*}
and for $k > 2$,
\begin{align*}
&\E_{ (\bR/p\zed)^k} \left[\Var_{L(p,k)} \left[
\Theta_M\left(
\ux,\frac{2t}{2k+1};\Lambda\right)\right]\right]  = o\left(\frac{1}{p^2}
\right).
 \end{align*}

\end{lemma}
\begin{proof}
The evaluations of the means follow from Lemma \ref{mean_lemma}.

In evaluating the variance term, we write, for $\lambda_1, \lambda_2 \in
\zed^k$,
$\lambda_1 \sim \lambda_2$ if $\lambda_2 \equiv a \lambda_1 \bmod p$ for some
$0,  1 \not \equiv a \bmod p$. We have the following evaluations (see Lemma \ref{expectation_lemma}):
\begin{align*}
&\bU_L\left(\Lambda: \lambda_1, \lambda_2 \in \Lambda \right)
-\bU_L\left(\Lambda: \lambda_1 \in \Lambda \right)\bU_L\left(\Lambda: \lambda_2
\in \Lambda \right)\\&\qquad= \left\{\begin{array}{lll}0 && \lambda_1 \in
p\zed^k \text{ or } \lambda_2 \in p\zed^k\\O \left(p^{-k} \right) &&
\lambda_1,
\lambda_2 \in \zed^k\setminus p\zed^k,\; \lambda_1 \neq \lambda_2, \lambda_1
\not \sim \lambda_2 \\
O\left( p^{-1}\right)&& \lambda_1,
\lambda_2 \in \zed^k\setminus p\zed^k,\; \lambda_1  \sim \lambda_2 \text{ or }
\lambda_1 = \lambda_2
\end{array}\right..
\end{align*}
The variance thus evaluates to
\begin{align}
 &\notag \E_{(\bR/p\zed)^k}\left[\Var_{L(p,k)}\left[\Theta_M\left(
\ux,\frac{2t}{2k+1};\Lambda\right) \right]\right] \ll \\
\label{variance_first}&\frac{1}{p^k}\sum_{\lambda_1, \lambda_2 \in \zed^k
\setminus (p\zed)^k}
\left(\one(\lambda_1 \sim \lambda_2)\frac{1}{p} + O(p^{-k}) \right)\\ \notag
&\times
\int_{\ux \in [-\frac{p}{2},\frac{p}{2})^k\cap S(\lambda_1, \sqrt{t}, \tau) \cap
S(\lambda_2, \sqrt{t}, \tau)}\eta_k\left(\sqrt{\frac{2t}{2k+1}},
\lambda_1 - \ux \right) \eta_k\left(\sqrt{\frac{2t}{2k+1}}, \lambda_2
- \ux \right)  d\ux\\
&\label{variance_second}+ \frac{1}{p^{k+1}}\sum_{\lambda \in \zed^k \setminus
(p\zed)^k}\int_{\ux \in
\left[-\frac{p}{2},\frac{p}{2}\right)^k \cap
S(\lambda,\sqrt{t},\tau)}\eta_k\left(\sqrt{\frac{2t}{2k+1}},
\lambda - \ux \right)^2d\ux.
\end{align}

The term
(\ref{variance_second}) captures $\lambda_1 =  \lambda_2$.
Replacing one Gaussian by the bound (\ref{gaussian_bound}) and then estimating
as for the mean of $\Theta_M$ gives a bound for this term of
\[(\ref{variance_second})= o\left(\frac{1}{p^2} \right).\]

The error term $O(p^{-k})$ of (\ref{variance_first})  may
be bounded by omitting the restriction on $\|\lambda_2 - \ux\|$ and summing over
$\lambda_2$, the
summation being bounded by a constant. The remaining summation over $\lambda_1$
and integral over $\ux$ are then evaluated as for the mean, and give an error
of $O(p^{-k})$.

It remains to treat those terms from (\ref{variance_first}) with $\lambda_1
\sim \lambda_2$.   Let $R(\tau) = 2 (1 + \tau)\sqrt{t}.$  Any
$\lambda_1 \sim \lambda_2$ contributing
to the variance satisfies $\lambda = \lambda_1 -
\lambda_2 \in B(0, R(\tau))\setminus\{0\}$ and $\lambda_1 \equiv
(a+1)\lambda \bmod p\zed^k$, $\lambda_2 \equiv a \lambda \bmod p\zed^k$ for
some $ a \bmod p$. Arranging the summation over $\lambda$ and $a$, we find that
the contribution of terms with $\lambda_1 \sim \lambda_2$ to
(\ref{variance_first}) is bounded by (by expanding the integral, this is now
independent of $a$, which we pull out)
\begin{align*}
\ll& \frac{1}{p^k}\sum_{\lambda \in \zed^k \cap B(0,
R(\tau))\setminus \{0\}} \int_{\ux \in S(0, \sqrt{t}, \tau)\cap S(\lambda,
\sqrt{t},\tau)}
\eta_k\left(\sqrt{\frac{2t}{2k+1}},
 \ux \right)  \eta_k\left(\sqrt{\frac{2t}{2k+1}}, \lambda
- \ux \right)d\ux.
\end{align*}
The total number of such $\lambda$ is $\ll 2^k (1+\tau)^k (1
+\epsilon)^k p$ by estimating with the volume of the ball, see Lemma
\ref{point_count_lemma}. Putting in the bound (\ref{gaussian_bound}) for one
Gaussian
and integrating the second over all of $\bR^k$, we obtain an estimate from the
terms with $\lambda_1 \sim \lambda_2$ of $\ll \frac{8^k}{p^k}.$

\end{proof}

\begin{proof}[Proof of Proposition \ref{theta_dist_prop}]
Consider separately the cases \[|\Theta_E|, \left|\E_{L(p,k)}[\Theta_M]-\frac{1}{p}\right|, |\Theta_M
- \E_{L(p,k)}[\Theta_M]| > \frac{\delta}{3p}\] and apply Markov's inequality.
\end{proof}

\section{The power-of-2 random walk}
\subsection{A Chebyshev cut-off criterion}
We begin by describing a commonly used second moment method for proving cut-off,
which we apply in analyzing the power-of-2 random walk. The following is a variant of the lower bound method from \cite{DS87}, see also Wilson's lemma in \cite{LPW09}.

Given a probability measure $\mu$ on $\zed/p\zed$ and frequency $\xi \in \zed/p\zed$, define the Fourier coefficient of $\mu$ at $\xi$ to be
\[
 \hat{\mu}(\xi) = \sum_{x \in\zed/p\zed}\mu(x)e\left(\frac{\xi x}{p}\right).
\]
Define, as before, the $L^2$ mixing time by
\[
 t^{\mix}_2 = \inf \left\{n \in \zed_{>0}:\sum_{0 \neq \xi \in \zed/p\zed}
\left|\hat{\mu}(\xi) \right|^{2n} \leq \frac{4}{e^2} \right\}
\]
and the spectral gap
\[
\gap = 1 - \max_{0 \neq \xi \in \zed/p\zed}\left|\hat{\mu}(\xi)
\right|.
\]

\begin{proposition}\label{Chebyshev_criterion} Let $\{A_p \subset
\zed/p\zed\}_{p \in \sP}$ be a sequence of symmetric, lazy, generating sets
for $\zed/p\zed$, with $\mu_{A_p}$ the corresponding uniform probability
measure. Assume that the spectral gap tends to 0 with increasing $p$.

Suppose the following holds for each fixed $\epsilon > 0$.
For each
$p \in \sP$ there exists symmetric subset $0 \not \in B_p \subset
\widehat{\zed/p\zed}$ such that as $p \to \infty$,
\begin{itemize}
 \item For all $\xi \in B_p$,
 \begin{equation}\label{large_frequency_condition}\qquad \hat{\mu}_{A_p}(\xi) =
1-o(1).\end{equation}
 \item For all $n < (1-\epsilon)t^{\mix}_2(p)$
 \begin{equation}\label{mean_grows}\frac{1}{\sqrt{|B_p|}} \sum_{\xi \in B_p}
\hat{\mu}_{A_p}^{n}(\xi) \to \infty\end{equation}
 \item For all $n < (1-\epsilon)t^{\mix}_2(p)$
 \begin{equation}\label{cut_off_condition}\sum_{\xi_1, \xi_2 \in
B_p}\hat{\mu}_{A_p}^{n}(\xi_1 - \xi_2)  \leq (1 + o(1)) \sum_{\xi_1, \xi_2 \in
B_p}
\hat{\mu}_{A_p}^{n}(\xi_1) \hat{\mu}_{A_p}^{n}(\xi_2). \end{equation}
\end{itemize}
Then the sequence $\{( \zed/p\zed, \mu_{A_p}, \bU_{\zed/p\zed})\}$ converges to
uniform in total
variation
distance with a cut-off at $t^{\mix}_1(p)\sim t^{\mix}_2(p)$ if and only if
the condition
\begin{equation}\label{Peres}t^{\mix}_2(p) \gap(p)  \to
\infty \qquad \text{as } p \to \infty\end{equation} is satisfied.
\end{proposition}

\begin{rem}
The condition (\ref{cut_off_condition}) is in fact equivalent to
\begin{equation} \label{weaker_cut_off_condition}\sum_{\xi_1, \xi_2 \in
B_p}\hat{\mu}_{A_p}^{n}(\xi_1 - \xi_2)  =  (1 + o(1)) \sum_{\xi_1, \xi_2 \in
B_p}
\hat{\mu}_{A_p}^{n}(\xi_1) \hat{\mu}_{A_p}^{n}(\xi_2)\end{equation} since
$$\sum_{\xi_1, \xi_2 \in
B_p}\hat{\mu}_{A_p}^{n}(\xi_1 - \xi_2)  \geq \sum_{\xi_1, \xi_2 \in B_p}
\hat{\mu}_{A_p}^{n}(\xi_1) \hat{\mu}_{A_p}^{n}(\xi_2)$$ by the following
application of
Cauchy-Schwarz:
\begin{align*}
\sum_{\xi_1, \xi_2 \in B_p}
\hat{\mu}_{A_p}^{n}(\xi_1) \hat{\mu}_{A_p}^{n}(\xi_2) & = \left| \sum_{\xi \in
B_p}
\hat{\mu}_{A_p}^{n}(\xi)\right|^2 \\ & = \left| \sum_{x \bmod p}
\mu_{A_p}^{*n}(x)
\sum_{\xi \in B_p} e\left(\frac{\xi x}{p}\right) \right|^2 \\
&\leq \left(\sum_{x \bmod p} \mu_{A_p}^{*n}(x)\right)\left( \sum_{x \bmod p}
\mu_{A_p}^{*n}(x) \sum_{\xi_1, \xi_2 \in B_p} e\left(\frac{(\xi_1 -
\xi_2)x}{p}\right)\right)
\\ & = \sum_{\xi_1, \xi_2 \in B_p} \hat{\mu}_{A_p}^{n}(\xi_1 - \xi_2).
\end{align*}

\end{rem}

\begin{proof}[Proof of Proposition \ref{Chebyshev_criterion}]
Since $t^{\mix}_1 \leq t^{\mix}_2$, if the condition
$\gap(p) \cdot t^{\mix}_2(p)\to \infty$ fails then there is
no cut-off in total variation, so we may assume that this condition holds.
Let $\epsilon>0$ be fixed.  For  $n > (1 +
\epsilon) t^{\mix}_2$, by Cauchy-Schwarz,
\begin{align}\label{upper_bound}\left\|\mu_{A_p}^{*n} -
\bU_{\zed/p\zed}\right\|_{\TV(\zed/p\zed)}^2 &\leq \frac{1}{4}\sum_{\xi \not
\equiv 0 \bmod
p} |\hat{\mu}_{A_p}(\xi)|^{(2 +
2\epsilon)t_2^{\mix}}\\& \notag \leq \frac{1}{4}(1-\gap)^{2 t^{\mix}_2 \epsilon}
\sum_{\xi \not \equiv 0 \bmod p}
\left|\hat{\mu}_{A_p}(\xi)\right|^{2 t_2^{\mix}} \to
0\end{align} since
$\gap \cdot t_2^{\mix} \to \infty$.

To prove the lower bound, let $n < (1-\epsilon) t^{\mix}_2$.
  Define function $f_p(x)$ on $\zed/p\zed$ by $f_p(x) =
\frac{1}{\sqrt{|B_p|}} \sum_{\xi \in B_p} \hat{\mu}_{A_p}(\xi)
e\left(\frac{-\xi x}{p}\right)$.
Writing $\E_\mu$, $\Var_\mu$ for expectation and variance with respect to
probability measure
$\mu$, we have
\begin{equation}\label{uniform_expectation}
\E_{\bU_{\zed/p\zed}}\left[f_p\right] =
\frac{1}{p} \sum_{x \bmod p} \frac{1}{\sqrt{|B_p|}}\sum_{\xi \in
B_p}\hat{\mu}_{A_p}(\xi)e\left(\frac{-\xi x}{p}\right) = \frac{1}{\sqrt{|B_p|}}
\sum_{\xi \in B_p}
\hat{\mu}_{A_p}(\xi)\delta_{\xi = 0} = 0 \end{equation} since $0 \not \in B_p$,
and
\begin{align}\label{uniform_variance}
 \Var_{\bU_{\zed/p\zed}}\left[f_p\right] &= \frac{1}{p} \sum_{x \bmod p}
\frac{1}{|B_p|}\sum_{\xi_1, \xi_2 \in
B_p} \hat{\mu}_{A_p}(\xi_1)\hat{\mu}_{A_p}(\xi_2)e\left(\frac{-(\xi_1 -
\xi_2)x}{p}\right) \\\notag &=
\frac{1}{|B_p|}\sum_{\xi\in B_p} \hat{\mu}_{A_p}(\xi)^2 \leq 1.
\end{align}

Meanwhile
\begin{align} \notag \E_{\mu_{A_p}^{*n}}\left[f_p\right] &= \sum_{x \bmod p}
\frac{1}{\sqrt{|B_p|}} \sum_{\xi \in B_p} \hat{\mu}_{A_p}(\xi)e\left(\frac{-\xi
x}{p}\right) \mu_{A_p}^{*n}(x)
\\ \label{mu_expectation}&= \frac{1}{\sqrt{|B_p|}} \sum_{x \bmod p}\sum_{\xi \in
B_p} \hat{\mu}_{A_p}(\xi)e\left(\frac{-\xi x}{p}\right) \frac{1}{p} \sum_{\eta
\bmod p} \hat{\mu}_{A_p}^{n}(\eta)
e\left(\frac{-\eta x}{p}\right) \\\notag & = \frac{1}{\sqrt{|B_p|}} \sum_{\xi
\in B_p} \hat{\mu}_{A_p}(\xi)
\hat{\mu}_{A_p}^n(-\xi) \\& \notag= (1 + o(1)) \frac{1}{\sqrt{|B_p|}}
\sum_{\xi \in B_p}
\hat{\mu}_{A_p}^n(\xi) \end{align}
and
\begin{align} \notag \E_{\mu_{A_p}^{*n}}\left[f_p^2\right] &= \sum_{x \bmod p}
\frac{1}{|B_p|} \sum_{\xi_1, \xi_2 \in B_p}
\hat{\mu}_{A_p}(\xi_1)\hat{\mu}_{A_p}(\xi_2)e\left(\frac{-(\xi_1-\xi_2)x}{p}
\right) \mu_{A_p}^{*n}(x)
\\ \label{mu_expectation_squared}&= \frac{1}{|B_p|} \sum_{\xi_1, \xi_2 \in B_p}
\hat{\mu}_{A_p}(\xi_1)\hat{\mu}_{A_p}(\xi_2)
\hat{\mu}_{A_p}^n(-(\xi_1 - \xi_2)) \\ \notag &= (1 + o(1)) \frac{1}{|B_p|}
\sum_{\xi_1, \xi_2 \in
B_p}
\hat{\mu}_{A_p}^n(\xi_1 - \xi_2). \end{align}
It follows by condition (\ref{cut_off_condition}) that
\begin{equation}\label{mu_variance} \Var_{\mu_{A_p}^{*n}}\left[f_p\right] =
\E_{\mu_{A_p}^{*n}}\left[f_p^2\right] - \E_{\mu_{A_p}^{*n}}\left[f_p\right]^2 =
o\left(\E_{\mu_{A_p}^{*n}}\left[f_p^2\right]\right). \end{equation}  Let $X_p$
be the subset of
$\zed/p\zed$ defined by
$$X_p = \left\{x: f_p(x) \leq \frac{1}{2}
\E_{\mu_{A_p}^{*n}}\left[f_p\right]\right\}.$$ By
(\ref{mu_expectation}) and condition (\ref{mean_grows}),
\begin{equation}\label{growing_mean} \E_{\mu_{A_p}^{*n}}\left[f_p\right]
\to \infty .\end{equation}  Hence Chebyshev's inequality,
(\ref{uniform_expectation}), (\ref{uniform_variance}) and
(\ref{growing_mean}) imply $$\bU_{\zed/p\zed}\left[X_p\right] = 1 - o(1)$$ while
Chebyshev,
(\ref{mu_variance}) and (\ref{growing_mean}) imply
$$\mu_{A_p}^{*n}\left(X_p\right) =
o(1).$$
We conclude
$$\left\|\mu_{A_p}^{*n} - \bU_{\zed/p\zed} \right\|_{\TV(\zed/p\zed)} \geq
\left|\bU_{\zed/p\zed}(X_p)-\mu_{A_p}^{*n}(X_p)\right| = 1 -o(1).$$

\end{proof}

\subsection{Proof of Theorem \ref{power_2_theorem}, lower bound} Recall that we
set $\ell = \ell_2(p) = \left\lceil \log_2 p \right\rceil$ and $$c_0 = \sum_{j =
1}^\infty \left(1- \cos
\frac{2\pi}{2^j}\right).$$ We prove the
lower bound of Theorem \ref{power_2_theorem} conditional on $t_2^{\mix} \lesssim
\frac{\ell \log \ell}{2c_0}$, which is proven in the next section. The proof of
the lower bound is a reduction to the conditions of Proposition
\ref{Chebyshev_criterion}.

Let $J = o(\log \log p)$ be a parameter.  With an eye toward applying
Proposition \ref{Chebyshev_criterion}, set
$$B_p = \left\{\xi \in \widehat{\zed/p\zed}: \exists\, 1
\leq j_1 \neq j_2 \leq \ell, \; \left\|\frac{\xi}{p}- 2^{-j_1} + 2^{-j_2}
\right\|_{\bR/\zed} \leq 2^{-\ell - J}\right\}.$$

\begin{lemma}\label{big_B} $|B_p| \gg \frac{\ell^2}{2^J} - \ell$. \end{lemma}
\begin{proof}
Let $$S = \left\{\xi \bmod p: \exists\, 1 \leq j \leq \ell, \;
\left\|\frac{\xi}{p} -\frac{1}{2^j} \right\| \leq 2^{-\ell}
\right\}.$$  We have $\ell \leq |S| \leq 2\ell$.  For each $s \in S$ write
$\frac{s}{p}$ in binary $$\frac{s}{p} = *.s_1 s_2 s_3 ....$$  Partition $S$
into $2^{J+1}$ sets $S_1, S_2, ..., S_{2^{J+1}}$ according to the digits
$s_{\ell}s_{\ell+1}...s_{\ell+J}$.  To each pair $s \neq s' \in S_i$ we obtain
$r = s
- s' \in B_p$.  The multiplicity with which a given such $r$ arises in this way
is  $O(1)$.  Hence
$$|B_p| \gg \sum_{j = 1}^{2^{J+1}} |S_j|(|S_j|-1) = -|S|+\sum_{j = 1}^{2^{J+1}}
|S_j|^2 .$$  By Cauchy-Schwarz,
$$|S|^2 = \left(\sum_j |S_j|\right)^2 \leq 2^{J+1} \sum_j |S_j|^2$$ so
$$|B_p| \gg \frac{|S|^2}{2^{J+1}} - |S| \geq \frac{\ell^2}{2^{J+1}} -\ell.$$
\end{proof}

\begin{lemma}\label{FC_lower_bound}
 For $\xi \in B_p$, $\hat{\mu}_{A_{2,p}}(\xi) \geq 1 - \frac{4c_0}{2\ell+1} -
O\left(\frac{1}{2^J \ell}\right)$.
\end{lemma}
\begin{proof}
After possibly replacing $\xi$ with $-\xi$ we may take $\xi = 2^{-j_1} -
2^{-j_2}
+ O\left(2^{-\ell-J}\right)$ with $j_1 < j_2$. Then
\begin{align*} &1-\hat{\mu}_{A_{2,p}}(\xi) = \frac{2}{2\ell+1} \sum_{i =
0}^{\ell-1}\left(1-\cos\left(2\pi \left(2^{i -
j_1} - 2^{i-j_2} + O\left(2^{i -\ell -J}\right)\right)\right)\right) \\&=
O\left(\frac{1}{2^J \ell }\right) + \frac{2}{2\ell+1}
\left[\sum_{i = 0}^{j_1-1}\left( 1-\cos\left(2\pi \left(2^{i - j_1} -
2^{i-j_2}\right)\right) \right)+\sum_{i =j_1}^{j_2-1} \left(1 - \cos\left(2\pi
2^{i-j_2}\right)\right)\right] \\ &\leq O\left(\frac{1}{ 2^J \ell}\right) +
\frac{2}{2\ell+1}
\left[\sum_{i = -\infty}^{j_1-1}\left( 1-\cos\left(2\pi 2^{i - j_1}\right)
\right)+\sum_{i =-\infty}^{j_2-1} \left(1 - \cos\left(2\pi
2^{i-j_2}\right)\right)\right] \\&=  \frac{4c_0}{2\ell+1} + O\left(\frac{1}{2^J
\ell}\right).\end{align*}
\end{proof}
\begin{lemma}\label{FC_upper_bound}
 For all but $O\left(J\ell^3\right)$ pairs $\xi_1 \neq \xi_2 \in B_p$
$$\hat{\mu}_{A_{2,p}}(\xi_1 - \xi_2) = 1 - \frac{8 c_0}{2\ell+1} +
O\left(\frac{1}{2^J \ell}\right).$$
For all but $O\left(J^2 \ell^2\right)$ pairs $\xi_1 \neq \xi_2 \in B_p$,
$$\hat{\mu}_{A_{2,p}}\left(\xi_1 - \xi_2\right) \leq 1 - \frac{6c_0}{2\ell+1} +
O\left(\frac{1}{2^J \ell}\right).$$
For all but $O\left(J^3 \ell\right)$ pairs $\xi_1 \neq \xi_2 \in B_p$,
$$\hat{\mu}_{A_{2,p}}(\xi_1 - \xi_2) \leq 1 - \frac{4c_0}{2\ell+1} +
O\left(\frac{1}{2^J \ell}\right).$$
\end{lemma}
\begin{proof}
 Write $\frac{\xi_1}{p} = 2^{-j_1} - 2^{-j_2} + O\left(2^{-\ell-J}\right)$,
$\frac{\xi_2}{p} =
2^{-j_3} - 2^{-j_4} + O\left(2^{-\ell - J}\right)$.  By excluding at most
$O\left(J\ell^3\right)$ quadruples
$(j_1, j_2, j_3, j_4)$ we may assume $j_i > J$ for all $i$ and $|j_i - j_k| \geq
J$ for all $i\neq k$ in $\{1, 2, 3, 4\}$. Then
\begin{align*} 1-\hat{\mu}_{A_{2,p}}(\xi) &= \frac{2}{2\ell+1} \sum_{i =
0}^{\ell-1}\left(1-\cos\left(2\pi \left(2^{i -
j_1} - 2^{i-j_2} - 2^{i - j_3} + 2^{i - j_4} + O\left(2^{i -\ell
-J}\right)\right)\right)\right) \\&=
O\left(\frac{1}{ 2^J \ell}\right) +  \frac{2}{2\ell + 1}
\left[\sum_{k = 1}^4\sum_{i = j_k - J}^{j_k-1}\left( 1-\cos\left(2\pi 2^{i -
j_k}\right)\right)\right]  \\ &= O\left(\frac{1}{ 2^J \ell}\right) +
\frac{8c_0}{2\ell+1}.
\end{align*}
For the second statement, by excluding $O\left(J^2\ell^2\right)$ tuples $(j_1,
j_2, j_3,j_4)
$ we may
assume that three of $j_1, j_2, j_3, j_4$ are larger than $J$ and mutually
separated by at least $J$.  One argues as before, using the additional
calculation that for $1 < j < J$,
$$\sum_{i = 1}^{J} \left(1 - \cos\left(2\pi\left(2^{-i} \pm
2^{j-i}\right)\right)\right) \geq \sum_{i = 1}^J \left(1
- \cos\left(2\pi 2^{-i}\right)\right) = c_0 + O\left(2^{-J}\right),$$ which
holds since for all
$i, j >0$, $$\|2^{-i} \pm 2^{j-i}\|_{\bR/\zed} \geq \|2^{-i}\|_{\bR/\zed}.$$ The
third
statement is similar.
\end{proof}

\begin{proof}[Proof of Theorem \ref{power_2_theorem}, lower bound]
Let $\epsilon > 0$ be given, and suppose that $n < (1-\epsilon)\frac{\ell \log
\ell}{2c_0}$.

Set $J = 2\log \log \ell$ and define $B_p$ as above. It suffices to show that
$B_p$ satisfies
conditions (\ref{mean_grows}) and (\ref{weaker_cut_off_condition}) of
Proposition \ref{Chebyshev_criterion}.

By Lemma \ref{FC_lower_bound}
\begin{align*}\frac{1}{\sqrt{|B_p|}} \sum_{\xi \in B_p}
(\hat{\mu}_{A_{2,p}}(\xi))^{n} &\geq
\sqrt{|B_p|} \exp\left[(1-\epsilon)\frac{\ell\log \ell}{2c_0}
\left(\frac{-4c_0}{2\ell+1} +
O\left(\frac{1}{2^J \ell}\right)\right)\right] \\&\geq \sqrt{|B_p|}
\ell^{\epsilon-1}\left(1 +
O\left(\frac{\log \ell}{2^J} \right)\right).\end{align*}
In particular
 Lemma \ref{big_B} implies
$$\frac{1}{\sqrt{|B_p|}} \sum_{\xi \in B_p} (\hat{\mu}_{A_{2,p}}(\xi))^{n}
\gg \frac{\ell^{\epsilon }}{2^{\frac{J}{2}}} = \ell^{\epsilon - o(1)}$$ and
condition
(\ref{mean_grows}) is satisfied.

To check (\ref{weaker_cut_off_condition}), split $\xi_1, \xi_2 \in B_p$
according as $\xi_1 = \xi_2$, or $\xi_1, \xi_2$ fall into one of the several
cases enumerated in Lemma  \ref{FC_upper_bound}.  This gives
\begin{align*}\sum_{\xi_1, \xi_2 \in B_p} \left(\hat{\mu}_{A_{2,p}}(\xi_1 -
\xi_2)\right)^{n} &\leq |B_p|+ |B_p|^2
\exp\left[(1-\epsilon)\frac{\ell \log \ell}{2c_0} \left(-\frac{4 c_0}{\ell} +
O\left(\frac{1}{2^J \ell}\right)\right)\right]\\& \qquad
+ O\left(J\ell^3\right)\exp\left[(1-\epsilon)\frac{\ell \log
\ell}{2c_0}\left(\frac{-3c_0}{\ell} + O\left(
\frac{1}{2^J \ell}\right)\right)\right] \\&\qquad + O\left(J^2 \ell^2\right)
\exp\left[(1-\epsilon)\left(\frac{\ell\log \ell}{2c_0} \left(\frac{-2c_0}{\ell}
+
O\left(\frac{1}{2^J \ell}\right)\right)\right)\right]\\&\qquad  + O\left(J^3
\ell\right).
\end{align*}
By Lemma \ref{big_B}, $|B_p| = \ell^{2 - o(1)}$, and thus all but the second
term is an error term. Condition (\ref{weaker_cut_off_condition}) holds, since
\[
 |B_p|^2
\exp\left[(1-\epsilon)\frac{\ell \log \ell}{2c_0} \left(-\frac{4
c_0}{\ell}\right)\right] \leq (1 + o(1))\left(\sum_{\xi \in B_p}
(\hat{\mu}_{A_{2,p}}(\xi))^{n} \right)^2.
\]

\end{proof}

\subsection{Proof of Theorem \ref{power_2_theorem}, upper bound}
We prove the following somewhat more precise estimate.

\begin{proposition}\label{upper_bound_prop}
For all $0 < \beta
< \log \ell$, for all  $n \geq \frac{\ell}{2c_0}(\log \ell + \beta)$
we have
$$\|\mu_{A_{2,p}}^{*n} - \bU_{\zed/p\zed}\|_{\TV(\zed/p\zed)}^2 \ll
e^{-\beta}
+ \frac{e^{-\frac{\beta}{c_0}}\log \ell }{\ell^{\frac{1}{c_0}}}.$$
\end{proposition}
\begin{remark}
 The second term results from a discrepancy between the eigenvalue generating
the spectral gap and the bulk of the large spectrum which determines the
mixing time. With more effort, the factor of $\log \ell$ could be removed.
\end{remark}

The proof uses the following frequently used application of the Cauchy-Schwarz inequality, see \cite{D88} for an introduction to these types of estimates, also \cite{D10}.

\begin{lemma}\label{cauchy-schwarz} Let $\mu$ be a probability measure on
finite abelian group $G$.  We have the upper bound
$$ \left\|\mu - \bU_G\right\|_{\TV(G)} \leq \frac{1}{2}\left(\sum_{0 \neq \chi
\in \widehat{G}}
\left|\hat{\mu}(\chi)\right|^2\right)^{\frac{1}{2}}.$$
In particular,
\begin{equation}\label{cs_bound}\| \mu_{A_{2,p}}^{*n} -
\bU_{\zed/p\zed}\|_{\TV(\zed/p\zed)} \leq
\frac{1}{2}\left( \sum_{0 \not \equiv \xi \bmod p}
\left|\hat{\mu}_{A_{2,p}}(\xi)\right|^{2n}
\right)^{\frac{1}{2}}.\end{equation}

\end{lemma}
\begin{proof}
We have
$$\left\|\mu - \bU_G\right\|_{\TV(G)} = \frac{1}{2} \sum_{x \in G} \left|\mu(x)
- \bU_G(x)\right|.$$ Hence, by
Cauchy-Schwarz,
\begin{align*} \left\|\mu- \bU(G)\right\|_{\TV(G)} &\leq \frac{1}{2} \left(|G|
\sum_{x \in G}
\left|\mu(x) - \bU_G(x)\right|^2\right)^{\frac{1}{2}} = \frac{1}{2}
\left(\sum_{0 \neq \chi \in \widehat{G}} \left|\hat{\mu}(\chi)\right|^2
\right)^{\frac{1}{2}}. \end{align*}
\end{proof}

The above lemma reduces to estimation of the size of the Fourier
coefficients $\hat{\mu}_{A_{2,p}}(\xi)$. In estimating these coefficients it
will
be convenient to use the following modified binary expansion of $\frac{\xi}{p}$.
\begin{lemma}
 Let $p \geq 3$ be prime.  For each $0 \not \equiv \xi \bmod p$ there is an
increasing sequence $\sI = \{i_j\}_{j=1}^{\infty}\subset \zed_{>0}$,  and
$\epsilon = \pm 1$ such that
\[
\frac{\xi}{p} \equiv \epsilon \sum_{j=1}^\infty (-1)^j 2^{-i_j} \bmod 1.
\]
This representation is unique.
\end{lemma}
\begin{proof}
 Write $-\frac{\xi}{p}$ in binary as $*.s_1 s_2 s_3...$ with each $s_i \in
\{0,1\}$, then write $\frac{\xi}{p} = -\frac{\xi}{p}-
\left(-\frac{2\xi}{p}\right)$
where $\left(-\frac{2\xi}{p}\right)$ is obtained by a left shift, and then the
subtraction is performed bitwise.

The uniqueness follows because any two distinct such representations
$(\epsilon,\{i_j\}),  (\epsilon', \{i'_j\})$ differ by $\gg 2^{-J}$, where $J$
is
$\min(i_1, i'_1)$ if $\epsilon \neq \epsilon'$, and otherwise is the least
integer  which appears in the symmetric difference
$\{i_j\}\Delta\{i'_j\}$.
\end{proof}

\subsubsection{Index sequences}
We introduce several notions which will
be useful in the remainder of the argument.

 Given a real parameter $J>0$, define a \emph{$J$-sequence} of non-negative
integers to be an ordered set $A \subset \zed_{\geq 0}$, with members enumerated
$A = a_1 < a_2 < ...$ such that any pair of consecutive elements differ by at
most $J$. $|A|$ denotes the cardinality. Set $i(A) = a_1$, $t(A) = \sup(A)$. A
$J$-sequence with $a_1 = 0$ is called \emph{normalized}. Given $J$-sequence $A =
a_1 < a_2 < ...$, its \emph{off-set sequence} is the normalized $J$-sequence $A'
= 0 < a_2-a_1 < a_3-a_1 <
...$. For instance, 
\[
 1,3,7,8,10,14
\]
is a 4-sequence with offset sequence
\[
 0,2,6,7,9,13.
\]
A $J$-sequence is called \emph{non-trivial} if it contains a pair of
elements that differ by more than 1.
We denote $\sJ$ the set of $J$-sequences, $\sJ_0$ the set of normalized
$J$-sequences and $\sJ_0' = \sJ_0 \setminus \{\{0\}, \{0,1\}\}$ the set of
non-trivial normalized $J$-sequences.

A $J$-sequence $A$ contained in sequence $B \subset \zed_{\geq 0}$ is
called a $J$-subsequence. We say that $J$-subsequence $A \subset
B$ is \emph{maximal} if it is not properly contained in another $J$-subsequence
$A' \subset B$. 
Given parameter $J$, one easily checks that any $B \subset \zed_{\geq 0}$ has a
unique partition into maximal $J$-subsequences. 
For instance, in the first sequence above,
\[
 1,3; \quad 7,8,10; \quad 14
\]
is a partition into maximal 2-subsequences.

We write $\sC(B)$ for the set
of maximal $J$-subsequences of $B$. The $J$-sequences in $\sC(B)$ are
\emph{$J$-separated} in the sense that if $A_1 \neq A_2 \in \sC(B)$ and $x_1
\in A_1, x_2 \in A_2$ then $|x_1-x_2| > J$.  The sequences in $\sC(B)$ are
naturally ordered by, for $A_1, A_2 \in \sC(B)$, $A_1 < A_2$ if and only if
for any $x_1 \in A_1, x_2 \in A_2$, $x_1 < x_2$.

In the remainder of the argument we think of the non-zero bits in the expansion of $\frac{\xi}{p}$ above as partitioned into maximal $J$-sequences.  These $J$-separated parts do not interact significantly in calculating the Fourier transform. The argument that follows quantifies the interaction. 

Let $J \geq \log_2 \ell$  be a parameter.
Given $\xi \bmod p$, represent $\xi$ as  $(\sI(\xi), \epsilon(\xi))$ as above.
Truncate $\sI(\xi)$ to $\sI'(\xi)
= \sI(\xi) \cap (0,\ell]$ (note that $\epsilon$ and $\sI'$ determine $\xi$) and
set \begin{equation}\label{def_sigma}\sigma(\xi) = |\sI'(\xi)|, \qquad \sC(\xi)
= \sC(\sI'(\xi)).\end{equation} We call
$\sC(\xi)$ the set of \emph{clumps} of
$\xi$, each clump being a $J$-sequence.   If there exists $C \in \sC(\xi)$ with
$i(C)\leq J$ we say that $C$ is \emph{initial}. A
clump $C$ with $t(C)  > \ell-J$ is
\emph{final}. We write $C_{\ini}(\xi)$, $C_{\fin}(\xi)$ for the initial and
final clump, with the convention that $C_{\ini}= \emptyset$ if there is no
initial clump, and similarly $C_{\fin}$.  A clump is \emph{typical} if it
is neither initial nor final.
$\sC_0(\xi)\subset \sC(\xi)$ is the subset of typical clumps.

Given frequency $\xi$, define the \emph{savings} of $\xi$ to be
\begin{equation}\label{def_sav_xi}\sav(\xi) = \frac{2 \ell+1}{2}\left(1 -
\hat{\mu}_{A_{2,p}}(\xi)\right)=
\sum_{l=0}^{\ell-1}\left(1 - \cos\left(2\pi\left(\sum_{k=1}^\infty (-1)^k
2^{l-i_k} \right) \right) \right).\end{equation} For a typical clump $C \in
\sC_0(\xi)$ also
define
\begin{equation}\label{sav_def}\sav(C) = \sum_{i(C)-J \leq l <
t(C) } \left[1 - \cos
\left(2\pi\sum_{i_k\in C} (-1)^k 2^{l-i_k}\right)\right]. \end{equation}

\begin{lemma}\label{fc_approx}
We have
$$\sav(\xi) \geq \sum_{C \in \sC_0(\xi)} \sav(C)+ |C_{\ini}(\xi)| +
|C_{\fin}(\xi)| + O\left(2^{-J}|\sC|\right).$$
\end{lemma}
\begin{proof}
Since the clumps $C \in \sC$ are $J$-separated, we have
\begin{align*}\sav(\xi)  &\geq \sum_{C \in \sC_0(\xi)} \sum_{i(C)-J \leq
l <
t(C)} \left[1 - \cos
\left(2\pi\sum_{i_k\in C} (-1)^k 2^{l-i_k}\right)+ O\left(2^{-J
-t(C)+l}\right)\right] \\ &+ \sum_{i \in
C_{\ini}(\xi)}\left(1 - \cos\left(\frac{\pi}{2}
\right)\right)+ \sum_{i \in
C_{\fin}(\xi)}\left(1-\cos\left(\frac{\pi}{2}\right)\right),\end{align*} where
in the last
two sums we specialize to $j = i-1$, and note that for any fixed $l$
$\frac{1}{2}\geq \left|\sum_{i_j \geq
l}(-1)^{i_j}2^{l-i_j-1}\right| \geq \frac{1}{4}.$
\end{proof}

In a similar spirit we have the following crude estimate for
savings.

\begin{lemma}\label{first_comp}
Let $0 \not \equiv \xi \in \widehat{\zed/p\zed}$ and let $C \in \sC_0(\xi)$.  We
have $\sav(C) \geq |C|$.
\end{lemma}

\begin{proof}
Write $C = i_1 < \cdots <i_j$.
We have
$$\sav(C) = \sum_{l =  i_1-J}^{i_j -1} \left[1-\cos
\left(2\pi\sum_{m = 1}^j(-1)^m
2^{l -i_m}\right)\right] \geq \sum_{m = 1}^j \left(1-\cos\left(
\frac{\pi}{2}\right)\right)$$
by specializing to $l = i_m - 1$, $m = 1, ..., j$.
\end{proof}
\begin{lemma}\label{many_digits}Let $0 \not \equiv \xi \in
\widehat{\zed/p\zed}$.  For fixed $\delta_3 > 0$, for $\sigma(\xi)$ as in (\ref{def_sigma}),
$$\left|\hat\mu(\xi)\right| \leq \max \left(1 - \frac{2\sigma(\xi)}{2\ell+1},
1-\delta_3 \right).$$
\end{lemma}

\begin{proof} The bound $\hat{\mu}(\xi) \leq 1 - \frac{2\sigma(\xi)}{2\ell+1}$
follows from Lemmas \ref{fc_approx} and \ref{first_comp}.
The bound
$\hat{\mu}(\xi)
\geq 1 - \delta_3$ follows from $\frac{1}{2}(\cos \theta + \cos 2\theta) \geq -1
+ c$ for a fixed $c >0$.
\end{proof}

For typical clumps we require slightly stronger estimates.

\begin{lemma}\label{digit_shift}
Let $0 \not \equiv \xi \in \widehat{\zed/p\zed}$.  Let $C\in \sC_0(\xi)$, and
suppose that $|C| = j > 1$.
Enumerate $C = i_1 < i_2 < ...< i_j$.
There exists fixed $\delta_1 >0$ such that if $i_2 > i_1 + 1$ then $\sav(C) >
\sav(C') + \delta_1$ where $C'$ is the $J$-sequence formed by $i_2- 1, i_2 ,
...,
i_j$, i.e. by shifting $i_1$ to the place adjacent to $i_2$.
\end{lemma}

\begin{proof}
We have
\begin{align*}&\sav(C) - \sav(C')= \\& \sum_{i_1 -J \leq l < i_2 - 1} \left[1 -
\cos\left(2\pi \sum_{m = 1}^j (-1)^m2^{l-i_m}\right)\right] - \\
&\qquad\qquad\qquad \sum_{i_2 - J - 1 \leq l < i_2 - 1} \left[ 1 - \cos \left(
2\pi
\left(-2^{l+1 - i_2} + \sum_{m = 2}^j(-1)^m 2^{l - i_m}\right)\right)\right].
\end{align*}
Set
\[
x =  \sum_{m=3}^j (-1)^{m-1} 2^{-i_m + i_2 -1}, \qquad 0 \leq x \leq
\frac{1}{4}.
\]
Take only the first $J$ terms of the first sum, to obtain for some $\delta_1 >
0$
\begin{align*}
\sav(C)- \sav(C')\geq \sum_{l = 1}^J \left[\cos\left(\frac{2\pi}{2^l}
\left(\frac{-3}{4} -
\frac{x}{2}\right)\right)-\cos\left(\frac{2\pi}{2^l}
\left(\frac{-1}{2}
- x\right)\right) \right] >
\delta_1 \end{align*} by noting that the worst case is $i_1 = i_2 -2$.
\end{proof}

\begin{lemma} \label{digit_drop}
Let $0 \not \equiv \xi \in \widehat{\zed/p\zed}$. Let $C\in \sC_0(\xi)$, and
suppose that $|C| = j > 1$, $C = i_1 < \cdots
< i_j$ with $i_2 = i_1 + 1$.  Then $\sav(C) \geq \sav(C')$ where $C'$ is the
$J$-sequence formed by $i_2, ..., i_j$, i.e. by dropping $i_1$.  Furthermore, if
$j
\geq 3$ and $i_3 = i_1 + 2$ then there exists fixed $\delta_2 > 0$ such that
$\sav(C) > \sav(C') + \delta_2$.
\end{lemma}

\begin{proof}
We have
\begin{align} \notag &\sav(C) - \sav(C') =\\ \notag &\sum_{l = i_1 - J}^{i_1-1}
\left[1 - \cos\left(2\pi \sum_{m = 1}^j (-1)^m 2^{l-i_m}\right)\right] - \sum_{l
= i_1 - J +1}^{i_1-1} \left[1 - \cos\left(2 \pi \sum_{m = 2}^j (-1)^m
2^{l-i_m}\right)\right].
\end{align}
Replace $l$ with $i_1 -1 - l$ and set $x = \sum_{m = 2}^j (-1)^m
2^{i_m - i_1-1}$, $\frac{1}{8} \leq x \leq \frac{1}{4}$ to obtain
\begin{align}
\label{drop_bound} \sav(C)- \sav(C')&  \geq \sum_{l = 0}^{J -2} \left[
\cos
\left(\frac{2\pi x}{2^l}\right) - \cos \left(\frac{2 \pi
(x-\frac{1}{2})}{2^l}\right)\right]. \end{align}  In the case $j \geq
3$ and $i_3 = i_1 +2$ we have $x \leq \frac{3}{16}$  which proves
$$ (\ref{drop_bound}) \geq \cos\left(\frac{3 \pi}{8}\right) - \cos\left(\frac{5
\pi}{8}\right).$$
\end{proof}

 The previous two lemmas imply the following one.
\begin{lemma}\label{component_bound}
Let $0 \not \equiv \xi \in \widehat{\zed/p\zed}$. Let $C \in \sC_0(\xi)$ be a
typical clump of $\sC(\xi)$ with digits $i_1 < i_2< ...< i_j$. If $j = 1$ or
$j=2$ and $i_2 = i_1 + 1$ we have
\[
 \sav(C) \geq c_0 + O(2^{-J}).
\]
Furthermore, there is a $\delta > 0$ such that, if $j \geq 3$ or $j = 2$ and
$i_2 > i_1 +1$ then
\[
 \sav(C) \geq c_0 + \delta j.
\]

\end{lemma}

\begin{proof} By a sequence of steps in which we either (i) move the first index
of $C$ adjacent to the second, or (ii) delete the first, we reduce to case of
$C_0$ containing a single element, which satisfies $\sav(C_0) = c_0 -
O(2^{-J})$.
\end{proof}

We collect together several easy combinatorial estimates.   Given frequency
$\xi$ we are most interested in typical clumps $C \in \sC_0(\xi)$
which consist of a single index, or a pair of adjacent indices.  Let the number
of
these be $x_1(\xi)$ and $x_2(\xi)$.  Let $x_3(\xi) = |\sC_0(\xi)|-x_1(\xi)-
x_2(\xi)$ be the number of non-trivial clumps in $\sC_0(\xi)$, and let $m =
\sigma(\xi)- |C_{\ini}| - |C_{\fin}| - x_1(\xi) - 2 x_2(\xi)$ be the number of
indices contained in the clumps counted in $x_3(\xi)$.

Given $m\geq 0$ and $x_3 \geq 0$, let \[\sT(m, x_3) = \left\{\uA \in
(\sJ_0')^{x_3}: \sum_{i=1}^{x_3}|A_j| = m\right\} \]
be the collection of $x_3$-tuples of non-trivial normalized $J$-sequences of
total cardinality $m$.
Given initial and final clumps $C_{\ini}$ and $C_{\fin}$, $T \in \sT(m, x_3)$
and
integers $x_1, x_2 \geq 0$, let $\sN(C_{\ini}, C_{\fin}, x_1, x_2, T)$ denote
the number of $\xi$
with initial clump $C_{\ini}$, final clump $C_{\fin}$, $x_1$
typical clumps with a single index, $x_2$ typical
clumps which consist of a pair of consecutive indices and $x_3$ non-trivial
typical clumps,
whose
offsets taken in order are given by
$T$.  For any $j \geq 0$, let $I(j)$ (resp. $F(j)$) be the number of
$J$-sequences on $j$ indices which may appear as the initial (resp. final) clump
of $\sI(\xi)$, $\xi \in \widehat{\zed/p\zed} \setminus \{0\}$.

\begin{lemma}\label{counting_lemma}
 Let $x_1, x_2, x_3, m, T$ be as above and let $C_{\ini}$, $C_{\fin}$ be any
initial and final clumps (possibly empty).  We have the bounds
 \[
  |\sT(m, x_3)| \leq (J+1)^{m-1}
 \]
 and, for any $T \in \sT(m, x_3)$,
 \begin{align*}
  \sN(C_{\ini}, C_{\fin}, x_1, x_2, T) &\leq 2\frac{\ell^{
x_1+x_2+x_3}}{x_1!x_2!x_3!}.
 \end{align*}
Also, for any $j \geq 0$,
\[
I(j), F(j) \leq J^j.
\]
\end{lemma}

\begin{proof}
 To bound $|\sT|$, neglecting $x_3$ and the non-triviality condition, choose for
each index $1 \leq j < m$ a
distance $1\leq d(j) \leq J+1$ between $j$ and $j+1$ in the arrangement, with a
distance of $J+1$ indicating that a new clump begins with $j+1$.

Similarly, the bound for $I(j)$ follows on choosing a first index in one of at
most $J$ ways, and then choosing sequentially distances between the consecutive
indices.  For $F(j)$, choose counting from the back instead.

The bound for $\sN(C_{\ini}, C_{\fin}, x_1, x_2, T)$ follows on choosing a first
index for each
clump, the factor of 2 coming from choosing the sign.
\end{proof}

Our results on savings may be summarized as follows.
\begin{lemma}\label{coordinate_savings}
 Let $0 \not \equiv \xi \in \widehat{\zed/p\zed}$ have parameters $x_1, x_2,
x_3, m, C_{\ini}, C_{\fin}$ as above.   There is a fixed $0 < \delta <
\frac{1}{2}$ such
that
 \[
  \sav(\xi) \geq c_0(x_1 + x_2 + x_3) + \delta m +|C_{\ini}| + |C_{\fin}|-
O\left(\frac{x_1+x_2+x_3}{2^J}\right).
 \]
\end{lemma}

\begin{proof}[Proof of Proposition \ref{upper_bound_prop}]
Let $\log J = o(\log \ell)$ and fix some $\theta$, $1 -
\frac{1}{c_0} < \theta < 1$.
By Lemma \ref{cauchy-schwarz}
\begin{align*} \left\|\mu_{A_{2,p}}^{*n} -
\bU_{\zed/p\zed}\right\|_{\TV(\zed/p\zed)}^2 &\leq \frac{1}{4} \sum_{\xi \not
\equiv
0 \bmod p} \left|\hat{\mu}_{A_{2,p}}(\xi)\right|^{2n} \\ &=
\frac{1}{4}\left\{\sum_{1 \leq \sigma(\xi) <
\ell^\theta} + \sum_{ \ell^\theta \leq \sigma(\xi) < \delta_3 \ell} +
\sum_{\delta_3 \ell \leq
\sigma(\xi)} \right\}
 \left|\hat{\mu}_{A_{2,p}}(\xi)\right|^{2n}\\&=
\frac{1}{4}\left(\sS_1 + \sS_2 + \sS_3\right).\end{align*}
By Lemma \ref{many_digits}, for some $c >0$,
\begin{equation}\label{S_3_bound}\sS_3 \leq 2^\ell
\left(1-\frac{2\ell\delta_3}{2\ell+1}\right)^{\frac{\ell}{c_0}(\log \ell+\beta)}
= O(e^{-c
\ell \log \ell}).\end{equation}
By Lemma \ref{many_digits}, again for some $c > 0$,
\begin{align}\notag  \sS_2 &\leq 2\sum_{\ell^\theta \leq j < \delta_3 \ell}
\binom{\ell}{j} \left(1 - \frac{2j}{2\ell+1}\right)^{\frac{\ell}{c_0}(\log \ell
+\beta)}\\\notag &\ll
\sum_{\ell^\theta \leq j < \delta_3 \ell} \exp\left(j\log \ell - j\log j + j -
\frac{2\ell j}{(2\ell+1)c_0} \log \ell - \frac{2j\ell\beta}{c_0(2
\ell+1)}\right) \\
\label{S_2_bound} &\ll
e^{-c\ell^\theta \beta}\sum_{\ell^\theta \leq j < \delta_3 \ell} \exp\left(
\left(-\theta + 1 -
\frac{1}{c_0}\right)j \log \ell + j\right) =
O\left(e^{-c \ell^\theta (\beta + \log \ell)}\right).\end{align}

 Conditioning on
$x_1(\xi), x_2(\xi), x_3(\xi), m$ as in Lemma
\ref{counting_lemma} and $i = |C_{\ini}|$, $f = |C_{\fin}|$ we find
\begin{align*}\sS_1 &\leq 2\sum_{1 \leq j < \ell^\theta} \sum_{x_1 + 2
x_2 + m +i + f= j}\sum_{|C_{\ini}| = i, |C_{\fin}|=j} \sum_{x_3 \leq \lfloor
\frac{m}{2}\rfloor}\\&\qquad \times
\sum_{T\in \sT(m, x_3)} \sN(C_{\ini}, C_{\fin}, x_1, x_2, T) \left(1 -
\frac{2 \sav}{2\ell+1}\right)^{\frac{\ell}{c_0}(\log \ell +
\beta)}
\end{align*}
where $\sav = c_0(x_1 + x_2 + x_3) + \delta m + i + f -
O\left(\frac{x_1+x_2+x_3}{2^J}\right)$. Inserting the estimates for $|\sT|$ and
$\sN$
from Lemma \ref{counting_lemma}, we obtain
\begin{align*}&\sS_1 \ll \sum_{1 \leq j < \ell^\theta} \sum_{x_1 + 2 x_2
+ m + i + f = j} (J+1)^{m+i+f} \frac{\ell^{x_1 + x_2}}{x_1!x_2!}\\&
\sum_{\substack{x_3 \leq \lfloor \frac{m}{2}\rfloor\\ m >0 \Rightarrow x_3 > 0}}
\frac{\ell^{x_3}}{x_3!} \exp\left(-\left(x_1 + x_2 + x_3 + \frac{\delta m
+ i + f}{c_0}
-O\left(\frac{x_1+x_2+x_3}{2^J}\right)\right)(\log \ell + \beta)\right).
\end{align*}Assume that $\frac{\log
\ell}{2^J} = o(1)$.  Then, when $m \geq 1$ we find that the sum over $x_3$
is \[O\left(\exp(-\beta) \right).\] The terms for which $x_1 = x_2 =
0$ thus contribute
$O\left(\frac{J}{e^{\beta}\ell^{\frac{\delta}{c_0}+o(1)}}  +
\frac{J}{e^{\frac{\beta}{c_0}} \ell^{\frac{1}{c_0}}}\right)$.
When $x_1 + x_2 \neq 0$ summation over $i,f, m$ reduces to $1 + o(1)$.
Thus
\begin{align*}
\sS_1
\leq & O\left(\frac{J}{e^{\frac{\beta}{c_0}}
\ell^{\frac{1}{c_0}}} \right)+O\left(\exp\left(-\beta
+
O\left(\frac{\log \ell + \beta}{2^J} \right) \right) \right).\end{align*}
Choose $2^J = \ell$ to complete the proof.
\end{proof}
\appendix
\section{Local limit theorem on $\bR^k$}\label{local_limit_appendix}
For $k > 1$ recall that we define the measure on $\zed^k$, \[\nu_k =
\frac{1}{2k+1}\left(\delta_0 + \sum_{j =1}^k (\delta_{e_j} + \delta_{-e_{j}})
\right)\] and that we write \[\eta_k\left(\sigma, \ux\right) = \frac{1}{(2\pi
\sigma^2)^{\frac{k}{2}}}\exp\left(- \frac{\|\ux\|_2^2}{2 \sigma^2} \right)\] for
the density of the centered standard normal distribution on $\bR^k$.  In this
appendix we prove Lemma \ref{normal_approximation_lemma}, which we recall for
convenience.
\begin{lemma*} Let $n, k(n) \geq 1$ with $k^2 = o\left(n\right)$
for
large $n$. As $n \to \infty$ we have
 \[
\left\| \nu_k^{*n} \ast \one_{\left[-\frac{1}{2}, \frac{1}{2}\right)^k} -
\eta_k\left(\sqrt{\frac{2n}{2k+1}}, \cdot\right)\right\|_{\TV(\bR^k)} = o(1).
 \]
\end{lemma*}

We actually prove a stronger estimate, which is a local limit
theorem
on $\bR^k$ for which we don't know an easy reference.

\begin{lemma}\label{gauss_approximation}
 Let $n, k(n) \geq 1$ with $k^2 = o\left(n\right)$ for
large
$n$.
Uniformly for $\ualpha \in \zed^k$ such that $\|\ualpha\|_2^2 \leq
\frac{2kn}{2k+1} +
\frac{n\log
n}{\sqrt{k}}$, and $\|\ualpha\|_4^4 \ll
\frac{n^{2}}{k}\left(1 +
\frac{\log n}{\sqrt{k}}\right)$, as
$n \to \infty$,
\[\nu_k^{*n}(\alpha) = \left\{1 +
o(1)\right\}
\eta_k\left(\sqrt{\frac{2n}{2k+1}}, \alpha \right).\]
\end{lemma}

The deduction of Lemma \ref{normal_approximation_lemma} is as follows.
\begin{proof}[Proof of Lemma \ref{normal_approximation_lemma}]
 We have, for any $A, \delta > 0$, and for some $C>0$,
 \begin{align*}\int_{\|\ux\|_2^2 > \frac{2kn}{2k+1} + \delta\frac{n \log
n}{\sqrt{k}}} \eta_k\left(\sqrt{\frac{2n}{2k+1}}, \ux \right)d\ux
&=O_{\delta, A}\left(n^{-A}
\right)\\
\int_{\|\ux\|_4^4 > C\frac{n^2}{k}\left(1 + \frac{\log n}{\sqrt{k}} \right)}
\eta_k\left(\sqrt{\frac{2n}{2k+1}}, \ux \right)d\ux
&=O_A\left(n^{-A}
\right)
\end{align*}
see Lemma \ref{gaussian_length_concentration}, so it suffices to estimate the
difference
\[
  \nu_k^{*n} \ast \one_{\left[-\frac{1}{2}, \frac{1}{2}\right)^k} (\ux) -
\eta_k\left(\sqrt{\frac{2n}{2k+1}}, \ux\right)
\]
for $\|\ux\|_2^2 \leq \frac{2kn}{2k+1} + O\left(\frac{n \log
n}{\sqrt{k}}\right)$  and $\|\ux\|_4^4 \ll
\frac{n^2}{k}\left(1 + \frac{\log n}{\sqrt{k}} \right)$.

For
$\ux \in \zed^k$ satisfying this  upper bound and for $\uy \in
[-\frac{1}{2},
\frac{1}{2})^k$, 
\begin{align*}
\eta_k \left(\sqrt{\frac{2n}{2k+1}}, \ux + \uy \right) &= \eta_k
\left(\sqrt{\frac{2n}{2k+1}}, \ux \right)\exp\left(-\frac{2k+1}{4n}\left(2 \ux
\cdot \uy + \|\uy\|_2^2\right) \right)\\
&= (1 + o(1))\eta_k
\left(\sqrt{\frac{2n}{2k+1}}, \ux \right) \exp\left(-\frac{(2k+1)\ux\cdot
\uy}{2n}\right).
\end{align*}
Therefore
\begin{align*}
 &\int_{\left[-\frac{1}{2},
\frac{1}{2}\right)^k}\left|\eta_k\left(\sqrt{\frac{2n}{2k+1}}, \ux + \uy \right)
- \eta_k\left(\sqrt{\frac{2n}{2k+1}}, \ux \right)\right|d\uy\\
&= \eta_k\left(\sqrt{\frac{2n}{2k+1}}, \ux \right) \left(o(1) +
(1+o(1))\int_{\left[-\frac{1}{2}, \frac{1}{2}\right)^k}\left|
\exp\left(-\frac{(2k+1)\ux \cdot \uy}{2n} \right)-1 \right|d\uy \right).
\end{align*}
We claim that for all $\|\ux\|_2^2 \ll \frac{2kn}{2k+1} + \frac{n \log
n}{\sqrt{k}}$,
\begin{equation}\label{hypercube_conc}\int_{\left[-\frac{1}{2},
\frac{1}{2}\right)^k}\left|
\exp\left(-\frac{(2k+1)\ux \cdot \uy}{2n} \right)-1 \right|d\uy  =
o(1).\end{equation}
To see that this suffices for the proof, let \[\sB = \left\{\ux \in \zed^k: \|\ux\|_2^2 \leq \frac{2kn}{2k+1} + \frac{n \log n}{\sqrt{k}},\; \|x\|_4^4 \leq C \frac{n^2}{k}\left(1 + \frac{\log n}{\sqrt{k}} \right)\right\}\]
and estimate
\begin{align*}
 &\sum_{\ux \in \sB} \int_{\left[-\frac{1}{2}, \frac{1}{2}\right)^k} \left|\eta_k\left(\sqrt{\frac{2n}{2k+1}}, \ux + \uy\right) - \eta_k\left(\sqrt{\frac{2n}{2k+1}}, \ux\right) \right|d\uy\\& \leq o(1) \sum_{\ux \in \sB} \eta_k\left(\sqrt{\frac{2n}{2k+1}}, \ux \right)
 \\ & = o(1)\sum_{\ux \in \sB} (1 + o(1))\eta_k^{*n}(\ux) = o(1) 
\end{align*}
where in the last line we apply Lemma \ref{gauss_approximation}.  Since \[\sum_{\ux \in \sB} \int_{\left[-\frac{1}{2}, \frac{1}{2}\right)^k}\eta_k\left(\sqrt{\frac{2n}{2k+1}}, \ux + \uy\right) d\uy = 1+o(1)\]
it follows that $\sum_{\ux \in \sB} \eta_k^{*n}(\ux) = 1+o(1)$ so that 
 \begin{align*}
&\left\| \nu_k^{*n} \ast \one_{\left[-\frac{1}{2}, \frac{1}{2}\right)^k} -
\eta_k\left(\sqrt{\frac{2n}{2k+1}}, \cdot\right)\right\|_{\TV(\bR^k)} \\&= \sum_{\ux \in \zed^k}\int_{\left[-\frac{1}{2},
\frac{1}{2}\right)^k}\left|\eta_k\left(\sqrt{\frac{2n}{2k+1}}, \ux + \uy \right)
- \eta_k\left(\sqrt{\frac{2n}{2k+1}}, \ux \right)\right|d\uy = o(1)
 \end{align*}
by bounding both terms in the sum over $\ux \in \sB^c$ separately.

To prove (\ref{hypercube_conc}), choose a parameter $A = A(n,k) \to \infty$
with $n$ such that $A = o\left(\frac{\sqrt{n}}{k} \right)$ and partition
$\left[-\frac{1}{2}, \frac{1}{2}\right)^k = \sS_{\good} \sqcup \sS_{\bad}$
with
\[
\sS_{\good} = \left\{\uy \in \left[-\frac{1}{2}, \frac{1}{2}\right)^k: |\ux
\cdot \uy| \leq A\sqrt{n}\left(1 + \frac{\log n}{\sqrt{k}}\right)
\right\}.
\]
By Azuma's inequality, for some fixed $C>0$,
\[
 \meas\left(\uy \in \left[-\frac{1}{2}, \frac{1}{2}\right)^k: \left|\ux \cdot
\uy\right| >  t \sqrt{n}\left(1 + \frac{\log n}{\sqrt{k}}\right)\right) \leq 2
\exp\left( - \frac{t^2}{C}\right),
\]
and thus $\meas(\sS_{\good}) = 1 -o(1)$.  Since $\exp\left(-\frac{(2k+1)\ux
\cdot \uy}{2n} \right) = 1+o(1)$ for all $\uy \in \sS_{\good}$ we have
\[
\int_{\uy \in \sS_{\good}}  \left|
\exp\left(-\frac{(2k+1)\ux \cdot \uy}{2n} \right)-1 \right| d\uy = o(1).
\]
Meanwhile,
\begin{align*}
&\int_{\uy \in \sS_{\bad}}  \left|
\exp\left(-\frac{(2k+1)\ux \cdot \uy}{2n} \right)-1 \right| d\uy  \leq
\int_{\uy \in \sS_{\bad}}
\exp\left(\left|\frac{(2k+1)\ux \cdot \uy}{2n} \right|\right) d\uy\\
&= -\int_{A}^{\infty}\exp\left(\frac{t(2k+1)\left(1 + \frac{\log n}{\sqrt{k}}
\right)}{\sqrt{n}} \right)d \meas\left(\uy: |\uy \cdot \ux| > t \sqrt{n}\left(1
+ \frac{\log n}{\sqrt{k}}\right) \right) \\
&\ll \exp\left(\frac{-A^2}{C} + o(1) \right) = o(1).
\end{align*}

\end{proof}

The proof of Lemma \ref{gauss_approximation} is a standard application of the
saddle point method.  As there are several intermediate lemmas, it may help the reader to skip ahead to first read the eventual proof.
Associate to $\nu_k$ the generating function $$f(z_1, ...,
z_k) = \frac{1}{2k+1}\left(1 + z_1 + z_1^{-1} + ... + z_k + z_k^{-1}\right),$$
so that
$\nu_k(\alpha) = C_\alpha[f]$, where for Laurent series in
multiple variables \[g(z_1, ..., z_k) = \sum_{n_1, ..., n_k = -\infty}^\infty
a_{n_1, ..., n_k} z_1^{n_1}...z_k^{n_k}\] we write $C_{\ualpha}[g] =
a_{\ualpha}$.   The generating function associated to $\nu_k^{*n}$
is thus $f^n$.

By symmetry we may assume $\ualpha \geq 0$ coordinatewise.  By Cauchy's theorem,
for any
$R_1, ..., R_k > 0$
\begin{align} \notag
\nu_k^{*n}(\alpha) &= \left(\frac{1}{2\pi i}\right)^k \int_{|z_1| = R_1}\cdots
\int_{|z_k| = R_k} \frac{f(z_1, ..., z_k)^n}{z_1^{\alpha_1} ... z_k^{\alpha_k}}
\frac{dz_1}{z_1}\cdots \frac{dz_k}{z_k}\\
&\label{integral_rep} = \frac{1}{R_1^{\alpha_1}...R_k^{\alpha_k}}
\int_{\left(\bR/\zed\right)^k} f_0(\theta_1, ..., \theta_k)^n
e\left(-\ualpha \cdot \vo\right) d\vo,
\end{align}
where
\[
 f_0(\theta_1, ..., \theta_k) = f(R_1 e(\theta_1), ..., R_k e(\theta_k)).
\]
and $\ualpha \cdot \vo$ is the usual dot product on $\bR^k$.
The asymptotic in Lemma \ref{gauss_approximation} is derived by choosing $R_1,
..., R_k$ such that the phase in $f_0(\vo)^n$ is approximately equal to
$e(\ualpha \cdot \vo)$ for $\vo$ near 0.  The main contribution of the
integral then comes from small $\vo$.

Let $\sD_{\sm}$ be the domain \[\sD_{\sm} = \left\{ \vo \in
\left(\bR/\zed\right)^k : \|\vo\|_\infty \leq \frac{1}{12}\right\}.\] For $\vo
\in \sD_{\sm}$  define
\begin{align*} F(\vo) &=  n \log
\left[\frac{1}{2k+1}\left(1 + R_1 e(\theta_1) + \frac{e(-\theta_1)}{R_1
} + ... + R_k e(\theta_k) + \frac{e(-\theta_k)}{R_k}\right)\right] \\&
\qquad\qquad\qquad\qquad\qquad\qquad\qquad\qquad - (\alpha_1 \log R_1 + ... +
\alpha_k \log R_k) - 2\pi i \ualpha \cdot \vo.
\end{align*}
Evidently $F(\vo)$ gives a continuous
definition of
$\log \frac{f_0(\vo)^n e(-\ualpha \cdot
\vo)}{R_1^{\alpha_1}...R_k^{\alpha_k}}$
on $\sD_{\sm}$.

\begin{lemma}\label{derivatives}
 The first few partial derivatives of $F(\vo)$ are given as follows
\begin{align*}
 D_j F(\vo) &= 2\pi i \left[ \frac{n}{2k+1}\frac{  R_j e(\theta_j) -
\frac{ e(-\theta_j)}{R_j}}{f_0(\vo)} - \alpha_j\right]
\\ D_{j_1} D_{j_2} F(\vo) &= \frac{4\pi^2 n}{(2k+1)^2} \frac{ \left(R_{j_1}
e(\theta_{j_1})- \frac{e(-\theta_{j_1})}{R_{j_1}}\right)\left(R_{j_2}
e(\theta_{j_2}) -
\frac{e(-\theta_{j_2})}{R_{j_2}}\right)}{f_0(\vo)^2}, \;\; j_1 \neq j_2\\
\\ D_j^2 F(\vo) &=- \frac{4\pi^2
n}{2k+1} \frac{R_j e(\theta_j) + \frac{e(-\theta_j)}{R_j}}{f_0(\vo)}
+\frac{4\pi^2 n}{(2k+1)^2} \frac{ \left(R_j
e(\theta_j)- \frac{e(-\theta_j)}{R_j}\right)^2}{f_0(\vo)^2}
\\ D_{j_1} D_{j_2} D_{j_3} F(\vo) &=\\ \frac{-16\pi^3 i n}{(2k+1)^3} &\frac{
\left(R_{j_1}
e(\theta_{j_1})- \frac{e(-\theta_{j_1})}{R_{j_1}}\right)\left(R_{j_2}
e(\theta_{j_2}) -
\frac{e(-\theta_{j_2})}{R_{j_2}}\right)\left(R_{j_3} e(\theta_{j_3}) -
\frac{e(-\theta_{j_3})}{R_{j_3}}\right)}{f_0(\vo)^3}, \\
&\qquad\qquad\qquad\qquad\qquad\qquad\qquad\qquad
\qquad\qquad j_1, j_2, j_3 \text{ distinct}
\\ D_{j_1}^2 D_{j_2} F(\vo) &=
\frac{8 \pi^3 i n}{(2k+1)^2}\frac{\left(R_{j_1} e(\theta_{j_1}) +
\frac{e(-\theta_{j_1})}{R_{j_1}}\right)\left(R_{j_2} e(\theta_{j_2}) -
\frac{e(-\theta_{j_2})}{R_{j_2}}\right)}{f_0(\vo)^2}
\\&  -\frac{16\pi^3 i n}{(2k+1)^3} \frac{ \left(R_{j_1}
e(\theta_{j_1})- \frac{e(-\theta_{j_1})}{R_{j_1}}\right)^2\left(R_{j_2}
e(\theta_{j_2}) -
\frac{e(-\theta_{j_2})}{R_{j_2}}\right)}{f_0(\vo)^3} , \;\; j_1 \neq j_2
\\ D_j^3 F(\vo) &=
-\frac{8 \pi^3 i
n}{2k+1} \frac{R_j e(\theta_j) -
\frac{e(-\theta_j)}{R_j}}{f_0(\vo)}
 \\&+ \frac{24 \pi^3 i n}{(2k+1)^2}\frac{\left(R_j e(\theta_j) +
\frac{e(-\theta_j)}{R_j}\right)\left(R_j e(\theta_j) -
\frac{e(-\theta_j)}{R_j}\right)}{f_0(\vo)^2}\\& \qquad \qquad -\frac{16\pi^3 i
n}{(2k+1)^3} \frac{ \left(R_j
e(\theta_j)- \frac{e(-\theta_j)}{R_j}\right)^3}{f_0(\vo)^3}.
\end{align*}

\end{lemma}

Choose $R_j$ by solving the stationary phase equation, for each $j$,  $D_j F(0)
= 0$, thus
\begin{equation}\label{stationary_phase} \frac{n}{2k+1}\frac{R_j -
\frac{1}{R_j}}{f_0(0)} - \alpha_j = 0.\end{equation}

\begin{lemma}\label{f_0_bound}
Let $ n, k(n) \in \zed_{>0}$ with $k^2 = o(n)$ as $n \to \infty$. Let $\ualpha
\in
\zed^k$ and assume
$\|\ualpha\|_{2}^2 \leq
n\left(1 + \frac{\log n}{\sqrt{k}} \right)$ and $\|\ualpha\|_4^4 \ll
\frac{n^2}{k}\left(1 + \frac{\log n}{\sqrt{k}}\right)$.
The
stationary phase equations (\ref{stationary_phase})  have a solution and the
solution satisfies
\begin{align}\label{f_0_eval} f_0(0) &= 1 + \frac{2k+1}{4n^2}\|\ualpha\|_2^2
+O\left(
\frac{k^3}{n^4}
\|\ualpha\|_4^4 \right)
\\\label{R_j_eval}
  R_j + \frac{1}{R_j}
& = 2 + \left(f_0(0) \alpha_j \frac{2k+1}{2n} \right)^2 + O\left(\frac{k^4
\alpha_j^4}{n^4} \right)\\
\notag &= 2 + O\left(\frac{\alpha_j^2 k^2}{n^2} \right)\\
\label{log_R}\log R_j &=  \frac{2k+1}{2n}f_0(0) \alpha_j +
O\left(\frac{k^3\alpha_j^3}{n^3}
\right).
\end{align}

\end{lemma}

\begin{proof}
 We have the system of equations
 \begin{equation}\label{f_0_formula}f_0(0) = 1 + \frac{1}{2k+1}\sum_{j=1}^k
\left(R_j + \frac{1}{R_j} -2
\right).\end{equation}
and
\begin{equation}\label{R_formula}
 R_j + \frac{1}{R_j} = \sqrt{4 + \left(f_0(0) \alpha_j \frac{2k+1}{n}
\right)^2}.
\end{equation}
Beginning from an initial guess $f_0(0) = 2$, solve for each $R_j$ in
$R_j \geq 1$ according
to (\ref{R_formula}), sequentially update $f_0(0)$ and then the $R_j$.  This
produces a decreasing sequence of guesses for $f_0(0)$ and for each $R_j$, as
is evident since the first step is decreasing, e.g. since
\[
 R_{j}+\frac{1}{R_{j}}- 2 \leq \frac{f_{0, \old}(0)^2 \alpha_j^2(2k+1)^2 }{4n^2}
\]
and therefore,
\[
 f_{0, \new}(0) \leq 1 + \frac{2k+1}{4n^2} f_{0, \old}(0)^2 \|\ualpha\|_2^2 \leq
1 +
O\left( \frac{1}{n}\right).
\]
As $f_0(0)$ is bounded below, the sequence necessarily converges.

To verify the asymptotics, note that $f_0(0) = O(1)$ leads to
\begin{align*}
 R_j + \frac{1}{R_j} &= 2 + \frac{\left(f_0(0) \alpha_j \frac{2k+1}{n}
\right)^2}{2 + \sqrt{4 + \left(f_0(0) \alpha_j \frac{2k+1}{n}\right)^2}}\\
& = 2 + \left(f_0(0) \alpha_j \frac{2k+1}{2n} \right)^2 + O\left(\frac{k^4
\alpha_j^4}{n^4} \right),
\end{align*}
which satisfies the claimed asymptotic.

Inserted into the formula for $f_0(0)$, this yields
\begin{align*}
 f_0(0) = 1 +\frac{2k+1}{4n^2}\|\ualpha\|_2^2 f_0(0)^2+  O\left(\frac{k^3}{n^4}
\|\ualpha\|_4^4 \right).
\end{align*}
The error introduced by the factor of
$f_0(0)^2$ may be absorbed into the last error term, since $\|\ualpha\|_2^4 \leq
k
\|\ualpha\|_4^4$.

Combining the stationary phase equation with (\ref{R_j_eval}) we find
\begin{align*}
 R_j &= 1 + \frac{2k+1}{2n} f_0(0) \alpha_j +
\frac{1}{2}\left(\frac{2k+1}{2n} f_0(0) \alpha_j
 \right)^2 + O\left(\frac{k^4}{n^4}\alpha_j^4  \right),
\end{align*}
so
\begin{align*}
 \log R_j = \frac{2k+1}{2n}f_0(0) \alpha_j + O\left(\frac{k^3\alpha_j^3}{n^3}
\right).
\end{align*}

\end{proof}

\begin{lemma}\label{F_approximation}
Let $ n, k(n) \in \zed_{>0}$ with $k(n)^2 = o(n)$ as $n \to \infty$. Let
$\ualpha \in
\zed^k$ and assume
$\|\ualpha\|_{2}^2 \leq
n\left(1 + \frac{\log n}{\sqrt{k}} \right)$ and $ \|\ualpha\|_4^4 \ll
\frac{n^2}{k}\left(1 + \frac{\log n}{\sqrt{k}} \right)$.
Let
$R_j$ be determined by the saddle point equations (\ref{stationary_phase}).  For
$\vo \in \sD_{\sm}$ we have
\begin{align*}
F(0) & =
-\frac{2k+1}{4n}\|\ualpha\|_2^2 + O\left(\frac{k^3}{n^3}\|\ualpha\|_4^4\right)\\
D_j F(0) & =0 \\
D_{j_1} D_{j_2} F(0) & =\frac{4\pi^2 \alpha_{j_1}\alpha_{j_2}}{n}, \;
j_1 \neq j_2\\
D_j^2 F(0) & = \frac{-8\pi^2 n}{2k+1} +
\frac{2\pi^2 \|\ualpha\|_2^2}{n} + O\left(\frac{k\alpha_j^2}{n}\right) +
O\left(\frac{k^2}{n^3}\|\ualpha\|_4^4 \right)\\
D_{j_1}D_{j_2}D_{j_3} F(\vo) &= O\left(\frac{n}{k^3} \left(|\theta_{j_1}| +
\frac{k|\alpha_{j_1}|}{n}\right)\left(|\theta_{j_2}| +
\frac{k|\alpha_{j_2}|}{n}\right)\left(|\theta_{j_3}| + \frac{k
|\alpha_{j_3}|}{n}\right)\right),\\& \qquad\qquad\qquad j_1,j_2, j_3 \text{
distinct} \\
D_{j_1}^2 D_{j_2} F(\vo) & =  O\left( \frac{n}{k^2}\left(|\theta_{j_2}| +
\frac{k|\alpha_{j_2}|}{n}\right)\right), \;\; j_1 \neq j_2
\\D_j^3 F(\vo) &= O\left(\frac{n}{k}\left(|\theta_j| + \frac{k
|\alpha_j|}{n}\right)\right).
\end{align*}

\end{lemma}

\begin{proof}
We have
\begin{align*}
 F(0) &= n \log f(0) - \sum_j \alpha_j \log R_j\\
      &= \frac{2k+1}{4n} \|\ualpha\|_2^2  - \frac{2k+1}{2n}f_0(0)
\|\ualpha\|_2^2+
O\left(\frac{k^3}{n^3}
\|\ualpha\|_4^4\right)\\
      &= -\frac{2k+1}{4n}\|\ualpha\|_2^2 + O\left(\frac{k^3}{n^3}
\|\ualpha\|_4^4\right).
\end{align*}
At the saddle point, the first derivatives vanish.  The mixed derivatives are
evaluated by plugging in
\[
 R_j - \frac{1}{R_j} = \frac{2k+1}{n}f_0(0) \alpha_j.
\]
We have
\begin{align*}
 D_j^2 F(0) &= -\frac{4\pi^2 n}{2k+1} \frac{R_j + \frac{1}{R_j}}{f_0(0)} +
\frac{4\pi^2 \alpha_j^2}{n}\\
            &= -\frac{8\pi^2 n}{2k+1} \frac{1}{f_0(0)} + O\left(\frac{k
\alpha_j^2}{n}\right)\\
            &= -\frac{8\pi^2 n}{2k+1} + \frac{2\pi^2 \|\ualpha\|_2^2}{n}
+O\left(\frac{k\alpha_j^2}{n}\right) + O\left(\frac{k^2}{n^3}\|\ualpha\|_4^4
\right).
\end{align*}

The triple derivatives are estimated by Taylor expanding $e(\theta)$ to degree
1 in the numerator, using $R_j - \frac{1}{R_j} \ll \frac{k\alpha_j}{n}$ and
$R_j + \frac{1}{R_j}, f_0(\vo) \asymp 1$.

\end{proof}

\begin{lemma}\label{taylor_expansion_bound}
Let $n, k(n) \in \zed_{>0}$ with $k^2 = o(n)$ as $n \to \infty$.  Let $\ualpha
\in \zed^k$ and assume
$\|\ualpha\|_{2}^2 \leq
n\left(1 + \frac{\log n}{\sqrt{k}} \right)$ and $ \|\ualpha\|_4^4 \ll
\frac{n^2}{k}\left(1 + \frac{\log n}{\sqrt{k}} \right)$. Let $\vo \in
\sD_{\sm}$. We have
\begin{align*}
 F(\vo) - F(0) =& \frac{-4\pi^2 n}{2k+1}\|\vo\|_2^2 + O\left(\left(1 +
\frac{\log n}{\sqrt{k}}\right) \|\vo\|_2^2 + \sqrt{k}\left(1 +
\frac{\log n}{\sqrt{k}}\right)\|\vo\|_4^2\right)\\
&+O\left(\frac{n}{k^3}\|\vo\|_2^6 + \frac{n}{k^2} \|\vo\|_2^4 +
\frac{\sqrt{n}}{k} \left(1 + \frac{\log n}{\sqrt{k}} \right)\|\vo\|_2^3\right)\\
&+O\left( \frac{n}{k} \|\vo\|_4^4 +
\frac{\sqrt{n}}{k^{\frac{1}{4}}}\left(1 + \frac{\log
n}{\sqrt{k}}
\right)
\|\vo\|_4^3\right).
\end{align*}
In particular, for any fixed constants $c_2, c_4>0$, for  $\|\vo\|_2 \leq c_2
\frac{k}{n^{\frac{1}{2}}}$ and $\|\vo\|_4
\leq c_4 \frac{k^{\frac{3}{4}}}{n^{\frac{1}{2}}}$ we have
\begin{equation}\label{integrand_bulk}
 F(\vo) - F(0) + \frac{4\pi^2 n}{2k+1}\|\vo\|_2^2 = o(1),
\end{equation}
while for $\|\vo\|_2 \leq c_2 \frac{k}{n^{\frac{1}{2}}}$ and $\|\vo\|_4 = o(1)$,
\begin{equation}\label{6_norm_large}
 F(\vo) - F(0) + \frac{4\pi^2 n}{2k+1}\|\vo\|_2^2 \ll o(1)
+\frac{n}{k}\|\vo\|_4^4 +  \left(1 +
\frac{\log n}{\sqrt{k}} \right)\left(\sqrt{k}\|\vo\|_4^2 +
\frac{\sqrt{n}}{k^{\frac{1}{4}}}\|\vo\|_4^3 \right),
\end{equation}
and in general for $\|\vo\|_\infty < \delta < \frac{1}{12}$,
\begin{equation}\label{2_norm_large}
 F(\vo) - F(0) + \frac{4\pi^2 n}{2k+1}\|\vo\|_2^2 \ll\delta^2
\frac{n}{k}\|\vo\|_2^2 + \left(1 + \frac{\log
n}{\sqrt{k}}\right)\left(\sqrt{\frac{n}{k}}\|\vo\|_2 +
\frac{\sqrt{n}}{k^{\frac{1}{4}}}\|\vo\|_2^{\frac{3}{2}}\right).
\end{equation}

\end{lemma}

\begin{proof}
By Taylor's theorem, for $\vo \in \sD_{\sm}$, for some $0 \leq t_{\vo} \leq 1$,
\[
 F(\vo) - F(0) = \frac{1}{2} \sD^2(0)(\vo, \vo) + \frac{1}{6}
\sD^3(t_{\vo}\vo)(\vo, \vo, \vo)
\]
where $\sD^2$ and $\sD^3$ represent the second and third derivatives of $F$.
Write
\[
 \sD^2(0) = \frac{-8\pi^2 n}{2k+1}I_k + \tilde{\sD}^2(0)
\]
We have
\begin{align*}
 \tilde{\sD}^2(0)(\vo, \vo)&\ll \frac{\|\ualpha\|_2^2 \|\vo\|_2^2}{n} +
\frac{k\|\ualpha\|_4^2 \|\vo\|_4^2}{n} + \frac{k^2}{n^3} \|\ualpha\|_4^4
\|\vo\|_2^2\\
&\ll \left(1 + \frac{\log n}{\sqrt{k}}\right) \|\vo\|_2^2 + \sqrt{k}\left(1 +
\frac{\log n}{\sqrt{k}}\right) \|\vo\|_4^2
\end{align*}
Also,
\begin{align*}
 \left|\sD^3(t_{\vo}\vo)(\vo, \vo, \vo)\right| \ll& \frac{n}{k^3}\left(
\|\vo\|_2^6 + \frac{k^3}{n^3}\|\vo\|_2^3 \|\ualpha\|_2^3\right)\\&+
\frac{n}{k^2}\left(\|\vo\|_2^4 + \frac{k}{n}\|\vo\|_2^3 \|\ualpha\|_2 \right)\\
&+ \frac{n}{k}\left(\|\vo\|_4^4 + \frac{k}{n}\|\vo\|_4^3\|\ualpha\|_4 \right)
\\ \ll& \frac{n}{k^3}\|\vo\|_2^6 + \frac{n}{k^2} \|\vo\|_2^4 +
\frac{\sqrt{n}}{k} \left(1 + \frac{\log n}{\sqrt{k}} \right)\|\vo\|_2^3\\& +
\frac{n}{k} \|\vo\|_4^4 + \frac{\sqrt{n}}{k^{\frac{1}{4}}}\left(1 + \frac{\log
n}{\sqrt{k}}
\right)
\|\vo\|_4^3 .
\end{align*}

For (\ref{2_norm_large})
use $\|\vo\|_4 \leq \delta^{\frac{1}{2}} \|\vo\|_2^{\frac{1}{2}}$,
 and $\|\vo\|_2
\leq \delta \sqrt{k}$.
\end{proof}

\begin{lemma}\label{general_bound}
 Keep the same assumptions on $k, n$ and $\ualpha$ as in Lemma
\ref{taylor_expansion_bound}.  We have
 \[
 \left|\IM f_0(\vo)\right| \ll \frac{\|\ualpha\|_2\|\vo\|_2}{n}.
 \]
Moreover, there is a constant $c > 0$ such that, if $\RE(f_0(\vo))>0$ then
\[
 \left|\frac{f_0(\vo)}{f_0(0)}\right| \leq 1 - \frac{c}{k}
\|\vo\|_{(\bR/\zed)^k}^2,
\]
and if $\RE(f_0(\vo)) < 0$ then
\[
 \left|\frac{f_0(\vo)}{f_0(0)}\right| \leq 1 - \frac{c}{k} -
\frac{c}{k}\left\|\vo - \left(\frac{1}{2}\right)_k\right\|_{(\bR/\zed)^k}^2,
\]
where $\left(\frac{1}{2}\right)_k$ denotes the vector of $\bR^k$, all of whose
coordinates are $\frac{1}{2}$.

\end{lemma}
\begin{proof}
 We have
 \begin{align*}
  \left|\IM f_0(\vo)\right| &= \left|\frac{1}{2k+1} \sum_{j=1}^k \left(R_j -
\frac{1}{R_j}\right) \sin(2\pi \theta_j)\right|\\
&= \frac{f_0(0)}{n} \left|\sum_{j=1}^k \alpha_j \sin(2\pi \theta_j)\right|
\\& \ll \frac{\|\ualpha\|_2 \|\vo\|_2 }{n}.
 \end{align*}
If $\RE(f_0(\vo)) > 0$ then
\begin{align*}
 \left|\frac{\RE f_0(\vo)}{f_0(0)}\right| &= 1-
\frac{1}{f_0(0)(2k+1)}\sum_{j=1}^k \left(R_j +\frac{1}{R_j}\right)(1 -
\cos(2\pi \theta_j))\\
& \leq 1 - \frac{c}{k} \|\vo\|_2^2.
\end{align*}
If, instead, $\RE(f_0(\vo)) < 0$ then
\begin{align*}
 \left|\frac{\RE f_0(\vo)}{f_0(0)}\right| &= 1 - \frac{1}{(2k+1)f_0(0)} -
\frac{1}{(2k+1)f_0(0)} \sum_{j=1}^k \left(R_j + \frac{1}{R_j}\right)(1 +
\cos(2\pi \theta_j))\\
& \geq 1 - \frac{c}{k} - \frac{c}{k} \left\|\vo -
\left(\frac{1}{2}\right)_k\right\|_{(\bR/\zed)^k}^2.
\end{align*}

The bound for $\left|\frac{f_0(\vo)}{f_0(0)} \right|$ in the case
$\RE(f_0(\vo))>0$ follows from, for some $c'>0$,
\[
 \left|\frac{f_0(\vo)}{f_0(0)} \right|^2 \leq 1 -  \frac{c'}{k}\|\vo\|_2^2 +
O\left(\left( 1 + \frac{\log n}{\sqrt{k}}\right)\frac{\|\vo\|_2^2}{n}\right),
\]
and the claim in the case $\RE(f(\vo)) <0$ is similar.
\end{proof}

We give our final estimate.

\begin{proof}[Proof of Lemma \ref{gauss_approximation}]  Let $0 < \delta <
\frac{1}{12}$ be a constant to be chosen.
\begin{align*}\nu_k^{*n}(\ualpha) &= \int_{(\bR/\zed)^k}
\frac{f_0(\vo)^n}{R_1^{\alpha_1}...R_k^{\alpha_k}}d\vo
\\ &=
\frac{f_0(0)^n}{R_1^{\alpha_1}...R_k^{\alpha_k}}\left[\int_{\|\vo\|_\infty \leq
\delta}e^{F(\vo)- F(0)}d\vo + \int_{\|\vo\|_\infty > \delta}
\left(\frac{|f_0(\vo)|} {f_0(0)}\right)^n d\vo
\right]. \end{align*}

By Lemma \ref{F_approximation},
\begin{equation}\label{leading_term}\frac{f_0(0)^n}{R_1^{\alpha_1}...R_k^{
\alpha_k } } = e^{F(0)} =
e^{-\frac{2k+1}{4n}\|\ualpha\|_2^2}\left[1 +  O\left(\frac{k^3}{n^3}
\|\ualpha\|_4^4\right)\right] \sim
e^{-\frac{2k+1}{4n}\|\ualpha\|_2^2}\end{equation}
since $\|\ualpha\|_4^4 \ll \frac{n^2}{k}\left(1 + \frac{\log
n}{\sqrt{k}}\right)$.

To treat the integral over $\|\vo\|_\infty < \delta$, write
\[
 \int_{\|\vo\|_\infty \leq \delta}e^{F(\vo)- F(0)}d\vo =
\left(\frac{2k+1}{4\pi
n}\right)^{\frac{k}{2}} \int_{\|\vo\|_\infty \leq \delta}
\eta_k\left( \sqrt{\frac{2k+1}{8\pi^2 n}}, \vo\right)\exp\left(G(\vo)
\right)d\vo.
\]
Partition $B_\infty(0,\delta)$ by choosing for some parameters $c_2, c_4$,
\begin{align*}
 B_\infty(0,\delta) &= B \sqcup E_1 \sqcup E_2\\
  B & = B_\infty(0,\delta)\cap B_2\left(0,
c_2\frac{k}{\sqrt{n}}\right)
\cap B_{4}\left(
0, c_4 \frac{k^{\frac{2}{3}}}{n^{\frac{1}{2}}} \right)\\
  E_1 &= B_\infty(0,\delta) \cap B_2\left(0,
c_2\frac{k}{\sqrt{n}}\right)
\setminus B_{4}\left(
0, c_4 \frac{k^{\frac{2}{3}}}{n^{\frac{1}{2}}} \right) \\
  E_2 &= B_\infty(0,\delta) \setminus B_2\left(0,
c_2\frac{k}{\sqrt{n}}\right).
\end{align*}
The parameters $c_2, c_4$ are considered fixed, but may be arbitrarily large.

On $B$, (\ref{integrand_bulk}) gives $G(\vo) = o(1)$ and we find
\begin{align*}
 \int_{\vo \in B} \eta_k\left( \sqrt{\frac{2k+1}{8\pi^2 n}},
\vo\right)\exp\left(G(\vo)
\right)d\vo &= (1+o(1))\int_{\vo \in B} \eta_k\left( \sqrt{\frac{2k+1}{8\pi^2
n}},
\vo\right) d\vo \\&= 1 + \varepsilon(c_2, c_4),
\end{align*}
where $\varepsilon(c_2, c_4) \to 0$ as $\min(c_2, c_4)\to \infty$,
as follows by Lemma \ref{gaussian_length_concentration} ($c_2$ and $c_4$ only
need be taken growing if $k$ does not grow).

In treating $E_1$ and $E_2$, let $C_2, C_4$ be the constants of Lemma
\ref{gaussian_length_concentration}.
To treat $E_1$, note that with respect to the Gaussian measure $\gamma =
\eta_k\left(\sqrt{\frac{2k+1}{8\pi^2 n}}, \vo \right)$, the event $\vo \in
B_{\infty}(0,\delta)\cap
B_2\left(0,
c_2\frac{k}{\sqrt{n}}\right)$ has probability $\asymp 1$, and thus, even
after
conditioning on this event, the probability of $\|\vo\|_4 >
C_4 k^{\frac{1}{4}}\sqrt{\frac{2k+1}{8\pi^2 n}} + t $ is, for some $C>0$,
$O\left(\exp\left(-\frac{n}{k}\frac{t^2}{C} \right) \right)$.
The bound  $\|\vo\|_4 \leq
\delta^{\frac{1}{2}}\|\vo\|_2^{\frac{1}{2}}$  implies that on $E_1$, $\|\vo\|_4
= o(1)$.  Set $t = \|\vo\|_4 -C_4 k^{\frac{1}{4}}\sqrt{\frac{2k+1}{8\pi^2 n}}$
and assume $c_4$ is larger than a sufficiently large multiple of $C_4$, so that
$t \gg \frac{k^{\frac{3}{4}}}{\sqrt{n}}$. By
(\ref{6_norm_large}) we find that for $\vo \in E_1$,
\begin{align*}
  G(\vo) &\leq g(t)\\
 g( t) &\ll o(1) + \frac{n}{k}t^4 +\left(1 +\frac{\log
n}{\sqrt{k}}\right)\left(\sqrt{\frac{n}{k}}t^2 +
\frac{\sqrt{n}}{k^{\frac{1}{4}}} t^3 \right).
\end{align*}
Then
\begin{align*}
&\int_{E_1} \eta_k\left(\sqrt{\frac{2k+1}{8\pi^2 n}}, \vo \right)
\exp(G(\vo))d\vo
\\&\leq -\int_{\frac{k^{\frac{3}{4}}}{\sqrt{n}} \ll t =
o(1)}\exp\left(g(t)\right)d\meas\left(\|\vo\|_4 \geq
C_4 k^{\frac{1}{4}}\sqrt{\frac{2k+1}{8\pi^2 n}} + t \right).
\end{align*}
Integrating by parts, we find that this integral is $o(1)$ as $c_4 \to \infty$.

To treat $E_2$, set $s = \|\vo\|_2$ and appeal to (\ref{2_norm_large}) to find
\begin{align*}
 G(\vo) &\leq h(s)\\
h(s) &\ll o(1) +\delta^2 \frac{n}{k}s^2+  \left(1 + \frac{\log
n}{\sqrt{k}}\right) \left(\sqrt{\frac{n}{k}}s +
\frac{\sqrt{n}}{k^{\frac{1}{4}}} s^{\frac{3}{2}}
\right).
\end{align*}
Also, for some $C>0$,
\[
 \meas\left(\|\vo\|_2 > C_2 k^{\frac{1}{2}}\sqrt{\frac{2k+1}{8\pi^2 n}} +
s\right) \ll \exp\left(-\frac{n}{k}\frac{s^2}{C}\right).
\]
We conclude
\begin{align*}
 &\int_{\vo \in E_2}
\eta_k\left(\sqrt{\frac{2k+1}{8\pi^2 n}}, \vo \right) \exp(G(\vo))d\vo \\& \leq
-
\int_{\frac{k}{\sqrt{n}} \ll s \leq \delta\sqrt{k}} \exp(h(s))
d\meas\left(\|\vo\|_2 > \sqrt{\frac{2k+1}{8\pi^2 n}} +
s\right).
\end{align*}
If $\delta$ is sufficiently small, this integral is in fact $o(1)$ as $c_2 \to
\infty$, as may be
checked again by integration by parts.

It remains to bound the integral over $\|\vo\|_\infty > \delta$.  Consider
first the case $\RE(f(\vo))>0$.  Let $S \subset [k]$ be the collection of
$\theta_j$ with $|\theta_j|> \delta$.  Write $\vo_S$ for the variables
in $S$ and $\vo_{S^c}$ for the variables in $S^c$.  Appealing to Lemma
\ref{general_bound}, we see that if $|S|\gg \log n$ then the integral is
negligible.  Using $1-x < e^{-x}$ in the remaining range we obtain a bound, for
some fixed $c > 0$,
\begin{align*}
& \ll \sum_{1 \leq j \ll \log n} \sum_{S \subset [k], |S| = j} \exp\left(-
\frac{c j n}{k} \right) \int_{\|\vo_{S^c}\|_\infty < \delta}
\exp\left(-\frac{cn}{k}\|\vo_{S^c}\|_2^2 \right)\\& \ll \left(\frac{2k+1}{4\pi
nc} \right)^{\frac{k}{2}} \sum_{1 \leq j \ll \log n}
\binom{k}{j}\left(\frac{4\pi nc}{2k+1} \right)^{\frac{j}{2}}\exp\left(-\frac{c
j n}{k}\right)\\&
= o\left(\left(\frac{2k+1}{4\pi
n} \right)^{\frac{k}{2}} \right).
\end{align*}

The terms for which $\RE(f(\vo))< 0$ are handled similarly.
\end{proof}

\bibliographystyle{plain}
\bibliography{cycle}

\end{document}